\documentclass{amsart}
\usepackage[utf8]{inputenc}
\usepackage[english,french]{babel}
\usepackage[T1]{fontenc}
\usepackage{listings}
\usepackage{color}      
\usepackage{indentfirst}        
\usepackage{latexsym,bm}        
\usepackage{array}
\usepackage{subfiles}
\usepackage{amsxtra}
\usepackage{amstext}
\usepackage{xfrac}

\usepackage{dsfont} 
\usepackage{amsmath,amsthm,amssymb,amscd,amsfonts}
\usepackage{calligra}
\usepackage{mathrsfs}

\usepackage[colorlinks = true,
            linkcolor = blue,
            urlcolor  = blue,
            citecolor = blue,
            anchorcolor = blue]{hyperref}
\usepackage{url}
\usepackage{cleveref}

\usepackage[mathscr]{eucal}
\usepackage{mathrsfs}
\usepackage{pdfsync}

\newcommand{\pp}{\mathfrak{p}}
\newcommand{\kk}{\mathfrak{k}}
\newcommand{\g}{\mathfrak{g}}
\newcommand{\z}{\mathfrak{z}}
\newcommand{\kt}{\mathfrak{t}}

\newcommand{\ks}{\mathfrak{s}}
\newcommand{\kb}{\mathfrak{b}}
\newcommand{\ku}{\mathfrak{u}}
\newcommand{\kn}{\mathfrak{n}}

\newcommand{\ke}{\mathfrak{e}}

\newcommand{\uzs}{\underline{\mathfrak{z}_\sigma}}

\newcommand{\cE}{\mathcal{E}}

\newcommand{\Z}{\mathbb{Z}}
\newcommand{\bbC}{\mathbb{C}}
\newcommand{\R}{\mathbb{R}}
\newcommand{\bN}{\mathbb{N}}

\newcommand{\Ad}{\mathrm{Ad}}
\newcommand{\ad}{\mathrm{ad}}

\begin{document}
\selectlanguage{english}

\title{Hypoelliptic Laplacian and twisted trace formula}
\author{Bingxiao LIU}
\address{Universität zu Köln, Weyertal 86-90, D-50931 Köln }

\begin{abstract}
	We give an explicit geometric formula for the twisted orbital 
	integrals using the method of the hypoelliptic Laplacian 
	developed by Bismut. Combining with the twisted trace formula, we 
	can evaluate the equivariant 
	trace of the heat operators of the Laplacians on a compact locally 
	symmetric space. In particular, we revisit the equivariant local 
	index theorems and twisted $L_{2}$-torsions for locally symmetric 
	spaces.
\end{abstract}


\theoremstyle{plain}
\newtheorem{theorem}{Theorem}[subsection]
\newtheorem{lemma}[theorem]{Lemma}
\newtheorem{proposition}[theorem]{Proposition}
\newtheorem{corollary}[theorem]{Corollary}

\theoremstyle{remark}
\newtheorem{definition}[theorem]{Definition}
\newtheorem{remark}[theorem]{Remark}
\newtheorem{example}[theorem]{Example}


\numberwithin{equation}{subsection}
 \renewcommand\thesection{\arabic{section}}
 \renewcommand\thesubsection{\arabic{section}.\arabic{subsection}}

\maketitle
\thispagestyle{empty}
\tableofcontents

\section{Introduction}
The purpose of this paper is to give an explicit geometric formula 
for the semisimple twisted orbital integrals associated with the 
Casimir operator on symmetric spaces, which extends an important result of Bismut for 
semisimple orbital integrals \cite[Chapter 
6]{bismut2011hypoelliptic}. The method that we use is the theory of hypoelliptic 
Laplacian developed by Bismut \cite{bismut2011hypoelliptic}. Here, we 
start with establishing a 
geometric formulation for the twisted orbital integral. Then we 
explain how to adapt Bismut's method to get 
our explicit formula. In the context of cyclic base change theory, we also 
exploit our formula by typical examples.

To explore the power of our formula, we use it to revisit the local equivariant 
index theorems for compact locally symmetric space, and especially, we 
exhibit the computations on the twisted orbital integrals using 
representation theory of compact Lie groups. In 
the last subsection, we also discuss briefly the equivariant real 
analytic torsion. For further study on this topic using our explicit formula, 
we refer to the author's paper \cite{LIU2021109117}.

Let us now give more details on the content of this paper.
\subsection{Real reductive group and symmetric space}
\label{s0.2}
Let $G$ be a connected real reductive Lie group (\cite[\S 
7.2]{knapp2002liegroupe})
with Lie algebra $\g$, and let $\theta\in \mathrm{Aut}(G)$ be a 
Cartan involution. Let $K$ be the fixed point set of $\theta$ in $G$. 
Then $K$ is a maximal compact subgroup of $G$ with Lie algebra $\kk$. 
The Cartan decomposition of $\g$ associated with $\theta$ is given by
\begin{equation}
	\mathfrak{g}=\mathfrak{p}\oplus \mathfrak{k}.
	\label{eq:0.2.1ugc}
\end{equation}
Put $m=\dim \pp$, $n=\dim \kk$.

Let $B$ be a $G$ and $\theta$-invariant nondegenerate bilinear 
symmetric form on $\g$, which is positive on $\pp$ and negative on $\kk$. Let $U\g$ be the enveloping algebra of $\g$, and let $C^\g\in U\g$ be the Casimir operator associated with $B$.

Let $X=G/K$ be the associated symmetric space. Then the projection $p: 
G\rightarrow X$ is a $K$-principal bundle. The bilinear form $B$ induces a 
Riemannian metric $g^{TX}$ on $X$ with nonpositive sectional curvature.
Let $d(\cdot,\cdot)$ denote the Riemannian 
distance on $X$.

If $(E,\rho^E)$ is a unitary representation of $K$, then $F=G\times_K 
E$ is a Hermitian vector bundle on $X$. Moreover, $C^\g$ descends to 
an elliptic operator $C^{\g,X}$ acting on $C^{\infty}(X,F)$. Our main 
object is to study the operator 
$\mathcal{L}^X$ acting on $C^\infty(X,F)$, which is defined as the 
sum of $\frac{1}{2}C^{\g,X}$ with an explicit real constant 
(Definition \ref{def:3.2.1sss}). For $t>0$, let 
$\exp(-t\mathcal{L}^X)$ be the associated heat operator.

\subsection{Twisted orbital integrals}\label{s0.8}
We introduce the geometric characterization for semisimple elements. Let 
$\mathrm{Isom}(X)$ be the Lie group of isometries of $X$. If 
$\phi\in \mathrm{Isom}(X)$, set $d_\phi(x)=d(x,\phi(x))$, $x\in X$. As in \cite{eberlein1996geometry}, $\phi$ is called semisimple if $d_\phi$ reaches its infimum value $m_\phi$ in $X$, and $\phi$ is called elliptic if $\phi$ has fixed points in $X$. If $\phi$ is semisimple, let $X(\phi)\subset X$ be the minimizing set of $d_\phi$, which is a convex submanifold of $X$.

In \cite[Chapters 3]{bismut2011hypoelliptic}, given a semisimple 
element $\gamma\in G$ (viewed as an isometry of $X$), 
$X(\gamma)$ is a symmetric space associated with $Z^{0}(\gamma)$, the 
identity component of the centralizer of $\gamma$. Then Bismut 
gave a geometric interpretation for the 
associated orbital integrals $\mathrm{Tr}^{[\gamma]}[\exp({-t\mathcal{L}^X })]$, so that they can be written as 
integrals along the fibres of the normal bundle $N_{X(\gamma)/X}$. This 
geometric formulation plays a central role in 
Bismut's approach to his explicit geometric formula of 
$\mathrm{Tr}^{[\gamma]}[\exp({-t\mathcal{L}^X })]$. Using Bismut's formula, Shen \cite{MR3473562,Shen_2016} 
 gave a full proof of the Fried conjecture for compact locally 
 symmetric spaces, completing the work of Moscovici and Stanton \cite{MS1991}.

In this paper, we extend Bismut's result to the case of 
twisted orbital integrals.
Let $\Sigma$ be the compact Lie group of the automorphisms of $\left(G,B,\theta\right)$.
If $\sigma\in \Sigma$, let 
$\Sigma^\sigma$ be the closure of the subgroup of $\Sigma$ generated by 
$\sigma$. Put $	G^{\sigma}=G\rtimes\Sigma^{\sigma},\;K^\sigma=K\rtimes 
\Sigma^\sigma$. We do not assume $\sigma$ to have finite order.

If $\sigma\in \Sigma$, we define the $\sigma$-twisted conjugation 
$C^{\sigma}$ so that if $h,\gamma\in G$, 
\begin{equation}
	C^{\sigma}(h)\gamma=h\gamma\sigma(h^{-1}).
	\label{eq:csigmaeng}
\end{equation}
Let 
$Z_{\sigma}(\gamma)\subset G$ be the 
$\sigma$-twisted centralizer of 
$\gamma\in G$. Then $\sigma$-twisted conjugacy class of $\gamma$ in $G$ can be identified with 
$Z_{\sigma}(\gamma)\backslash G$. The twisted orbital integral, 
defined as a certain integral on $Z_{\sigma}(\gamma)\backslash G$, has been vastly studied in cyclic base change theory (cf.
\cite{Langlands1988base}, \cite{Clozel1989}, 
\cite{CM_1992__81_3_261_0}, \cite{BeLip2017}, 
etc). 

Due to the possible nontrivial large center of $G$, the Lie group 
$G\rtimes \Sigma$, even 
$G^{\sigma}$, may fail to be reductive. A typical example is 
$\R^{m}\rtimes \mathrm{O}(m)$. In Subsection \ref{s1-4}, we explain 
the key point that the above groups do not displace very far from a 
reductive one. In particular, if $\gamma\in G$ is
such that $\gamma\sigma$ is semisimple as an isometry of $X$, we 
establish, via a geometric argument, a decomposition theorem 
 for $Z_{\sigma}(\gamma)$. Then we show that  
$X(\gamma\sigma)$ is a symmetric space associated 
with $Z^{0}_{\sigma}(\gamma)$, the identity 
component of $Z_{\sigma}(\gamma)$. This way, in Subsection 
\ref{section2-4}, we give a 
geometric interpretation for the twisted orbital integrals, as an 
extension of \cite[Definition 4.2.2]{bismut2011hypoelliptic}.

We now assume $(E,\rho^{E})$ to be a unitary representation of $K^\sigma$. Then the action of 
$G^\sigma$ on $X$ lifts to $F$. The operator $\mathcal{L}^{X}$ 
commutes with $G^{\sigma}$. For $t>0$, let $\mathrm{Tr}^{[\gamma\sigma]}[\exp(-t 
\mathcal{L}^X)]$ denote the $\sigma$-twisted orbital integral of the kernel of 
$\exp(-t \mathcal{L}^X)$ associated with $\gamma$.

\subsection{Statement of the main results}\label{s0.4}

If $\gamma\sigma$ is semisimple, after conjugation, we may and we 
will assume that $\gamma=e^a k^{-1}$ with $a\in\mathfrak{p}, k\in 
K$ and $\mathrm{Ad}(k)a=\sigma a$. Then $\theta$ acts on 
$Z_{\sigma}(\gamma)$. Let $\z_{\sigma}(\gamma)\subset \g$ be the Lie 
algebra of $Z_{\sigma}(\gamma)$, and let $\kk_{\sigma}(\gamma)$ be 
the Lie algebra of $K_{\sigma}(\gamma)=Z_{\sigma}(\gamma)\cap K$. As in \eqref{eq:0.2.1ugc}, we have 
the Cartan decomposition
\begin{equation}
	\z_{\sigma}(\gamma)=\pp_{\sigma}(\gamma)\oplus 
	\kk_{\sigma}(\gamma).
	\label{eq:bonn3333}
\end{equation}
Put $p=\dim \pp_{\sigma}(\gamma)$, $q=\dim \kk_{\sigma}(\gamma)$.

The analytic function $J_{\gamma\sigma}(Y^\kk_0)$ in $Y^\kk_0\in 
\kk_{\sigma}(\gamma)$ will be defined in Definition \ref{def:3.1.1sss} by an explicit formula. The main result of this paper is as follows.
\begin{theorem}\label{thm:0000}
	For	$t>0$, the following identity holds:
\begin{equation}
\begin{split}
		&\mathrm{Tr}^{[\gamma\sigma]}\big[\exp(-t 
		\mathcal{L}^X)\big]=\frac{\exp(-|a|^2/2t)}{(2\pi t)^{p/2}}\\
		&\cdot\int_{\mathfrak{k}_{\sigma}(\gamma)} 
		J_{\gamma\sigma}(Y_0^{\mathfrak{k}})\mathrm{Tr}^E\big[\rho^E(k^{-1}\sigma)\exp(-i \rho^E(Y_0^\mathfrak{k}))\big]
		e^{-|Y^\mathfrak{k}_0|^2/2t} \frac{dY^\mathfrak{k}_0}{(2\pi 
		t)^{q/2}}. 
\end{split}	
	\label{eq:55}
\end{equation}
\end{theorem}

If $\sigma=\mathrm{Id}_{G}$, it is just Bismut's formula given in \cite[Theorem 
6.1.1]{bismut2011hypoelliptic}. In \cite[Sections 8.1 and 10.6]{bismut2011hypoelliptic}, Bismut 
	explained that a formula like \eqref{eq:55} holds for the 
	cases such as $G=K$ non-connected, and $G=\R^{m}$, 
	$\sigma\in \mathrm{O}(m)$. Our theorem here confirms his observation in a more general setting. We will restate the above theorem in
Subsection \ref{s4.2}, and the proof will be given in Section 
\ref{section:proof}, which is partly derived from \cite[Chapter 
9]{bismut2011hypoelliptic}. In Subsection \ref{section:basechange}, 
we exploit the formula \eqref{eq:55} in the context of cyclic base 
change theory over $\R$, so that we only need elementary computation 
from linear algebra to establish some nontrivial identities.

Let $\pp^{\perp}_{\sigma}(\gamma)\subset \pp$ be the orthogonal space 
of $\pp_{\sigma}(\gamma)$ in $\pp$ with respect to $B$. Let $P^\perp_{\sigma}(\gamma)\subset X$ be the image of $\pp^{\perp}_{\sigma}(\gamma)$ by the map $f\rightarrow pe^{f}$. Put
\begin{equation}
	\Delta^{\gamma\sigma}_{X}=\{(x,\gamma\sigma(x))\,:\, x\in 
	P^\perp_{\sigma}(\gamma)\}\subset X\times X.
	\label{eq:sousvariete0}
\end{equation}

Let 
$(a,\kk_{\sigma}(\gamma))$ denote the affine subspace of 
$\z_{\sigma}(\gamma)=\pp_{\sigma}(\gamma)\oplus 
\kk_{\sigma}(\gamma)$. Set
\begin{equation}
	H^{\gamma}_{\sigma}= \{0\}\times 
	\big(a,\kk_{\sigma}(\gamma)\big)\subset 
	\z_{\sigma}(\gamma)\times\z_{\sigma}(\gamma).
	\label{eq:Haffine0}
\end{equation}
Let $\Delta^{\z_{\sigma}(\gamma)}$ denote the standard Laplacian on 
$\z_{\sigma}(\gamma)$.

In Subsection \ref{section:5.3}, using 
Theorem \ref{thm:0000}, we get an extension of 
\cite[Theorem 6.3.2]{bismut2011hypoelliptic} for the twisted orbital 
integrals for wave operators.
\begin{theorem}\label{thm:waveoperator00}
	We have the identity of even distributions on $\mathbb{R}$ 
	(defined in Subsection \ref{section:5.3})
	supported on $\{s\in\mathbb{R}\,:\, |s|\geq \sqrt{2}|a|\}$ with 
	singular support included in $\pm \sqrt{2}|a|$,
	\begin{equation}
		\begin{split}
					&\int_{\Delta^{\gamma\sigma}_{X}}\mathrm{Tr}^{F}\big[\gamma\sigma 
		\cos\big(s\sqrt{\mathcal{L}^{X}}\big)\big]\\
		&=\int_{H^{\gamma}_{\sigma}} 
		\mathrm{Tr}^{E}\Big[\cos\Big(s\sqrt{-\Delta^{\z_{\sigma}(\gamma)}/2}\Big)J_{\gamma\sigma}(Y^{\kk}_{0})\rho^{E}(k^{-1}\sigma)\exp\left(-i\rho^{E}(Y^{\kk}_{0})\right)\Big].
		\end{split}
		\label{eq:waveop00}
	\end{equation}
\end{theorem}

\subsection{Hypoelliptic Laplacian on symmetric spaces}
\label{s0.2sud}
Let us briefly recall the theory of 
hypoelliptic Laplacian developed by Bismut in 
\cite{bismut2011hypoelliptic}. We also refer to \cite{Ma2017bourbaki} 
for an introduction to this theory.

Put $N=G\times_{K}\kk$. Then $TX\oplus N$ is canonically trivial on $X$.
Let $\widehat{\pi}:\widehat{\mathcal{X}}\rightarrow X$ be the total 
space of $TX\oplus N$, so that $\widehat{\mathcal{X}} = X\times\g$.  
The hypoelliptic Laplacian is defined as a family 
of hypoelliptic differential operators $\{\mathcal{L}^X_b\}_{b>0}$ acting on 
$C^{\infty}(\widehat{\mathcal{X}},\widehat{\pi}^{*}(\Lambda^{\bullet}(T^{*}X\oplus N^{*})\otimes F))$.

Let $\Delta^{TX\oplus N}$ be the standard Laplace along the fibre 
$TX\oplus N$. Then $\mathcal{L}^X_b$ is given as follows \cite[Section 2.13]{bismut2011hypoelliptic},
\begin{equation}\label{hypoopX00}
\begin{split}
\mathcal{L}^X_b = \frac{1}{2}\big|[Y^N,Y^{TX}]\big|^2+\frac{1}{2b^2} 
\big(-\Delta^{TX\oplus N}+|Y|^2-m-n\big)
+\frac{N^{\Lambda^\bullet(T^*X\oplus N^*)}}{b^2}&\\
 +\frac{1}{b}\Big( \nabla^{C^\infty(TX\oplus N, \widehat{\pi}^*(\Lambda^\bullet(T^*X\oplus N^*)\otimes F))}_{Y^{TX}}+ 
\widehat{c}\big(\ad(Y^{TX})\big)&\\
- c\big(\mathrm{ad}(Y^{TX})+i\theta \mathrm{ad}(Y^N)\big)
-i\rho^E(Y^N)\Big).&
\end{split}
\end{equation} 
The structure of $\mathcal{L}^X_b$ is close to the structure of 
the hypoelliptic Laplacian studied by Bismut 
\cite{bismut2005cotangent, MR2473254}, and by Bismut-Lebeau \cite{BL2008RaySinger}.

In 
\cite{bismut2011hypoelliptic}, the proper functional analytic 
machinery was developed in order to obtain the analytic properties of 
the resolvent and of the heat kernel of $\mathcal{L}^X_b$. Let 
$\exp({-t\mathcal{L}^X_b})$ be the heat operator associated with 
$\mathcal{L}^X_b$. In \cite[Chapters 11, 14]{bismut2011hypoelliptic}, Bismut proved that there is a smooth 
kernel $q^X_{b,t}$ associated with $\exp({-t\mathcal{L}^X_b})$, and that as $b\rightarrow 0$, $q^X_{b,t}$ converges in the 
proper sense to the kernel of $\exp(-t\mathcal{L}^X)$.

In \eqref{hypoopX00}, the term $\nabla^{C^\infty(TX\oplus 
N, \widehat{\pi}^*(\Lambda^\bullet(T^*X\oplus N^*)\otimes 
F))}_{Y^{TX}}$ represents the left action of the generator of the
geodesic flow. If we forget the first quartic term in the right-hand 
side of \eqref{hypoopX00}, then after rescaling, as $b\rightarrow +\infty$, 
$\mathcal{L}^X_b$ converges in a na\"{i}ve sense to the generator of the 
geodesic flow. More precisely, the diffusion associated with the 
scalar part of $\mathcal{L}^X_b$ tends to propagate along the 
geodesic flow. In \cite[Chapters 12, 15]{bismut2011hypoelliptic}, Bismut 
established the uniform estimates on $q^{X}_{b,t}$ for $b$ large, 
from which he gave a quantitative estimate on how much this diffusion differs from 
the geodesic flow.

In Subsection \ref{section-infinite}, we also define the 
$\sigma$-twisted orbital integral for the (hypoelliptic) heat kernel of $\mathcal{L}^X_b$. Then 
in Theorem \ref{thm_trequaltrs}, we establish an identity which says that, for $b>0$, $t>0$,
\begin{equation}
	\mathrm{Tr}^{[\gamma\sigma]}[\exp(-t\mathcal{L}^X)]=\mathrm{Tr_{s}}^{[\gamma\sigma]}[\exp(-t\mathcal{L}^X_b)].
	\label{eq:3300}
\end{equation}
Theorem \ref{thm:0000} is obtained by evaluating the right-hand side 
of \eqref{eq:3300} as 
$b\rightarrow +\infty$. As we will explain in Subsection 
\ref{proof5.3}, this evaluation can be 
localized near $X(\gamma\sigma)$, more precisely, the 
$\gamma\sigma$-periodic points of the geodesic flow on 
$\widehat{\mathcal{X}}$. Then, using methods of local index theory, we 
can explicitly work out its limit as $b\rightarrow +\infty$ and 
obtain \eqref{eq:55}.

\subsection{Local equivariant index theorems}

Let $\Gamma$ be a cocompact torsion-free discrete subgroup of $G$ such that 
$\sigma(\Gamma) = \Gamma$. Then $Z=\Gamma\backslash X$ is a compact smooth manifold 
equipped with an isometric action of $\Sigma^\sigma$. The vector bundle $F$ 
descends to one on $Z$, so that the action of 
$\Sigma^\sigma$ on $Z$ lifts to $F\rightarrow Z$.
Moreover, $\mathcal{L}^{X}$ descends 
to a Bochner-like Laplacian $\mathcal{L}^{Z}$ on $Z$, whose heat 
operators are trace class. 

The twisted trace formula shows,
\begin{equation}
\mathrm{Tr}\big[\sigma^Z \exp(-t\mathcal{L}^{Z})\big]=\sum_{\underline{[\gamma]}_{\sigma}\in 
[\Gamma]_{\sigma}} \mathrm{Vol}\big(\Gamma\cap 
Z_{\sigma}(\gamma)\backslash 
X(\gamma\sigma)\big)\mathrm{Tr}^{[\gamma\sigma]}\big[\exp(-t\mathcal{L}^{X})\big],
\label{eq:0.5.1sss}
\end{equation}
where $[\Gamma]_{\sigma}$ is the set of $\sigma$-twisted conjugacy 
classes in $\Gamma$.

Under the geometric setting in Section \ref{section5bonn}, the 
operator $\mathcal{L}^{Z}$ can be replaced by the Laplacian for spinors, or the Hodge Laplacian for a 
Hermitian flat vector bundle ($F$ is defined by a 
$G^{\sigma}$-representation $(E,\rho^{E})$). Then, combining \eqref{eq:55} with 
\eqref{eq:0.5.1sss}, we get a formula for the $\sigma$-equivariant 
heat trace, from which we can evaluate the equivariant Dirac index or $\sigma$-equivariant real analytic torsion.

For the example of equivariant Euler characteristic number 
$\chi_{\sigma}(Z,F)$ with 
a flat vector bundle $F$ ($\mathcal{L}^{Z}$ 
is the Hodge Laplacian $\mathbf{D}^{Z,F,2}$ up to a parallel 
endomorphism of $F$), we will show that all the term 
$\mathrm{Tr}^{[\gamma\sigma]}[\exp(-t\mathbf{D}^{X,F,2})]$ in 
\eqref{eq:0.5.1sss} vanish except for the elliptic class 
$\underline{[\gamma]}_{\sigma}$ (i.e., $\gamma\sigma$ has fixed 
points in $X$). Let $\underline{E}_{\sigma}\subset [\Gamma]_{\sigma}$ 
denote the finite set of elliptic classes, and let 
${}^{\sigma}Z\subset Z$ denote the fixed point set of isometry 
$\sigma$. Then, in Subsection \ref{s1.9}, we get
\begin{equation}\label{eq:0.0.10bb}
    {}^\sigma Z=\cup^{\mathrm{disjoint}}_{\underline{[\gamma]}_{\sigma} \in \underline{E}_\sigma} \Gamma\cap Z_{\sigma}(\gamma)\backslash X(\gamma\sigma),
\end{equation}
where $X(\gamma\sigma)$ is just the fixed point set of $\gamma\sigma$ 
in $X$.

In Subsections \ref{ss:5.2elliptic} and \ref{s7.7}, we exhibit how to 
proceed a 
further evaluation on the integral in the right-hand side of 
\eqref{eq:55} by analyzing the representation $(E,\rho^{E})$. In 
particular, if $\underline{[\gamma]}_{\sigma} \in 
\underline{E}_\sigma$, 
\begin{equation}
	\mathrm{Tr}^{[\gamma\sigma]}\big[\exp(-t\mathbf{D}^{X,F,2})\big]=\Big[e\big(TX(\gamma\sigma), \nabla^{TX(\gamma\sigma)}\big)\Big]^{\mathrm{max}}\mathrm{Tr}^{E}\big[\rho^{E}(\gamma\sigma)\big],
	\label{eq:0.5.3stt}
\end{equation}
where $e\big(TX(\gamma\sigma), \nabla^{TX(\gamma\sigma)}\big)$ denotes the Euler 
form of $X(\gamma\sigma)$, hence identified locally with the Euler 
form of ${}^{\sigma}Z$. Finally, we assembly together the above 
computations, we get
\begin{equation}
	\chi_{\sigma}(Z,F)=\sum_{\underline{[\gamma]}_{\sigma} \in 
	\underline{E}_\sigma}\chi\Big(\Gamma\cap Z_{\sigma}(\gamma)\backslash 
	X(\gamma\sigma)\Big) 
	\mathrm{Tr}^{E}\big[\rho^{E}(\gamma\sigma)\big],
\end{equation}
where $\chi(\cdots)$ denotes the corresponding Euler characteristic 
number. This is clearly a specialization of the local equivariant 
index theorem (cf. \cite[Chapter 6]{berline2003heat}) for the locally 
symmetric space $Z$.

In the last subsection, we introduce the $\sigma$-twisted 
$L_{2}$-torsion $\mathcal{T}_{\sigma, L_{2}}(Z,F)$, as an 
extension of the definition in \cite{BeLip2017}. We explain briefly 
that $\mathcal{T}_{\sigma, L_{2}}(Z,F)$ plays a similar role as the 
ordinary $L_{2}$-torsions (\cite{MR1158345}, \cite{MATHAI1992369}), but 
associated with the fixed point set ${}^{\sigma}Z$.

\subsection{The organization of the paper}\label{s0.13}
This paper is organized as follows. In Section \ref{s1}, we 
introduce the real reductive Lie group $G$ and the twist $\sigma$, 
and we explain the associated 
geometric structure for $X(\gamma\sigma)$ when $\gamma\sigma$ is 
semisimple.

In Section \ref{section3}, we establish the geometric formulation for 
the $\sigma$-twisted orbital integrals 
associated with $\gamma$.

In Section \ref{s4}, we restate Theorem \ref{thm:0000} as Theorem 
\ref{thm_orbitalintegral}, and we give a vanishing theorem by classifying the 
representations of $K^{\sigma}$. We also explain our formula for the 
examples from cyclic base change theory.

In Section \ref{ch3}, we recall the construction of the hypoelliptic 
Laplacian $\mathcal{L}^{X}_{b}$ of Bismut and the properties of its 
heat kernel \cite{bismut2011hypoelliptic}.

In Section \ref{section:proof}, we prove Theorem \ref{thm_orbitalintegral}.

Finally, in Section \ref{section5bonn}, we show the compatibility of our 
formula \eqref{eq:55} with the local equivariant index 
theorems for compact locally symmetric spaces. In the last part, we 
discuss briefly the twisted $L_{2}$-torsion introduced in 
\cite{BeLip2017}.

This paper is mainly the first part of the author's thesis 
\cite{liu:tel-01841334}, and the main results were announced in \cite{LIU201974}.

In the sequel, if $V$ is a real vector space and if $E$ is a 
complex vector space, we will denote by $V\otimes E$ the complex 
vector space $V\otimes_{\R} E$. We use the same convention for the 
tensor product of vector bundles.

If $H$ is a Lie group, let $H^0$ denote the connected component of identity.

\subsection*{Acknowledgments}
This paper presents a part of the results obtained in my Ph.D. thesis 
at Universit\'{e} Paris-Sud 11 (Orsay). I am deeply grateful to my 
thesis supervisor, Prof. Jean-Michel Bismut, for encouraging me to work on this subject, and for many useful discussions. I also want to 
express my sincere gratitude to Laboratoire de Math\'{e}matiques 
d’Orsay (LMO) for providing so nice research and study
environment. Last but not least, I would like to thank the 
anonymous referees for their suggestions and comments that greatly improved the quality of the redaction.

\section{The symmetric space $X=G/K$ and semisimple isometries}\label{s1}

In this section, we consider a connected real reductive Lie group $G$, and let $X$ be the 
associated symmetric space. We introduce a compact subgroup $\Sigma$ of $\mathrm{Aut}(G)$ which acts on $X$ 
isometrically, then for each semisimple element $\gamma\sigma:= 
(\gamma,\sigma)\in G\rtimes \Sigma$, we construct a symmetric space 
$X(\gamma\sigma)\subset X$ associated with the $\sigma$-twisted 
centralizer of $\gamma$ in $G$. Our results here are direct extensions of the results 
obtained in \cite[Chapter 3]{bismut2011hypoelliptic}. They are 
necessary to establish the geometric formulation of the twisted 
orbital integrals in Section \ref{section3}.

\subsection{Symmetric spaces and homogeneous vector bundles}\label{section1-1}

Let $G$ be a connected real reductive Lie group \cite[\S 
7.2]{knapp2002liegroupe} with a Cartan involution $\theta$. Let 
$K\subset G$ be the fixed point set of $\theta$, which is a connected maximal 
compact subgroup. Let $\mathfrak{g}, \kk$ be the Lie algebras of $G$, 
$K$ respectively. Let $\pp\subset \g$ be the eigenspace of $\theta$ 
associated with the eigenvalue $-1$. Then the Cartan decomposition of 
$\mathfrak{g}$ is given by
\begin{equation}\label{cartandecom1}
	\mathfrak{g}=\mathfrak{p}\oplus\mathfrak{k}.
\end{equation}
Moreover,
\begin{equation}
[\kk,\pp]\subset \pp,\;\; [\kk,\kk]\subset \kk,\;\; [\pp,\pp]\subset \kk.
\label{eq:1.1.2ugc}
\end{equation}
Put $m=\dim\pp, n=\dim\kk$. Then $\dim \g=m+n$

Let $B$ be a nondegenerate bilinear symmetric form on $\g$ which is positive-definite on $\pp$ and negative-definite 
on $\kk$. We also assume that $B$ is invariant under the action of 
$\theta$ and the adjoint action of $G$. Let 
$\langle\cdot,\cdot\rangle$ be the scalar product on $\g$ defined by 
$-B(\cdot,\theta\cdot)$. Then the splitting \eqref{cartandecom1} is orthogonal with respect to $B$ and $\langle\cdot,\cdot\rangle$.

For $g,g'\in G$, put
\begin{equation}
C(g)g'=gg'g^{-1}\in G.
\label{eq:1.1.3pp}
\end{equation}
Let $\Ad(\cdot)$, $\ad(\cdot)$ 
denote respectively the adjoint actions of $G$, $\g$ on $\g$. We also use $\Ad(g)$ abusively to denote the conjugation $C(g)$ on $G$.

Let $\omega^\g=\omega^\pp+\omega^\kk$ be the canonical left-invariant $1$-form on $G$ 
with values in $\g=\pp\oplus\kk$. Then by \eqref{cartandecom1}, \eqref{eq:1.1.2ugc}, we get
\begin{equation}
d\omega^\pp=-[\omega^\kk,\omega^\pp],\;\; d\omega^\kk=-\frac{1}{2}[\omega^\kk,\omega^\kk]-\frac{1}{2}[\omega^\pp,\omega^\pp].
\label{eq:1.1.7ugc}
\end{equation}

Let $X=G/K$ be the associated symmetric space. The projection $p: 
G\rightarrow G/K$ defines a $K$-principal
bundle on $X$, then the 
splitting \eqref{cartandecom1} gives it a connection with the connection 
form $\omega^\kk$. 
Let $\Omega$ be the associated curvature, then by \eqref{eq:1.1.7ugc},
\begin{equation}\label{eq:1.1.5n}
    \Omega=-\frac{1}{2}[\omega^\pp,\omega^\pp]\in \Lambda^2 
    (\pp^*)\otimes \kk.
\end{equation}

If $(E,\rho^E, h^{E})$ is a finite dimensional orthogonal 
(resp. unitary) representation of $K$, then $(F= G\times_K E, h^{F})$ is a Euclidean 
(resp. Hermitian) vector bundle on $X$. The connection form 
$\omega^\kk$ induces a Euclidean (resp. Hermitian) connection 
$\nabla^F$ on $F$. The actions of $G$ and $\theta$ on $X$ lift to $F$.

Note that $K$ acts on $\pp$ via adjoint action. Then we have the 
identification 
\begin{equation}
TX=G\times_K \pp.
\label{eq:1.1.3ppaa}
\end{equation}
Moreover, $B|_\pp$ induces a Riemannian metric
$g^{TX}$ on $TX$. Then $G$ and $\theta$ act on $X$ isometrically. 
Let $d(\cdot,\cdot)$ denote the Riemannian distance on $X$. By \eqref{eq:1.1.7ugc}, \eqref{eq:1.1.3ppaa}, $\omega^\kk$ induces the 
Levi-Civita connection $\nabla^{TX}$ on $(TX, g^{TX})$. Let $R^{TX}$ 
denote its curvature. Then $X$ has nonpositive sectional curvature. 
If $x_{0}=p1\in 
X$, the exponential map $\exp_{x_{0}}: \pp\rightarrow X$ given by 
$Y^{\pp}\in\pp \rightarrow \exp_{x_{0}}(Y^{\pp})= \exp(Y^{\pp})\cdot 
x_{0}$ is a 
diffeomorphism between $\pp$ and $X$. 

Put
\begin{equation}
\label{eq:1.1.6n} 
N=G\times_K \kk.
\end{equation}
Let $\nabla^N$ be the connection on $N$ associated with $\omega^\kk$. By \eqref{eq:1.1.3ppaa}, \eqref{eq:1.1.6n},
\begin{equation}
    TX\oplus N=G\times_K \g.
    \label{eq:1.1.8pps}
\end{equation}
Let $\nabla^{TX\oplus N}$ be the connection on $TX\oplus N$ 
associated with $\omega^\kk$, equivalently, $\nabla^{TX\oplus 
N}=\nabla^{TX}\oplus \nabla^N$.

In the sequel, let $\pi: \mathcal{X}\rightarrow X$ be the total space of $TX$ to 
$X$, and let $\widehat{\pi}: \widehat{\mathcal{X}}\rightarrow X$ be the 
total space of $TX\oplus N$ to $X$. The map $(g,a)\in G 
\times_{K}\g\rightarrow (pg,\mathrm{Ad}(g)a)\in 
X\times\g$ identifies $TX\oplus N$ with the trivial vector bundle 
$\g$ over $X$. Then
\begin{equation}
\widehat{\mathcal{X}} \simeq X\times\g.
\label{eq:1.1.13sud}
\end{equation}

We now go back to the Hermitian vector bundle $F$ on $X$ associated 
with a unitary representation $(E,\rho^E)$ of $K$.
Let $C^\infty(G,E)$ denote the set of smooth functions on $G$ valued in $E$. 
The right multiplication of $K$ on $G$ induces a dot-action of $K$ on 
$C^\infty(G,E)$, such that for $k\in K$, $s\in C^\infty(G,E)$,
\begin{equation}
    \label{eq:actionKs}
	(k.s)(g)=\rho^E(k)s(gk).
\end{equation}
Let $C^\infty_K(G,E)$ be the subspace of $C^\infty(G,E)$ 
of the sections fixed by $K$-dot-action. Let $C^\infty(X,F)$ be the vector space of the smooth sections of $F$ over $X$. 
Then
\begin{equation}\label{eq:1.8.4bis}
C^\infty(X,F)=C^\infty_K(G,E).
\end{equation}

The left action of $G$ on itself induces an equivariant action of $G$ on 
$C^\infty(X,F)$ such that if $s\in C^\infty_K(G,E)$, if $g,h\in G$, then
\begin{equation}
(hs)(g)=s(h^{-1}g).
\label{eq:1.1.13didot}
\end{equation}
Moreover, $\nabla^F$ is $G$-invariant.

\subsection{Semisimple isometries of $X$}\label{section1.05}
Let $\mathrm{Isom}(X)$ be the
Lie group of isometries of $(X, g^{TX})$. Then we have an obvious group homomorphism $G\rightarrow \mathrm{Isom}(X)$.
\begin{definition}\label{def:1.1.1sud}
	 If $\phi\in\mathrm{Isom}(X)$, the displacement 
	 function $d_\phi$ associated with $\phi$  
	 is the function on $X$ defined as
	 \begin{equation}
	 d_\phi(x)= d(x,\phi x)\;,\; x\in X.
	 \label{eq:1.1.12pps}
	 \end{equation}
	 Put $m_\phi=\inf_{x\in X}d_\phi(x)\in \R_{\geq 0}$.
\end{definition}

Since $X$ has nonpositive sectional curvature, by \cite[Chapter 1, 
Example 1.6.6]{eberlein1996geometry}, $d^2_\phi$ is a smooth convex 
function on $X$.

\begin{definition}\label{def:1.1.2n}
	We say that $\phi\in \mathrm{Isom}(X)$ is semisimple 
	if $d_\phi(x)$ reaches its infimum $m_\phi$ in $X$. A semisimple isometry $\phi$ is called elliptic if it has fixed points in $X$, i.e. $m_{\phi}=0$. If $\phi$ is semisimple, put $X(\phi)=\{x\in X\;|\; d_\phi(x)=m_\phi\}$.
\end{definition}

\subsection{A compact subgroup of $\mathrm{Aut}(G)$}\label{s1-2}
Let $\mathrm{Aut}(G)$ be the Lie group of automorphism of $G$ \cite[Theorem 2]{Hochschild1952}.

\begin{definition}
	The semidirect product of $G$ 
	and $\mathrm{Aut}(G)$ is defined as
	\begin{equation}
		G\rtimes \mathrm{Aut}(G):
		=\{(g,\phi)\;:\; g\in G, \phi\in \mathrm{Aut}(G)\},
		\label{eq:1.2.2pps}
	\end{equation}
	with the group multiplication:
	\begin{equation}
		(g_1,\phi_1)\cdot(g_2,\phi_2)=\big(g_1\phi_1(g_2),\phi_1\phi_2\big).
		\label{eq:1.2.3pps}
	\end{equation}
	The unit element is $(1, \mathrm{Id}_G)$. 
	Also $(g,\phi)^{-1}=(\phi^{-1}(g^{-1}),\phi^{-1})$.
\end{definition}

We will often write $g\phi$ instead of $(g,\phi)$ for an element in $G\rtimes \mathrm{Aut}(G)$. 

\begin{definition}\label{def:1.2.2ss}
Put 
\begin{equation}\label{eq:Sigmadef}
\Sigma := \{\phi\in \mathrm{Aut}(G)\;:\;
\phi\theta=\theta\phi,\; \phi \mathrm{\;preserves\;the\;bilinear\;form\;} B\}.
\end{equation}
\end{definition}

Then $\Sigma$ is a compact Lie subgroup of $\mathrm{Aut}(G)$. Let $\ke$ be its Lie algebra.  
The action of $\Sigma$ on $\g$ 
preserves the splitting $\eqref{cartandecom1}$ 
and the scalar product of $\g$. 
Note that $\Sigma$ contains all 
the inner automorphisms
defined by elements in $K$.

Set
\begin{equation}
	\widetilde{G}=G\rtimes \Sigma,\; \widetilde{K}=K\rtimes \Sigma.
	\label{eq:1.3.4-2021}
\end{equation}
They are closed subgroups of $G\rtimes \mathrm{Aut}(G)$. Let 
$\widetilde{\g}$, $\widetilde{\kk}$ denote their Lie algebras 
respectively. As vector spaces, we have $\widetilde{\g}=\g\oplus 
\ke$, $\widetilde{\kk}=\kk\oplus \ke$. Then we have the Cartan 
splitting of $\widetilde{\g}$ (associated with the conjugation of 
$\theta$),
\begin{equation}\label{decomptilde}
 \widetilde{\g}=\pp\oplus \widetilde{\kk}.
\end{equation}
Moreover, we also have the global Cartan decomposition for 
$\pp\times \widetilde{K}\simeq\widetilde{G}$, where the 
diffeomorphism is given by $(f,\tilde{k})\mapsto \exp(f)\tilde{k}\in 
\widetilde{G}$.

\begin{remark}\label{rmk:1.2.3ugcd}
The group $\widetilde{G}$ is not necessarily 
reductive. An example is the case $\R^m$. In this case 
$\widetilde{G}=\R^m \rtimes \mathrm{O}(m)$ and the corresponding Lie 
algebra is $\widetilde{\g}=\R^m \oplus \mathfrak{so}(m)$ with a twisted Lie 
bracket.
\end{remark}

Given $\sigma\in\Sigma$, the map $g\in G\rightarrow \sigma(g)\in G$ 
descends to a diffeomorphism of $X$: $x\in X\rightarrow \sigma(x)\in X$. 
By \eqref{eq:1.1.3ppaa}, \eqref{eq:Sigmadef}, the derivative of $\sigma$ 
is given by $(g,f)\rightarrow(\sigma(g),\sigma(f))$ 
with $g\in G, f\in\pp$. Then $\widetilde{G}$ acts on $X$ isometrically. 
Then
	\begin{equation}
	X=\widetilde{G}/{\widetilde{K}}.
	\label{eq:1.2.16pps}
	\end{equation}

Fix $\sigma\in\Sigma$, let $\Sigma^\sigma$ be the closure of the 
subgroup of $\Sigma$ generated by $\sigma$. Set
	\begin{equation}
		G^\sigma=G\rtimes \Sigma^\sigma,\; K^\sigma=K\rtimes \Sigma^\sigma.
		\label{eq:1.2.13pps}
	\end{equation}
Similarly to \eqref{eq:1.2.16pps}, 
\begin{equation}\label{eq:1.2.16}
X=G^\sigma/ K^\sigma.
\end{equation}
Moreover, we have 
\begin{equation}\label{eq_TXidentity}
TX=\widetilde{G}\times_{\widetilde{K}}\pp = G^\sigma\times_{K^\sigma}\pp.
\end{equation} 
In the sequel, $p$ denotes both the projections $\widetilde{G}\rightarrow X$ and $G^\sigma\rightarrow X$.

If the representation $\rho^E: K\rightarrow \mathrm{Aut}(E)$ extends to 
a representation $\rho^E: K^{\sigma}\rightarrow \mathrm{Aut}(E)$, 
then we have the identification of vector bundles over $X$,
\begin{equation}\label{eq_Fid}
 F= G^{\sigma}\times_{K^{\sigma}} E.
\end{equation} 
The question on such extensions is studied in Subsection 
\ref{newsection}.
In this case, the equivariant action of $\sigma$ on $F$ is 
represented by 
$\sigma(g,f)\rightarrow (\sigma(g),\rho^E(\sigma)f)$. Moreover, as in \eqref{eq:1.8.4bis}, we have
\begin{equation}\label{XFsection}
C^\infty(X,F)=C^\infty_{K^{\sigma}}(G^{\sigma},E).
\end{equation}
Then $G^{\sigma}$ acts on $C^\infty(X,F)$. Also $\nabla^F$ is 
invariant under the action of
$G^{\sigma}$.

\subsection{The decomposition of semisimple elements in $\protect\widetilde{G}$}\label{s1-4}
\begin{definition}\label{def:1.2.8ss}
An element of $\widetilde{G}$ is semisimple (resp. elliptic) if its 
isometric action on $X$ is semisimple (resp. elliptic).
\end{definition}

If $\gamma\in \widetilde{G}$, let $\widetilde{Z}(\gamma)$ be the 
centralizer of $\gamma$ in $\widetilde{G}$. Then $d_{\gamma}$ 
is $\widetilde{Z}(\gamma)$-invariant. Recall that if $\gamma$ is 
semisimple, $X(\gamma)$ is the minimizing set of $d_{\gamma}$.

Using instead the identification \eqref{eq:1.2.16pps}, and by the geometric 
properties of $(X,g^{TX})$, the same arguments in the proof of
\cite[Theorem 3.1.2]{bismut2011hypoelliptic} give the 
following criterion on the set $X(\gamma)$.

\begin{theorem}\label{thm_keythm}
Assume that $\gamma\in\widetilde{G}$ is semisimple. If $g\in 
\widetilde{G}$, $x=pg\in X$, then $x\in X(\gamma)$ if and only if 
there exist $a\in\mathfrak{p},\, k\in \widetilde{K}$ such that 
$\mathrm{Ad}(k)a=a$ and $\gamma=C(g)(e^a k^{-1})$. If 
$g_{t}=ge^{ta}$, then $t\in[0,1]\rightarrow x_{t}=pg_{t}$ is the 
unique geodesic connecting $x$ and $\gamma x$. Moreover, 
$m_\gamma=|a|$, and $k\in \widetilde{K}$ is the parallel transport 
along the above geodesic.
\end{theorem}

By Theorem \ref{thm_keythm}, $\gamma\in \widetilde{G}$ is elliptic if and only if it 
is conjugate in $\widetilde{G}$ to an element of $\widetilde{K}$. An element $\gamma\in \widetilde{G}$ is said to be 
hyperbolic if it is conjugate in $\widetilde{G}$ to $e^{a}, a\in\pp$, 
which is always semisimple. Moreover, a hyperbolic element 
always lies in $G$, and can be conjugate to $\exp(\pp)$ by an element 
in $G$.

If $a\in \g$, let $Z(a)\subset G$, $\widetilde{Z}(a)\subset 
\widetilde{G}$ be the stabilizers of $a$, and let $\z(a)\subset \g$, 
$\tilde{\z}(a)\subset \tilde{\g}$ be their Lie algebras.
If $a\in \pp$, by the same arguments as in the proof to 
\cite[Proposition 3.2.8]{bismut2011hypoelliptic}, we have
\begin{equation}\label{eq_1.24}
	Z(a)=Z(e^{a}),\;\widetilde{Z}(a)=\widetilde{Z}(e^a).
\end{equation}
Also we have,
\begin{equation}
\tilde{\z}(a)=\{f\in \tilde{\g}\;:\; [f,a]=0\},\;\z(a)=\tilde{\z}(a)\cap\g,\;
\label{eq:1.3.3pps}
\end{equation}

The group $\widetilde{G}$ may fail to be a reductive Lie group, but it is 
not far from it (twisted by a compact group), so that $\widetilde{G}$ 
still has the properties of a connected reductive Lie 
group discussed in \cite[Theorem 2.19.23]{eberlein1996geometry} and 
\cite[Subsection 3.1]{bismut2011hypoelliptic}. For the sake of 
completeness, we include proofs to these properties.

\begin{proposition}\label{prop_commutative}
	Assume that $\gamma\in\widetilde{G}$ is such that
\begin{equation}
\gamma=e^a k^{-1},\; a\in \pp,\; k\in\widetilde{K},\; \mathrm{Ad}(k)a=a.
\label{eq:1.3.5pps}
\end{equation}
Then we have
\begin{equation}
\widetilde{Z}(\gamma)=\widetilde{Z}(e^a)\cap \widetilde{Z}(k^{-1}).
\label{eq:1.3.6pps}
\end{equation}
\end{proposition}

\begin{proof}
	By Theorem \ref{thm_keythm}, $\gamma$ is semisimple, and 
	$x_{0}=p1\in X(\gamma)$. We only need to prove that 
	$\widetilde{Z}(e^a)\cap \widetilde{Z}(k^{-1})\supset 
	\widetilde{Z}(\gamma)$. We will adapt the arguments of \cite[Theorem 3.2.6 and Proposition 3.2.8]{bismut2011hypoelliptic}.

	Take $h\in \widetilde{Z}(\gamma)$. Then there exists unique $f\in 
	\pp$ and $k'\in \widetilde{K}$ such that $h=e^f k'$. Then $\gamma 
	x_{0}, hx_{0}=pe^f, \gamma h x_{0}\in X(\gamma)$. 
	
	Let $y_s=pe^{sa}, s\in [0,1]$ be the unique geodesic in $X$ joining 
	$x_{0}$ and $\gamma x_{0}$. Let $x_t=pe^{tf},t\in [0,1]$ be the unique 
	geodesic connecting $x_{0}$ and $hx_{0}$. Since $X(\gamma)$ is 
	geodesically convex, then the paths $y_\cdot,x_\cdot$ lie in 
	$X(\gamma)$. Also we have two other geodesics $\gamma x_\cdot, 
	hy_\cdot$ in $X(\gamma)$. These four geodesics form a geodesic 
	rectangle in $X(\gamma)$ with the vertexes $x_{0},\gamma x_{0}, 
	hx_{0},\gamma hx_{0}=h\gamma x_{0}$.
	
	Let $c_t(s),0\leq s\leq 1$ be the geodesic connecting $x_t$ and $\gamma x_t$ for all $t$. If $t\in [0,1]$, let $E_f(t)$ be the energy function associated with $c_t(\cdot)$, i.e.,
	\begin{equation}\label{energyfunction}
	E_f(t)=\frac{1}{2}d^2_\gamma(x_t).
	\end{equation}
	In particular, $E_f(t)$ is a constant function in $t$, so that
	\begin{equation}
	E''_f(0)=0.
	\label{eq:1.3.9sud}
	\end{equation}

	Put $J_s=\frac{\partial}{\partial t}|_{t=0} c_t(s)$ the Jacobi 
	field along $c_{0}(s)$. In the trivialization by parallel transport,
	\begin{equation}
	\begin{split}
	&\ddot{J}_s-\mathrm{ad}^2(a)J_s=0,\\
	&J_0 =f, \;\; J_1=\mathrm{Ad}(k^{-1})f,
	\end{split}
	\label{eq:1.3.8sudugc}
	\end{equation}
	where $\dot{J},\ddot{J}$ are taken with respect 
	to the Levi-Civita connection along $y_{\cdot}$.

We also have
	\begin{equation}
	E''_f(0)=\int_0^1 \left(|\dot{J}_s|^2+\big|[a, J_s]\big|^2\right)ds.
	\label{eq:1.3.11sud}
	\end{equation}

By \eqref{eq:1.3.9sud}, \eqref{eq:1.3.8sudugc}, \eqref{eq:1.3.11sud}, we get
	\begin{equation}\label{eq_1.26}
		f\in \mathfrak{z}(a)\cap \pp,\;\;\mathrm{Ad}(k)f=f.
	\end{equation}
	Applying \eqref{eq_1.26} to $h=e^f k'$, $h\gamma=\gamma h$, we obtain
	\begin{equation}
	e^{\mathrm{Ad}(k')a} k'k^{-1}=e^a k^{-1}k'.
	\label{eq:1.3.9pps}
	\end{equation}
	Using the uniqueness of global Cartan decomposition of $G$, we get
	\begin{equation}
	\mathrm{Ad}(k')a=a,\; k'k^{-1}=k^{-1}k'.
	\label{eq:1.3.10pps}
	\end{equation}	
	By \eqref{eq_1.26}, \eqref{eq:1.3.10pps}, we get $h\in 
	\widetilde{Z}(e^a)\cap \widetilde{Z}(k^{-1})$. This completes our 
	proof.
\end{proof}

In general, if $\gamma\in \widetilde{G}$ is semisimple, then by Theorem \ref{thm_keythm}, there exist $g\in\widetilde{G}$, $a\in\pp$, $k\in\widetilde{K}$ such that
\begin{equation}
	\gamma=g e^a k^{-1} g^{-1},\; \mathrm{Ad}(k)a=a.
	\label{eq:1.3.15ugcsud}
\end{equation} 

Put
\begin{equation}
\gamma_h=ge^ag^{-1},\;\; \gamma_e=g k^{-1}g^{-1}.
\label{eq:1.3.16ugc}
\end{equation}
The element $\gamma_h$ (resp. $\gamma_e$) is called the hyperbolic (resp. elliptic) part of $\gamma$. Then $\gamma=\gamma_h\gamma_e=\gamma_e\gamma_h$. By Proposition \ref{prop_commutative},
\begin{equation}\label{centralizerint}
\widetilde{Z}(\gamma)=\widetilde{Z}(\gamma_e)\cap \widetilde{Z}(\gamma_h).
\end{equation}

\begin{lemma}\label{thm_uniqueness}
	If $\gamma\in \widetilde{G}$ is semisimple, then the 
	decomposition of $\gamma$ as the commuting product of a 
	hyperbolic element and an elliptic element in $\widetilde{G}$ is 
	unique.
\end{lemma}

\begin{proof}
	It is enough to prove our lemma for $\gamma$ given in 
	\eqref{eq:1.3.5pps}, where we have
	\begin{equation}
		\gamma_{h}=e^{a}\in G,\; \gamma_{e}=k^{-1}\in \widetilde{K}.
		\label{eq:oberwolfach}
	\end{equation}
	Now suppose that $\gamma'_{h}\in G, \gamma'_{e}\in \widetilde{G}$ are 
	respectively hyperbolic, elliptic elements such that
	\begin{equation}
		\gamma=\gamma'_{h}\gamma'_{e}=\gamma'_{e}\gamma'_{h}.
		\label{eq:MFOb}
	\end{equation}
	Then we only need to prove that
	\begin{equation}
		\gamma'_{h}=\gamma_{h},\; \gamma'_{e}=\gamma_{e}.
		\label{eq:MFOc}
	\end{equation}

	Note that the conjugation of $\widetilde{G}$ preserves $G$, then 
	$\gamma'_{h}\in G$. Set
	\begin{equation}
		H=\ker(\mathrm{Ad}:\widetilde{G}\rightarrow \mathrm{Aut}(\g)).
		\label{eq:MFOd}
	\end{equation}
	Then $H\cap G$ is just the center of $G$.
	
	Then the uniqueness of the Jordan 
	decomposition of $\mathrm{Ad}(\gamma)$ implies
	\begin{equation}
		\mathrm{Ad}(\gamma'_{h})=\mathrm{Ad}(\gamma_{h}),\;\mathrm{Ad}(\gamma'_{e})=\mathrm{Ad}(\gamma_{e})\in \mathrm{Aut}(\g).
		\label{eq:MFOe}
	\end{equation}
	This implies that there exists $h\in 
	H\cap G\cap \widetilde{Z}(\gamma'_{e})\cap\widetilde{Z}(\gamma_{e})$ 
	such that
	\begin{equation}
		\gamma'_{h}=h\gamma_{h},\; \gamma'_{e}=h^{-1}\gamma_{e}.
		\label{eq:MFO2019}
	\end{equation}
	
	Write $h=e^{f}k''\in G$ with $f\in \pp$, $k''\in K$. Then by 
	Theorem \ref{thm_keythm} and the assumption 
	that $\gamma'_{e}$ is elliptic, we get $f=0$, so that $h\in K$ 
	and $\gamma'_{e}\in K$.
	
	Since $\gamma'_{h}$ is hyperbolic, then there exist $g'\in G$, 
	$a'\in\pp$ such that $\gamma'_{h}=g e^{a'}g^{-1}$. Then we rewrite the first identity of \eqref{eq:MFO2019} as follows
	\begin{equation}
		ge^{a'}g^{-1}=e^{a}h\in G,\;\mathrm{Ad}(h)a=a.
		\label{eq:MFO2019b}
	\end{equation}
	Using the uniqueness of the elliptic part of a semisimple element 
	in $G$ (cf. \cite[Theorem 2.19.23]{eberlein1996geometry}), we get $h=1$, 
	which implies exactly \eqref{eq:MFOc}. This completes the proof 
	of our lemma.
\end{proof}

\subsection{The minimizing set}\label{s1-5}
Take $\gamma\in G$, $\sigma\in \Sigma$ such that 
$\gamma\sigma\in\widetilde{G}$ is semisimple. For $g\in G$, we have
\begin{equation}
C(g)(\gamma\sigma)=g\gamma \sigma(g^{-1}) \sigma\in G^\sigma.
\label{eq:1.4.1pps}
\end{equation}
Let $C^\sigma : G\rightarrow G$ be such that if $g,h\in G$,
\begin{equation}\label{Csigma}
C^\sigma(g)h = gh\sigma(g^{-1})\in G.
\end{equation}
If $g\in G$, $C^{\sigma}(g)$ acts on the left on $G$, and moreover, 
$C^{\sigma}(g)C^{\sigma}(g')=C^{\sigma}(gg')$.

\begin{definition}\label{def:thalys}
If $\gamma\in G$, let $Z_\sigma(\gamma)\subset G$ be the stabilizer 
of $\gamma$ under the action of $G$ by $C^\sigma$, which is also 
called the $\sigma$-twisted centralizer of $\gamma$ in $G$. Then
\begin{equation}\label{eq:thalys0102}
Z_\sigma(\gamma)=G\cap \widetilde{Z}(\gamma\sigma).
\end{equation}
The orbit 
of $\gamma$ under this action is called $\sigma$-twisted conjugacy 
class of $\gamma$ in $G$. 
\end{definition}

Fix $g\in G$ such that $x=pg\in X(\gamma\sigma)$. By Theorem \ref{thm_keythm}, there exists $a\in\mathfrak{p}$, $k\in K$ such that
\begin{equation}\label{fixg0}
\mathrm{Ad}(k)a=\sigma a,\;\;\gamma=C^\sigma(g)(e^ak^{-1}).
\end{equation}
We have $X(\gamma\sigma)= g X(e^ak^{-1}\sigma)$. Then it is enough to consider the case
\begin{equation}\label{eq:fixg1}
\gamma=e^ak^{-1}\in G,\;a\in\pp,\, k\in K,\, \mathrm{Ad}(k)a=\sigma a.
\end{equation}
In the sequel, we always take $\gamma$ as in \eqref{eq:fixg1}.

By \eqref{centralizerint}, we have
\begin{equation}\label{prop_centralizer}
Z_\sigma(\gamma)=Z(e^a)\cap Z_\sigma(k^{-1}).
\end{equation}
Let $\z_\sigma(\gamma),\,\mathfrak{z}_\sigma(k^{-1})$ be the Lie algebras of $Z_\sigma(\gamma),\,Z_\sigma(k^{-1})$. Then
\begin{equation}\label{lieksigma}
 \mathfrak{z}_{\sigma}(k^{-1})=\{f\in\mathfrak{g}\;:\; \mathrm{Ad}(k)f=\sigma f\}.
\end{equation}
By \eqref{eq_1.24}, \eqref{prop_centralizer}, we get
\begin{equation}\label{minsetdecom}
\mathfrak{z}_\sigma(\gamma)=\mathfrak{z}(a)\cap \mathfrak{z}_\sigma(k^{-1}).
\end{equation}

Put 
\begin{equation}
\pp_\sigma(\gamma):=\z_{\sigma}(\gamma)\cap\pp,\;\;\kk_{\sigma}(\gamma):=\z_{\sigma}(\gamma)\cap \kk.
\end{equation}
Since $\sigma$ preserves the splitting \eqref{cartandecom1}, by 
\eqref{lieksigma}, \eqref{minsetdecom}, we get
\begin{equation}\label{liegammasigma}
\z_\sigma(\gamma)=\pp_{\sigma}(\gamma)\oplus \kk_{\sigma}(\gamma).
\end{equation}

Put
\begin{equation}
K_{\sigma}(\gamma)=Z_\sigma(\gamma)\cap K.
\label{eq:1.4.24ugc}
\end{equation}
Then $\kk_{\sigma}(\gamma)$ is the Lie algebra of 
$K_{\sigma}(\gamma)$.

Let $Z^0_\sigma(\gamma)$ denote the identity component of 
$Z_\sigma(\gamma)$. The following 
result extends \cite[Theorem 3.3.1]{bismut2011hypoelliptic}. Note 
that $x_{0}=p1\in X(\gamma\sigma)$.

\begin{theorem}\label{prop_minset1}
We have
\begin{equation}
X(\gamma\sigma)= X(e^a)\cap X(k^{-1}\sigma)\subset X.
\label{eq:1.4.9ugc}
\end{equation}

In the coordinate system $(\pp,\exp_{x_{0}})$, we have
\begin{equation}\label{eq:thalysparis}
X(e^a)=\pp(a)=\mathfrak{z}(a)\cap \mathfrak{p},\;\; 
X(k^{-1}\sigma)=\pp_\sigma(k^{-1}).
\end{equation}
Then 
\begin{equation}\label{eq1.50}
X(\gamma\sigma)=\pp_{\sigma}(\gamma).
\end{equation}

The action of $Z^{0}_\sigma(\gamma)$ on $X(\gamma\sigma)$ is transitive, 
and the stabilizer of $x=p1\in X(\gamma\sigma)$ is given by 
$Z^{0}_\sigma(\gamma)\cap K$. Then we have the following 
identifications,
\begin{equation}
\begin{split}
 X(\gamma\sigma) &\simeq Z_\sigma(\gamma)/K_{\sigma}(\gamma)\simeq Z^0_\sigma(\gamma)/(Z^0_\sigma(\gamma)\cap K).
 \end{split}
 \label{eq:1.4.25ugc}
\end{equation}

Moreover, $Z^0_\sigma(\gamma)\cap K$ coincides with the identity component $K^0_\sigma(\gamma)$ of $K_\sigma(\gamma)$.
The embedding $K_{\sigma}(\gamma)\rightarrow Z_\sigma(\gamma)$ 
induces the isomorphism of finite groups,
\begin{equation}\label{eq1.69}
\begin{split}
K^0_{\sigma}(\gamma)\backslash K_{\sigma}(\gamma) \simeq Z^0_\sigma(\gamma)\backslash Z(\gamma).
\end{split}
\end{equation}
\end{theorem}
\begin{proof}
	We only prove \eqref{eq:1.4.9ugc} and \eqref{eq:thalysparis}, 
	since other results, as their consequences, follow from the 
	standard arguments on symmetric spaces.
	
	Note that $p:G\rightarrow X$ is surjective. For $y\in 
	X(\gamma\sigma)$, there exists $g\in G$ such that $y=pg$. By Theorem \ref{thm_keythm}, there 
	exists $a'\in\pp,k'\in K$ such that $\gamma\sigma = 
	C(g)(e^{a'}(k')^{-1}\sigma)$. By Lemma \ref{thm_uniqueness}, 
	$e^{a}=ge^{a'}g^{-1}$, $k^{-1}\sigma=g(k')^{-1}\sigma g^{-1}$, 
	thus from Theorem \ref{thm_keythm}, we get
\begin{equation}
y=pg\in X(e^a)\cap X(k^{-1}\sigma).
\label{eq:1.4.10ugc}
\end{equation}
Then
\begin{equation}
	X(\gamma\sigma) \subset X(e^a)\cap X(k^{-1}\sigma).
	\label{eq:1.4.11ugc}
\end{equation}

If $y=pg\in X(e^a)\cap X(k^{-1}\sigma)$. By Theorem \ref{thm_keythm}, there exist $a'\in \pp$, $k_1,k_2\in K$ such that 
\begin{equation}
e^a=C(g) (e^{a'}k_1^{-1})\,,\; \mathrm{Ad}(k_1)a'=a'\,,\; k^{-1}= C^\sigma(g)(k_2^{-1}).
\label{eq:1.4.12ugc}
\end{equation}

By \eqref{centralizerint}, \eqref{fixg0}, we have $k_2^{-1}\sigma\in C(g^{-1})\widetilde{Z}(e^a)=\widetilde{Z}(a')\cap\widetilde{Z}(k_1)$. Put $k'=k_2 k_1\in K$, then 
$e^ak^{-1} \sigma= g e^{a'} (k')^{-1} \sigma g^{-1}$ with 
$\mathrm{Ad}(k')a'=\sigma a'$. Thus $y=pg\in 
X(e^ak^{-1}\sigma)=X(\gamma\sigma)$. This proves \eqref{eq:1.4.9ugc}.

The first identification in \eqref{eq:thalysparis} is proved in 
\cite[Theorem 3.2.6]{bismut2011hypoelliptic}. We only prove the 
second one.  Clearly, $\pp_{\sigma}(k^{-1})\subset X(k^{-1}\sigma)$ under the coordinate $(\pp,\exp_{x_{0}})$.
If $b\in \pp$ is such that $\exp_{x_{0}}(b)\in X(k^{-1}\sigma)$, then there exists $k'\in K$ such that
\begin{equation}
k^{-1}\exp(\sigma(b))=\exp(b)k'. 
\label{eq:1.4.11pps}
\end{equation}

We can rewrite \eqref{eq:1.4.11pps} as
\begin{equation}
\exp\left(\mathrm{Ad}(k^{-1})\sigma(b)\right)k^{-1}=\exp(b)k'.
\label{eq:1.4.12pps}
\end{equation}
Then we get
\begin{equation}
\mathrm{Ad}(k^{-1}\sigma)b=b\;,\; k'=k^{-1}.
\label{eq:1.4.13pps}
\end{equation}
This means exactly $b\in \pp_{\sigma}(k^{-1})$. Then the proof to 
\eqref{eq:thalysparis} is completed.
\end{proof}

\begin{remark}\label{rk:new2022}
 Note that $\theta$ acts on $Z^{0}_{\sigma}(\gamma)$ as an 
	automorphism so that $Z^{0}_{\sigma}(\gamma)$ is a real reductive 
	Lie group, in the sense of \cite[\S 
7.2]{knapp2002liegroupe}, equipped with the Cartan involution 
$\theta|_{Z^{0}_{\sigma}(\gamma)}$ and the invariant bilinear form $	
B|_{\z_{\sigma}(\gamma)}$.
\end{remark}

By \eqref{eq:1.4.25ugc}, as in \eqref{eq_TXidentity}, the tangent 
bundle of $X(\gamma\sigma)$ is given as follows
\begin{equation}
	TX(\gamma\sigma)=Z_{\sigma}(\gamma)\times_{K_{\sigma}(\gamma)}\pp_{\sigma}(\gamma).
	\label{eq:1.5.17gallieni}
\end{equation}
Let $\z^{\perp}_{\sigma}(\gamma)$ be the orthogonal subspace of $\z_{\sigma}(\gamma)$ in $\g$ with respect to $B$. Put
\begin{equation}
\pp^{\perp}_{\sigma}(\gamma)=\z^{\perp}_{\sigma}(\gamma)\cap\pp,\;\kk^{\perp}_{\sigma}(\gamma)=\z^{\perp}_{\sigma}(\gamma)\cap\kk.
\label{eq:1.5.5ugc}
\end{equation}
By \eqref{liegammasigma}, we get
\begin{equation}
\z^{\perp}_{\sigma}(\gamma)=\pp^{\perp}_{\sigma}(\gamma)\oplus \kk^{\perp}_{\sigma}(\gamma).
\label{eq:1.5.6ugc}
\end{equation}

If $N_{X(\gamma\sigma)/ X}$ is the normal vector bundle of 
$X(\gamma\sigma)$ in $X$, by \eqref{eq:1.5.17gallieni}, then
\begin{equation}
\begin{split}
 N_{X(\gamma\sigma)/ X} = Z_\sigma(\gamma) \times_{K_\sigma(\gamma)} {\pp}^{\perp}_\sigma(\gamma).
 \end{split}
 \label{eq:1.5.7ugc}
\end{equation}
Let $\mathcal{N}_{X(\gamma\sigma)/ X}$ be the total space of $N_{X(\gamma\sigma)/ X}\rightarrow X(\gamma\sigma)$. 

Let $P_{\gamma\sigma}: X\rightarrow 
X(\gamma\sigma)$ be the orthogonal 
projection from $X$ into $X(\gamma\sigma)$ \cite[Proposition 
1.6.3]{eberlein1996geometry}. Due to the geometric 
structures established in Theorem \ref{prop_minset1}, the same 
proof (using essentially the convexity of displacement functions) of \cite[Theorems 
3.4.1 and 3.4.3]{bismut2011hypoelliptic} gives the following estimates 
for the displacement function $d_{\gamma\sigma}$ along the normal 
vector space of $X(\gamma\sigma)$. It will guarantee the 
convergence of the twisted orbital integrals defined in Section 
\ref{section3}.

\begin{theorem}\label{thm_normalcoord}
We have the diffeomorphism of $Z_{\sigma}(\gamma)$-manifolds,
\begin{equation}
\rho_{\gamma\sigma}:(g,f)\in \mathcal{N}_{X(\gamma\sigma)/ X 
}\longrightarrow p(g\exp(f))\in X.
\label{eq:1.5.8ugc}
\end{equation}

The action of $\gamma\sigma$ on $X$, through the above diffeomorphism, is represented by the 
map $(g,f)\mapsto (\exp(a)g,\Ad(k^{-1})\sigma(f)\,)$, and the projection $P_{\gamma\sigma}$ is given by $P_{\gamma\sigma}(g,f)=(g,0)$.

There exists $c_{\gamma\sigma}>0$, such that if 
	$f\in\pp^{\perp}_{\sigma}(\gamma)$, $|f|\geq 1$, then
	\begin{equation}\label{lineardist}
	d_{\gamma\sigma}(\rho_{\gamma\sigma}(1,f))\geq |a|+ c_{\gamma\sigma}|f|.
	\end{equation}
	
	There exist $C^\prime_{\gamma\sigma}>0,\; C^{\prime\prime}_{\gamma\sigma}>0$ such that, for $f\in \pp^\perp_\sigma(\gamma)$, if $|f|\geq 1$,
	\begin{equation}
		\big|\nabla d_{\gamma\sigma}(\rho_{\gamma\sigma}(1,f))\big|\geq C^{\prime}_{\gamma\sigma},
		\label{eq:1.5.12ugc}
	\end{equation}
	and if $|f|\leq 1$,
	\begin{equation}
		\big|\nabla 
		d^2_{\gamma\sigma}(\rho_{\gamma\sigma}(1,f))/2\big|\geq C^{\prime\prime}_{\gamma\sigma} |f|.
		\label{eq:1.5.13ugc}
	\end{equation}
\end{theorem}

The group $K_\sigma(\gamma)$ acts on the left on $K$ and on 
$\pp^{\perp}_{\sigma}(\gamma)$. Let
$\pp^{\perp}_{\sigma}(\gamma)_{K_\sigma(\gamma)} \times K$ be the vector bundle on $K_\sigma(\gamma)\backslash 
K$ given by the relation, for $f\in \pp^{\perp}_{\sigma}(\gamma)$, $k\in K$ and $h\in K_\sigma(\gamma)$,
\begin{equation}
(f,k)\sim (\Ad(h)f, hk).
\label{eq:1.5.22}
\end{equation}
Right multiplication by $K$ lifts to 
$\pp^{\perp}_{\sigma}(\gamma)_{K_\sigma(\gamma)} \times K$. By 
Theorem \ref{thm_normalcoord}, we get a diffeomorphism,
\begin{equation}
	\begin{split}
	\varrho_{\gamma\sigma}:(g,f,k)\in Z_{\sigma}(\gamma)\times_{K_\sigma(\gamma)} (\pp^{\perp}_{\sigma}(\gamma)\times K) \rightarrow g e^f k\in G.
	\end{split}
	\label{eq:1.5.23ugc}
\end{equation}
As a consequence, we have
	\begin{equation}\label{eq1.89}
	\begin{split}
	\pp^{\perp}_{\sigma}(\gamma)_{K_\sigma(\gamma)}\times  K = Z_{\sigma}(\gamma)\backslash G.
	\end{split}
\end{equation}
Similarly, by taking the identity components of the twisted 
centralizers, we get
\begin{equation}
	\pp^{\perp}_{\sigma}(\gamma)_{K^0_{\sigma}(\gamma)}\times  K = 
	Z^0_{\sigma}(\gamma)\backslash G.
\end{equation}

\begin{remark}\label{rmk:1.5.5ugc}
	Let $Z^{\sigma}(\gamma\sigma)$, $K^{\sigma}(\gamma\sigma)$ denote 
	the centralizers of $\gamma\sigma$ in $G^{\sigma}$, $K^{\sigma}$ 
	respectively. In Theorems \ref{thm_normalcoord} and \eqref{eq1.89}, if we replace $Z_{\sigma}(\gamma)$ , 
	$K_\sigma(\gamma)$, $G$, and $K$ by $Z^{\sigma}(\gamma\sigma)$, 
	$K^{\sigma}(\gamma\sigma)$, $G^{\sigma}$, and $K^{\sigma}$ respectively, we 
	still have analogue results. The reason is that our proof to them 
	relies on the identities obtained in Proposition 
	\ref{prop_commutative} and Lemma \ref{thm_uniqueness}, which 
	holds for group $G^{\sigma}$, $K^{\sigma}$.
\end{remark}

The following result is classical for linear algebraic groups, such 
as \cite[18.2 Proposition pp.117]{humphreysLinearAlgebraicGroups1975}, \cite[III. Theorem 9.2]{borel1991linear}, \cite[Chapter 1, 
pp.22]{Clozel1989}, etc. Here in our setting, we reproduce a proof using the above 
geometric constructions.

\begin{proposition}\label{prop:closedclass}
	For $\gamma\in G$, the element $\gamma\sigma$ is semisimple if and only if 
	the 
	$\sigma$-conjugacy class of $\gamma$ in $G$ is a closed subset.
\end{proposition}
\begin{proof}
	At first, we assume that $\gamma\sigma$ is semisimple, moreover, 
	we may and we will assume that $\gamma$ is given as in
	\eqref{eq:fixg1}, and let 
$[\gamma]_{\sigma}\subset G$ denote the $\sigma$-conjugacy class of 
$\gamma$. 
Let $\{\gamma_i\}_{i\in\bN}\subset 
[\gamma]_{\sigma}$ be a Cauchy sequence in $G$ with the 
limit $h_0\in G$. In particular, we have, as $i\rightarrow +\infty$,
\begin{equation}
d\big(p\gamma_i, ph_0\big)\rightarrow 0.
\label{eq:1.5.28kkm}
\end{equation}

By \eqref{eq1.89}, for $i\in \bN$, there exists $g_i=e^{f_i}k_i$, 
$f_i\in \pp_{\sigma}^\perp(\gamma)$, $k_i\in K$ such that
\begin{equation}
	\gamma_i=g_i^{-1}\gamma\sigma(g_i).
	\label{eq:1.5.28ppz}
\end{equation}
Then we get, as $i\rightarrow +\infty$,
\begin{equation}
d\big(\gamma\sigma pe^{f_i}, g_i ph_0\big)\rightarrow 0.
\label{eq:1.5.30ppz}
\end{equation}
Using the triangle inequality for the distance $d$ on $X$, by \eqref{eq:1.5.30ppz}, we get, as $i\rightarrow +\infty$,
\begin{equation}
d_{\gamma\sigma}(e^{f_{i}})=d\big(\gamma\sigma pe^{f_i}, 
pe^{f_i}\big)\rightarrow d\big(p1, ph_0\big).
\label{eq:1.5.31ppz}
\end{equation}

Then by the estimate \eqref{lineardist}, we get the set 
$\{f_i\}_{i\in \bN}$ is a bounded set in 
$\pp_{\sigma}^\perp(\gamma)$. Then we can assume that there exist 
$f'\in \pp_{\sigma}^\perp(\gamma)$, $k'\in K$ such that, by 
extracting a 
sub-sequence, as $i\rightarrow +\infty$,
\begin{equation}
f_i\rightarrow f',\;\; k_i\rightarrow k'.
\label{eq:1.5.32ppz}
\end{equation}

Put $g'=e^{f'}k'\in\widetilde{G}$, then as $i\rightarrow +\infty$,
\begin{equation}
	g_i\rightarrow g'.
	\label{eq:1.5.33ppz}
\end{equation}
By \eqref{eq:1.5.28ppz}, we get
\begin{equation}
h_0 = (g')^{-1}\gamma\sigma(g') \in [\gamma]_{\sigma},
\label{eq:1.5.34ppz}
\end{equation}
so that $[\gamma]_{\sigma}$ is a closed subset of $G$.

Now we prove another direction, and assume that $[\gamma]_{\sigma}$ is a 
closed subset. For the semisimplicity of $\gamma\sigma$, it is 
enough to find $x\in X$ such that 
$d_{\gamma\sigma}(x)=m_{\gamma\sigma}$. Let 
$\{g_{i}\}_{i\in\bN}\subset G$ be such that as $i\rightarrow +\infty$,
\begin{equation}
	d_{\gamma\sigma}(pg_{i})=d(pg_{i},\gamma\sigma pg_{i})\rightarrow 
m_{\gamma\sigma}.
	\label{eq:2.5.42paris2022}
\end{equation}
Set $\gamma_{i}=g_{i}^{-1}\gamma\sigma(g_{i})\in[\gamma]_{\sigma}$. 
Then \eqref{eq:2.5.42paris2022} is equivalent to 
$d(p1,p\gamma_{i})\rightarrow m_{\gamma\sigma}$. As a consequence, 
the set
$\{\gamma_{i}\}$ lies in a bounded subset of $G$, hence there exists 
a subsequence $\{\gamma_{k_{i}}\}_{i\in\bN}$ which converges to an 
element 
$h_{0}\in G$ as $i\rightarrow +\infty$. The closedness of 
$[\gamma]_{\sigma}$ infers that $h_{0}=g^{-1}\gamma\sigma(g)$ for 
some $g\in G$. Taking $x=pg\in X$, then 
$d_{\gamma\sigma}(x)=m_{\gamma\sigma}$ and $\gamma\sigma$ is 
semisimple by definition. This completes the proof of our proposition.
\end{proof}

\subsection{The locally symmetric space $Z$}
\label{s1.9}

We fix $\sigma\in \Sigma$ and fix a discrete torsion-free cocompact subgroup 
$\Gamma\subset G$ such that $\sigma(\Gamma)=\Gamma$.
The following lemma is given by \cite[Lemmas 
  1 and 2]{Selberg1960}. 
\begin{lemma}\label{lm_semisimplecocompact}
If $\gamma\in \Gamma$, then 
$\gamma\sigma\in \widetilde{G}$ is semisimple, and 
$\Gamma \cap Z_{\sigma}(\gamma)$ is a cocompact discrete subgroup of 
$Z_{\sigma}(\gamma)$.
\end{lemma}

\begin{definition}\label{def:sigmacc}
	We denote by $[\Gamma]_\sigma$ the set of $\sigma$-twisted conjugacy classes in $\Gamma$. If $\gamma\in \Gamma$, let 
	$\underline{[\gamma]}_{\sigma}$ be the $\sigma$-twisted conjugacy class 
	of $\gamma$ in $\Gamma$. If $\gamma\sigma$ is elliptic, we say that 
	$\underline{[\gamma]}_{\sigma}$ is an elliptic class. Let $\underline{E}_\sigma$ be the set of elliptic classes in 
	$[\Gamma]_\sigma$.
	
	The map $\gamma'\in \Gamma\mapsto 
(\gamma')^{-1}\gamma\sigma(\gamma')\in\underline{[\gamma]}_\sigma$ 
induces the identification
\begin{equation}
	\underline{[\gamma]}_\sigma\simeq {\Gamma\cap 
	Z_{\sigma}(\gamma)}\backslash \Gamma.
	\label{eq:idclass}
\end{equation}

\end{definition}

\begin{lemma}\label{lem:finiteelliptic}
	 The set $\underline{E}_\sigma$ is finite.
\end{lemma}

\begin{proof}
Let $U\subset 
G$ be a compact fundamental domain for the left action of $\Gamma$ on 
$G$ such that $G=\cup_{\gamma\in\Gamma}\gamma U$. Note that $p:G\rightarrow X$ is a proper map. Put
\begin{equation}
	V=p^{-1}(p(U))=U\cdot K.
	\label{eq:Vdef}
\end{equation}
Then $V$ is a compact subset of $G$. We denote by $V^{-1}$ the set of 
the inverses of elements in $V$, both $V^{-1}$ and $V\cdot 
\sigma(V^{-1})$ are 
compact.

For any $\underline{[\gamma]}_{\sigma}\in \underline{E}_\sigma$, there 
exists $\gamma'\in \underline{[\gamma]}_{\sigma}$ such that 
$\gamma'\sigma$ has fixed 
points in $p(V)=p(U)$. Let $g_{\gamma}\in U$ be such that $pg_{\gamma}$ 
is fixed by $\gamma'\sigma$. Then we get
\begin{equation}
	\gamma'\in U K\sigma(U^{-1})\cap \Gamma\subset V\cdot 
	\sigma(V^{-1})\cap \Gamma.
	\label{eq:UK}
\end{equation}
Since $V\cdot \sigma(V^{-1})$ is compact, $V\cdot 
	\sigma(V^{-1})\cap \Gamma$ is a finite set, and the lemma follows.
\end{proof}

Put $Z=\Gamma\backslash X= \Gamma\backslash G/ K$, then $Z$ is a 
compact locally symmetric manifold. The homogeneous vector bundle 
$(F,h^{F},\nabla^{F})$ 
on $X$ defined in Subsection \ref{section1-1} descends to a vector 
bundle on $Z$, which we still denote by $(F,h^{F},\nabla^{F})$.  In particular, the tangent bundle $TX$ descends to 
the tangent bundle $TZ$, and $N$ also descends to a Euclidean vector 
bundle, which we still denote it by $N$. 

Since $\sigma(\Gamma)=\Gamma$. Then $\sigma$ acts isometrically on 
$Z$. Let ${}^\sigma Z\subset Z$ be the fixed point set of 
$\sigma$ in $Z$. If $g\in G$ (resp. $x\in X$), we denote by $[g]_Z$ 
(resp. $[x]_Z$) the corresponding 
point in $Z$.

\begin{lemma}\label{lm:fixedpoint}
	If $\gamma_1,\gamma_2\in \Gamma$ are $\sigma$-twisted conjugate in $\Gamma$, then
	\begin{equation}\label{eq_fixedpoint1}
     [X(\gamma_1\sigma)]_Z=[X(\gamma_2\sigma)]_Z\subset Z.
 \end{equation} 	
	
     If $g\in G$, then $[g]_Z\in{}^\sigma Z$ if and only if there is 
	 $\gamma\in \Gamma$ such that $\gamma\sigma$ is elliptic and that 
	 $pg\in X(\gamma\sigma)\subset X$. 
     If $\underline{[\gamma_1]}_\sigma,\;\underline{[\gamma_2]}_\sigma\in \underline{E}_\sigma$ are distinct classes, then
    \begin{equation}\label{eq_fixedpoint2}
    [X(\gamma_1\sigma)]_Z\cap[X(\gamma_2\sigma)]_Z=\emptyset.
    \end{equation} 
\end{lemma}
 
\begin{proof}
	The first part of our lemma is clear. If $[g]_Z\in  {}^\sigma Z$, 
	then there are $\gamma_0\in \Gamma$, $k_0\in K$ such that
    \begin{equation}
    \sigma(g)=\gamma_0 gk_0.
    \end{equation}
    Then $\gamma_0^{-1}\sigma(g)=gk_0$, so that $pg\in X$ is a fixed point of 
    $\gamma_{0}^{-1}\sigma$, and $\gamma_{0}^{-1}\sigma$ is 
    elliptic. If $x\in X$ and $\gamma\sigma(x)=x$ with some $\gamma\in \Gamma$, then $[x]_Z\in {}^\sigma Z$ by definition. 
    
    Suppose that $\underline{[\gamma_{1}]}_{\sigma}$, 
	$\underline{[\gamma_{2}]}_{\sigma}\in \underline{E}_\sigma$. 
    If $[X(\gamma_1\sigma)]_Z\cap[X(\gamma_2\sigma)]_Z\neq\emptyset$ in 
    $Z$, since $\gamma_{1}\sigma, \gamma_{2}\sigma$ are elliptic,
    there exists $\gamma\in\Gamma$ and $x\in X$ such that
    \begin{equation}\label{eq_prooffix}
    \gamma^{-1}\gamma_1\sigma(\gamma)\sigma(x)=\gamma_2\sigma(x)=x.
    \end{equation}
    Then $\gamma_2^{-1}\gamma^{-1}\gamma_1\sigma(\gamma)\sigma(x)=\sigma(x)$. 
    Since $\Gamma$ is torsion-free, 
    then $\gamma_2=\gamma^{-1}\gamma_1\sigma(\gamma)$, i.e.,
    $\underline{[\gamma_{1}]}_{\sigma}=\underline{[\gamma_{2}]}_{\sigma}$. 
    Then we get \eqref{eq_fixedpoint2}. This completes our proof.
\end{proof}

Using Lemma \ref{lm:fixedpoint}, we get that 
    \begin{equation}\label{eq:1.9.21mm}
    {}^\sigma Z=\cup_{\underline{[\gamma]}_{\sigma} \in 
    \underline{E}_\sigma} [X(\gamma\sigma)]_Z.
    \end{equation}
Moreover, the right-hand side in \eqref{eq:1.9.21mm} is a finite disjoint union.
By Lemma \ref{lm_semisimplecocompact}, $\Gamma\cap Z_{\sigma}(\gamma)$ 
is a cocompact torsion-free discrete subgroup of $Z_{\sigma}(\gamma)$, so that $\Gamma\cap Z_{\sigma}(\gamma)\backslash 
X(\gamma\sigma)$ is a  
    compact smooth manifold

Take $\underline{[\gamma]}_{\sigma}\in 
\underline{E}_\sigma$, let $\gamma\in \Gamma$ be one representative of $\underline{[\gamma]}_{\sigma}$.  If $x\in X(\gamma\sigma)$, if $\gamma_0\in\Gamma$ is such that $\gamma_0 x\in X(\gamma\sigma)$, then an argument like \eqref{eq_prooffix} 
gives that $\gamma_0 \in Z_{\sigma}(\gamma)$. 
Thus the projection $X\rightarrow Z$
induces an identification between 
$\Gamma\cap Z_{\sigma}(\gamma)\backslash X(\gamma\sigma)$ 
and $[X(\gamma\sigma)]_Z\subset Z$. Then \eqref{eq:1.9.21mm} can be 
rewritten as
\begin{equation}\label{eq:7.4.10bb}
    {}^\sigma Z=\cup_{\underline{[\gamma]}_{\sigma} \in \underline{E}_\sigma} \Gamma\cap Z_{\sigma}(\gamma)\backslash X(\gamma\sigma),
\end{equation}

Let $C(Z,F)$ be the 
vector space of continuous sections of $F$ on $Z$, which can be 
identified with 
the subspace of $C(X,F)$ of $\Gamma$-invariant sections over $X$, i.e.,
\begin{equation}\label{ZFsection}
C(Z,F)=C(X,F)^\Gamma.
\end{equation}
Then by \eqref{XFsection}, we get
\begin{equation}\label{ZFsection2}
C(Z,F)=C_K(G,E)^\Gamma.
\end{equation}

Assume now that the vector bundle $F$ is defined via a 
$K^\sigma$-representation $(E,\rho^E)$. Since $\sigma$ preserves $\Gamma$, the action of $\sigma$ descends to $F\rightarrow Z$.

\begin{proposition}\label{prop:1.9.6}
    Take $\underline{[\gamma]}_{\sigma}\in 
\underline{E}_\sigma$. Under the 
identification \eqref{eq:7.4.10bb}, the action of $\sigma$ on the 
bundle $F$ restricted   
   to $[X(\gamma\sigma)]_Z \subset {}^\sigma Z$ is given by the action of 
    $\gamma\sigma$ on the vector bundle $F$ over $\Gamma\cap Z_{\sigma}(\gamma)\backslash X(\gamma\sigma)$.
\end{proposition}

\begin{proof}
    Take $x_0=pg_0\in X(\gamma\sigma)$. There is $k\in K$ such that
    \begin{equation}\label{eq_5.71}
    \gamma=C^\sigma(g_0)(k^{-1}).
    \end{equation}
    By Proposition \ref{prop_minset1} and \eqref{eq_5.71}, we have
    \begin{equation}
    X(\gamma\sigma)={g_0}(X(k^{-1}\sigma)).
    \label{eq:1.9.44ugc}
    \end{equation}
    
 By \eqref{eq_Fid}, \eqref{eq:1.9.44ugc}, we have the identification 
    of vector bundles,
    \begin{equation}\label{eq_equivalencebundles}
    F|_{[X(\gamma\sigma)]_Z}\simeq \Gamma\cap Z_{\sigma}(\gamma)\backslash\; {g_0}\big( Z_\sigma(k^{-1})\times_{K_\sigma(k^{-1})} E\big).
    \end{equation}
    If $g\in Z_\sigma(k^{-1})$, by \eqref{eq_5.71}, we get
    \begin{equation}
    \sigma(g_0 g)=\gamma^{-1} g_0 g k^{-1}.
    \label{eq:1.9.46ugc}
    \end{equation}
     Put $x=p(g_0 g)\in X(\gamma\sigma)$ and $z=[g_0g]_Z\in [X(\gamma\sigma)]_Z$. If $v\in F_{z}\simeq E$, then
    \begin{equation}
    \begin{split}
	    \sigma(z, v)&=(\sigma(z), \sigma v)\\
	    &=[(\sigma(g_0 
	    g),\rho^E(\sigma) v)]_Z\\
	    &=[(g_0 g,\rho^E(k^{-1}\sigma)v)]_Z\in 
	    F_{\sigma(z)}.
	    \end{split}
	    \label{eq:1.9.47ugc}
    \end{equation}
    Take the lift of $[(g_0 g,\rho^E(k^{-1}\sigma)v)]_Z$ around $x$, 
	as $gk^{-1}\sigma=k^{-1}\sigma g$, we have 
    \begin{equation}
    [(g_0 g,\rho^E(k^{-1}\sigma^E)v)]_Z={g_0} {k^{-1}\sigma}{g^{-1}_0}(x,v)={\gamma\sigma}(x,v).
    \label{eq:1.9.48ugc}
    \end{equation}
    This completes the proof of our proposition.
\end{proof}

\section{The twisted orbital integrals}
\label{section3}
In this section, we give a geometric interpretation for the twisted 
orbital integrals associated with a semisimple element in 
$\widetilde{G}$. The constructions given here generalize the
results of \cite[Chapter 4]{bismut2011hypoelliptic}. We fix one element $\sigma\in \Sigma$.

\subsection{An algebra of invariant kernels on $X$}\label{section4.1}
Recall that $(E,\rho^{E})$ is a unitary representation of $K^{\sigma}$, and 
that $F=G\times_{K}E$ is the associated Hermitian vector bundle on $X$.
We introduce a vector 
space $\mathcal{Q}^{\sigma}$ of continuous invariant kernels as follows.
\begin{definition}
	Let $\mathcal{Q}^{\sigma}$ be the vector space of maps $q\in 
	C(G, \mathrm{End}(E))$ which satisfy that
	\begin{itemize}
		\item There exist $C,C'>0$, such that 
		\begin{equation}\label{kerneldecay}
		|q(g)|\leq C\exp\left(-C'd^2(p1,pg)\right),\;\forall\; g\in G.
		\end{equation}
		\item For $k, k'\in K$, we have
		\begin{equation}\label{kernelk}
		q(kgk')=\rho^E(k)q(g)\rho^E(k').
		\end{equation}
		\item Set $\sigma^E=\rho^E(\sigma)\in\mathrm{Aut}(E)$,
		\begin{equation}\label{kernelinvariant2}
q(\sigma(g))=\sigma^E q(g)(\sigma^E)^{-1}\in\mathrm{End}(E).
\end{equation}
	\end{itemize}
\end{definition}

Let $C^b(G,E)$ be the set of bounded continuous functions on $G$ 
valued in $E$.
For $q\in \mathcal{Q}^{\sigma}$ and $g,g'\in G$, put 
\begin{equation}
q(g,g')=q(g^{-1}g')\in \mathrm{End}(E). 
\label{eq:4.1.4ugcdidot}
\end{equation}
By \eqref{kernelk}, $q(g,g')$ defines an integral operator $Q$ acting on 
$C^{b}(G,E)$, which is $K$-equivariant. Then it descends to an 
operator acting on $C^{b}(X,F)$. Let $q(x,x')\in 
\mathrm{Hom}(F_{x'}, F_x)$ be the corresponding continuous kernel on 
$X\times X$, which is just the descent of $q(g,g')$ to $X\times X$. 
Moreover, the condition \eqref{kernelinvariant2} is 
equivalent to say that, for $x,x'\in X$,
\begin{equation}\label{kernelinvariant}
q^X\left(\sigma(x),\sigma(x')\right)=\sigma q^X(x,x')\sigma^{-1}\in \mathrm{Hom}(F_{\sigma(x')}, F_{\sigma(x)}).
\end{equation}
By  \eqref{eq:4.1.4ugcdidot}, and 
\eqref{kernelinvariant}, $Q$ commutes with $G^{\sigma}$-action.

\begin{remark}
	The vector space $\mathcal{Q}^{\sigma}$ with the convolution of 
	kernel functions becomes an associative algebra, it is a 
subalgebra of the one defined in \cite[Definition 
4.1.1]{bismut2011hypoelliptic}.
\end{remark}

We can extend $q\in \mathcal{Q}^\sigma$ to a continuous map 
$\tilde{q}\in C(G^\sigma, \mathrm{End}(E))$ by setting
\begin{equation}\label{eq:3.47n}
\tilde{q}(g\mu)=q(g)\rho^E(\mu)\in \mathrm{End}(E)\; ,\; g\in G, \mu\in\Sigma^\sigma. 
\end{equation}
Then it lifts to a continuous kernel defined on $G^\sigma\times 
G^\sigma$ such that
\begin{equation}\label{eq:3.48n}
 \tilde{q}(g\mu, h\mu')=\tilde{q}((g\mu)^{-1} h\mu')\in \text{End}(E).
\end{equation}
The operator $Q$ can be also expressed as the integral operator on 
$C^b_{K^\sigma}(G^\sigma,E)$ associated with kernel $\tilde{q}$.

Since we are going to define the twisted orbital integral in next 
subsection, we need to introduce the volume measures which are 
involved here. Let $dx$ be the volume element on $X$ induced by the 
Riemannian 
metric. 
Let $dk$ be the normalized Haar measure of $K$. Put
\begin{equation}
	dg=dxdk.
	\label{eq:dgdef}
\end{equation}
Then $dg$ is a left-invariant Haar measure on $G$.  Since $G$ is unimodular, $dg$ is also right-invariant.

Let $dy$ be the volume element on $X(\gamma\sigma)$ induced by Riemannian metric, let $df$ be the volume element on the Euclidean space $\pp^\perp_\sigma(\gamma)$. Then $dydf$ is a volume element on $Z_\sigma(\gamma)\times_{K_\sigma(\gamma)} \pp^\perp(\gamma\sigma)$ which is $Z_\sigma(\gamma)$-invariant.  By Theorem \ref{thm_normalcoord}, there is a smooth positive $K_\sigma(\gamma)$-invariant function $r(f)$ on $\pp^\perp_{\sigma}(\gamma)$ such that we have the identity of volume elements on $X$,
\begin{equation}\label{eq3.15}
dx=r(f)dydf,
\end{equation}
with $r(0)=1$. Moreover, there exist $C>0,\; C'>0$ such that for 
$f\in\pp^\perp(\gamma\sigma)$ (cf. \cite[Chapter IV, Theorem 4.1]{MR0145455}),
\begin{equation}\label{eq_1.99}
r(f)\leq C\exp\big(C' |f|\big).
\end{equation}

Let $dk'$ be the Haar measure on $K_\sigma(\gamma)$ that gives volume $1$ to $K_\sigma(\gamma)$, and let $du$ be the $K-$invariant volume form on $K_\sigma(\gamma)\backslash K$, so that 
\begin{equation}\label{eq3.17}
dk=dk'du.
\end{equation}
Then $dydk'$ defines an invariant Haar measure on the reductive Lie 
group
$Z_{\sigma}(\gamma)$ such that
\begin{equation}
dg=dydk'\cdot r(f)dfdu.
\label{eq:1.5.41ugc}
\end{equation}
By \eqref{eq1.89}, \eqref{eq:1.5.41ugc}, $dv=r(f)dfdu$ is a 
$G$-invariant measure on $Z_{\sigma}(\gamma)\backslash G$.

\subsection{Twisted orbital integrals}\label{section2-4}
Let $\gamma\in G$ be as 
follows
\begin{equation}\label{eq:fixg1bonn}
\gamma=e^ak^{-1}\in G,\;a\in\pp,\, k\in K,\, \mathrm{Ad}(k)a=\sigma a.
\end{equation}
Then $\gamma\sigma$ is semisimple. 

If $q\in\mathcal{Q}^\sigma$, then for $x\in X$, we have $\gamma\sigma 
q\big(x,\gamma\sigma(x)\big)\in \mathrm{End}(F_{\gamma\sigma(x)})$. 
Therefore, $\mathrm{Tr}^F\big[\gamma\sigma 
q\big(x,\gamma\sigma(x)\big)\big]$ is well-defined 
function on $X$. Let $h(y)$ be a compactly supported bounded measurable function on $X(\gamma\sigma)$. 
\begin{proposition}\label{prop:integration}
	The function $\mathrm{Tr}^F\big[\gamma\sigma 
	q\big(x,\gamma\sigma(x)\big)\big]h(p_{\gamma\sigma}x)$ is integrable on $X$. 
	Moreover,
	\begin{equation}\label{eq_3.21oi}
		\begin{split}
			&\int_X \mathrm{Tr}^F\big[\gamma\sigma 
			q\big(x,\gamma\sigma(x)\big)\big]h(p_{\gamma\sigma}x)dx\\
			&\qquad\qquad\qquad=\int_{\pp^{\perp}_{\sigma}(\gamma)} 
			\mathrm{Tr}^E\big[\sigma^E q(e^{-f}\gamma e^{\sigma 
			f})\big]r(f)df \int_{X(\gamma\sigma)}h(y)dy.
		\end{split}
	\end{equation}
\end{proposition}

\begin{proof}
	By \eqref{lineardist} and \eqref{kerneldecay}, the 
	function $\mathrm{Tr}^E[\sigma^E q(e^{-f}\gamma e^{\sigma f})]$ 
	is bounded by $C'\exp(-C|f|^{2})$ with some constants $C, C'>0$ for 
	$f\in \pp^{\perp}_{\sigma}(\gamma)$. By \eqref{eq_1.99}, the 
	integrals in the right-hand side of \eqref{eq_3.21oi} are well-defined. By the identification 
	$\rho_{\gamma\sigma}$ defined in \eqref{eq:1.5.8ugc} and using 
	the Fubini's theorem, we get 
	exactly \eqref{eq_3.21oi}.
\end{proof}

By \eqref{eq1.89}, and using the fact that the 
Haar measures of $K$, $K_\sigma(\gamma)$ have volume $1$, we have
\begin{equation}
\begin{split}
\int_{\pp^{\perp}_{\sigma}(\gamma)} \mathrm{Tr}^E\big[\sigma^E 
q(e^{-f}\gamma e^{\sigma f})\big]r(f)df 
=\int_{Z_\sigma(\gamma)\backslash G}\mathrm{Tr}^E\big[\sigma^E 
q\big(v^{-1}\gamma\sigma(v)\big)\big]dv.
\end{split}
\label{eq:4.2.3ugc}
\end{equation}

Recall that $Z^\sigma(\gamma\sigma)$ denotes the centralizer of $\gamma\sigma$ 
in $G^\sigma$. As said in Remark \ref{rmk:1.5.5ugc}, an analogue of \eqref{eq1.89} for the pair 
$(G^{\sigma},K^{\sigma},Z^\sigma(\gamma\sigma))$ still holds. Then 
the above integrals can be rewritten as integrals on the quotient $Z^\sigma(\gamma\sigma)\backslash G^\sigma$ with the kernel $\tilde{q}$ defined in \eqref{eq:3.47n}. More precisely, put
\begin{equation}\label{eq:3.2.4bonn}
d\tilde{k}=dkd\mu.
\end{equation}
Then $d\tilde{k}$ is the normalized Haar measure on $K^\sigma$. Let 
$d\tilde{k}^\sigma$ be the normalized Haar measure on $K^\sigma(\gamma\sigma)$, and let $d\tilde{\mu}^\sigma$ be the $K^\sigma$-invariant measure on $K^\sigma(\gamma\sigma)\backslash K^\sigma$ such that
\begin{equation}\label{eq:3.2.5bonn}
d\tilde{k}=d\tilde{k}^\sigma d\tilde{\mu}^\sigma.
\end{equation}
Then by \eqref{eq1.89}, we get that
\begin{equation}\label{eq:3.2.6bonn}
d\tilde{v}^\sigma=r(f)dfd\tilde{\mu}^\sigma
\end{equation}
is a measure on $Z^\sigma(\gamma\sigma)\backslash G^\sigma$. Then
\begin{equation}\label{eq3.35}
\int_{\pp^{\perp}_{\sigma}(\gamma)} \mathrm{Tr}^E[\sigma^E q(e^{-f}\gamma e^{\sigma f})]r(f)df=\int_{Z^\sigma(\gamma\sigma)\backslash G^\sigma}\mathrm{Tr}^E[ \tilde{q}(\tilde{v}^{-1}\gamma\sigma \tilde{v}]d\tilde{v}^\sigma.
\end{equation}

Let $[\gamma\sigma]$ denote the conjugation class of $\gamma\sigma$ in $G^\sigma$.

\begin{definition}\label{defn_orbitalintegral}
	For $q\in \mathcal{Q}^\sigma$, set
	\begin{equation}\label{orbitaldef1}
		\begin{split}
		\mathrm{Tr}^{[\gamma\sigma]}[Q]	
		&=\int_{Z_{\sigma}(\gamma)\backslash G}\mathrm{Tr}^E\big[\sigma^E 
		q(v^{-1}\gamma\sigma(v))\big]dv\\
		&=\int_{\pp^{\perp}_{\sigma}(\gamma)} \mathrm{Tr}^E\big[\sigma^E 
		q(e^{-f}\gamma e^{\sigma f})\big]r(f)df.
		\end{split}
	\end{equation}
Integrals like \eqref{eq:4.2.3ugc}, \eqref{eq3.35}, 
\eqref{orbitaldef1} are called twisted orbital integrals. By 
\eqref{eq3.35}, we see that $\mathrm{Tr}^{[\gamma\sigma]}[Q]$ only 
depends on the conjugacy class of $\gamma\sigma$ in $G^\sigma$. In 
particular, if $\gamma'\in G$ is $\sigma$-twisted 
conjugate to $\gamma$, then 
$\mathrm{Tr}^{[\gamma'\sigma]}[Q]=\mathrm{Tr}^{[\gamma\sigma]}[Q]$.

If taking $\sigma=\mathrm{Id}_G$ in \eqref{orbitaldef1}, we get the ordinary (un-twisted) orbital integral $\mathrm{Tr}^{[\gamma]}[Q]$ 
(cf. \cite[Definition 4.2.2]{bismut2011hypoelliptic})
associated with a semisimple element $\gamma\in G$.
\end{definition}

The following proposition extends \cite[Theorem 
4.2.3]{bismut2011hypoelliptic}.
\begin{proposition}\label{prop_3.3.4}
	For $Q,Q'\in \mathcal{Q}^\sigma$, we have
	\begin{equation}\label{eq:3.2.10bonn}
	\mathrm{Tr}^{[\gamma\sigma]}\big[[Q,Q']\big]=0.
	\end{equation}
	Equivalently, $\mathrm{Tr}^{[\gamma\sigma]}[\cdot]$ is a trace on the algebra $\mathcal{Q}^\sigma$.
\end{proposition}
\begin{proof}
Let $R$ be an integral operator defined by a kernel function in 
$\mathcal{Q}^{\sigma}$, and let $R'$ be an integral operator 
associated with a bounded continuous 
invariant kernel function in $C^{b}(G^{\sigma},\mathrm{End}(E))$. They act on continuous sections of $F$ over $X$ with compact 
support. The operators $RR'$, $R'R$ also have integral kernels which 
are bounded on $G^{\sigma}$. We have
\begin{equation}\label{eq:3.2.11bonn}
\mathrm{Tr}^{[1]}\big[[R,R']\big]=0.
\end{equation} 

Let $\delta_{\gamma\sigma}$ be the current on $G^\sigma$ so that
\begin{equation}
\int_{G^\sigma} f\delta_{\gamma\sigma}=\int_{Z^\sigma(\gamma\sigma)\backslash G^\sigma} f((\tilde{v})^{-1}\gamma\sigma\tilde{v})d\tilde{v}^\sigma.
\label{eq:4.2.8didotugc}
\end{equation}
Since $d\tilde{v}^\sigma$ is invariant under the right action of 
$G^\sigma$ on $Z^\sigma(\gamma\sigma)\backslash G^\sigma$, 
$\delta_{\gamma\sigma}$ is invariant by conjugation. For $q\in 
\mathcal{Q}^\sigma$, $\tilde{q}$ is defined by 
\eqref{eq:3.47n}. By \eqref{eq3.35}, \eqref{orbitaldef1},
\begin{equation}
\mathrm{Tr}^{[\gamma\sigma]}[Q]=\int_{G^\sigma}\mathrm{Tr}^E[\tilde{q}]\delta_{\gamma\sigma}=\mathrm{Tr}^E[\tilde{q} * \delta_{(\gamma\sigma)^{-1}} (1)],
\label{eq:4.2.9didotsud}
\end{equation}
where $*$ denotes the convolution on $G^{\sigma}$.

The current 
$\delta_{(\gamma\sigma)^{-1}}$ defines a linear operator 
$R_{(\gamma\sigma)^{-1}}$ acting on $C^{b}(X,F)$. If $Q\in \mathcal{Q}^\sigma$, we have
\begin{equation}
Q R_{(\gamma\sigma)^{-1}}=R_{(\gamma\sigma)^{-1}}Q.
\label{eq:4.2.10didotsud}
\end{equation}
The operator $Q R_{(\gamma\sigma)^{-1}}$ is an integral operator with a bounded continuous invariant kernel.

Then we can rewrite \eqref{eq:4.2.9didotsud} as
\begin{equation}
\mathrm{Tr}^{[\gamma\sigma]}[Q]=\mathrm{Tr}^{[1]}[Q R_{(\gamma\sigma)^{-1}}].
\label{eq:4.2.11didotsud}
\end{equation}
By \eqref{eq:3.2.11bonn}, \eqref{eq:4.2.10didotsud} and \eqref{eq:4.2.11didotsud}, we get
\begin{equation}
\mathrm{Tr}^{[\gamma\sigma]}\big[[Q,Q']\big]=\mathrm{Tr}^{[1]}\big[[Q,Q']R_{(\gamma\sigma)^{-1}}\big]=\mathrm{Tr}^{[1]}\big[[Q,Q'R_{(\gamma\sigma)^{-1}}]\big]=0.
\label{eq:4.2.12didotsud}
\end{equation}
This completes the proof of our proposition.
\end{proof}

\subsection{Infinite dimensional orbital integrals}\label{section-infinite}
Let $dY^\pp$, $dY^\kk$ denote the 
volume elements on the Euclidean spaces $\pp$, $\kk$. Then these volume 
elements are $K^\sigma$-invariant. Moreover, $dY=dY^\pp dY^\kk$ is a 
$G^\sigma$-invariant
volume element on $\g$. Let $dY^{TX},dY^N, dY$ be the corresponding volume elements on the fibres of $TX,N, TX\oplus N$ over $X$.

Let $C^{\infty,b}(\g,\R)$ be the vector space of real valued smooth 
bounded functions on $\g$.  We replace the finite-dimensional vector 
space $E$ by the infinite dimensional space 
$$\mathcal{E}=\Lambda^\bullet 
(\pp^*\oplus \kk^*)\otimes C^{\infty,b}(\g,\R)\otimes E$$
equipped with the natural action of $K^\sigma$. Then the vector bundle $F$ on $X$ is replaced by
\begin{equation}
\mathcal{F}=\Lambda^\bullet(T^*X\oplus N^*)\otimes C^{\infty,b}(TX\oplus N,\R)\otimes F.
\end{equation}
Let $C^b(\widehat{\mathcal{X}},\widehat{\pi}^*(\Lambda^\bullet 
(T^*X\oplus N^*)\otimes F))$ be the space of continuous bounded 
sections of $\widehat{\pi}^*(\Lambda^\bullet(T^*X\oplus N^*)\otimes F)$ over $\widehat{\mathcal{X}}$.

The group $K^\sigma$ acts on $C^b(G^\sigma\times \g, \Lambda^\bullet(\pp^*\oplus \kk^*)\otimes E)$, so that if $s\in C^b(G^\sigma\times \g, \Lambda^\bullet(\pp^*\oplus \kk^*)\otimes E)$ then for $\tilde{k}\in K^\sigma$
\begin{equation}
(\tilde{k}\cdot s)(\tilde{g},Y)=\rho^{\Lambda^\bullet(\pp^*\oplus \kk^*)\otimes E}(\tilde{k})s(\tilde{g}\tilde{k}, \text{Ad}(\tilde{k}^{-1})Y).
\end{equation}
Let $C^b_{K^\sigma}(G^\sigma\times\g,\Lambda^\bullet(\pp^*\oplus \kk^*)\otimes E)$ be the vector space of $K^\sigma$-invariant continuous bounded function on $G^\sigma\times \g$ with values in $\Lambda^\bullet(\pp^*\oplus \kk^*)\otimes E$. Then we have 
\begin{equation}\label{sectionid}
\begin{split}
C^b\big(\widehat{\mathcal{X}},\widehat{\pi}^*(\Lambda^\bullet 
(T^*X\oplus N^*)\otimes 
F)\big)&=C^b_{K^\sigma}\big(G^\sigma\times\g,\Lambda^\bullet(\pp^*\oplus \kk^*)\otimes E\big)\\
&=C^b_K\big(G\times\g, \Lambda^\bullet(\pp^*\oplus \kk^*)\otimes E\big).
\end{split}
\end{equation}

\begin{definition}\label{def:3.3.1bis}
	Let $\mathfrak{Q}^\sigma$ be the vector space of continuous kernels $q(g,Y,Y')$ defined on $G\times \g\times \g$ with values in $\mathrm{End}(\Lambda^\bullet(\pp^*\oplus \kk^*)\otimes E)$ such that
	\begin{itemize}
	\item If $g\in G, k,k'\in K,Y,Y'\in\g$, then 
	\begin{equation}\label{condition1}
		\begin{split}
	&q\big(kgk',Y,Y'\big)\\
	&=\rho^{\Lambda^\bullet(\pp^*\oplus \kk^*)\otimes 
	E}(k)q\big(g,\mathrm{Ad}(k^{-1})Y,\mathrm{Ad}(k')Y'\big)\rho^{\Lambda^\bullet(\pp^*\oplus \kk^*)\otimes E}(k').
		\end{split}
	\end{equation}
	\item Put $\sigma^{\Lambda^\bullet(\pp^*\oplus \kk^*)\otimes E}=\rho^{\Lambda^\bullet(\pp^*\oplus \kk^*)\otimes E}(\sigma)\in \mathrm{Aut}(\Lambda^\bullet(\pp^*\oplus \kk^*)\otimes E)$, then
	\begin{equation}\label{condition2}
	q(\sigma(g),\sigma Y,\sigma Y')=\sigma^{\Lambda^\bullet(\pp^*\oplus \kk^*)\otimes E}q(g,Y,Y')(\sigma^{\Lambda^\bullet(\pp^*\oplus \kk^*)\otimes E})^{-1}.
	\end{equation}
	\item There exist $C, C'>0$ such that 
	\begin{equation}\label{condition3}
	|q(g,Y,Y')|\leq C\exp\Big(-C'\big(d^2(p1,pg)+|Y|^2+|Y'|^2\big)\Big).
	\end{equation}
\end{itemize} 

We will denote by $\mathfrak{Q}^{\sigma,\infty}$ the subspace of 
$\mathfrak{Q}^{\sigma}$ consisting of smooth kernels.
\end{definition}

If $q\in\mathfrak{Q}^\sigma$, put $q((g,Y),(g',Y'))=q(g^{-1}g',Y,Y')$.
If $s\in C^b_K(G\times \g,\Lambda^\bullet(\pp^*\oplus \kk^*)\otimes E)$, put
\begin{equation}
(Qs)(g,Y)=\int_{G\times\g} q\big((g,Y),(g',Y')\big)s(g',Y')dg'dY'.
\label{eq:4.3.8didot}
\end{equation}
By \eqref{condition1}, \eqref{condition3}, $Q$ is an operator acting on $C^b_K(G\times \g,\Lambda^\bullet(\pp^*\oplus \kk^*)\otimes E)$. Equivalently, the operator $Q$ acts on $C^b(\widehat{\mathcal{X}},\widehat{\pi}^*(\Lambda^\bullet (T^*X\oplus N^*)\otimes F))$ with a continuous kernel $q((x,Y),(x',Y'))$.

The action of $\sigma$ on $C^b_K(G\times \g,\Lambda^\bullet(\pp^*\oplus 
\kk^*)\otimes E)$ is represented by
\begin{equation}\label{oplsigma}
(\sigma s)(g,Y)=\sigma^{\Lambda^\bullet(\pp^*\oplus \kk^*)\otimes E}s(\sigma^{-1}(g),\sigma^{-1}Y).
\end{equation}
Then \eqref{condition2} is equivalent to 
$Q\sigma=\sigma Q$. Therefore, $Q$ commutes with $G^\sigma$.

By \cite[Proposition 4.3.2]{bismut2011hypoelliptic} and using the 
fact that $\sigma$ preserves $dxdY$, $\mathfrak{Q}^\sigma$ is an 
algebra with respect to the convolution of kernels. Let 
$[\cdot,\cdot]$ denote the supercommutator with respect to the $\Z_2$-graded structure of $\mathrm{End}(\Lambda^\bullet(\pp^*\oplus \kk^*)\otimes E)$, and let $\mathrm{Tr_s}^{\Lambda^\bullet(\pp^*\oplus \kk^*)\otimes E}[\cdot]$ be the supertrace on $\mathrm{End}(\Lambda^\bullet(\pp^*\oplus \kk^*)\otimes E)$.

If $g\in G$, let $q(g)$ be the operator on $\mathcal{E}$ defined by 
the kernel $q(g, Y,Y')$. Let 
$\sigma^\mathcal{E}\in\mathrm{End}(\mathcal{E})$ denote the action of 
$\sigma$ on $\mathcal{E}$. 
Then for $g\in G$, the operator
$\sigma^{\mathcal{E}}q(g^{-1}\gamma\sigma(g))$ acting on 
$\mathcal{E}$ is given by the continuous kernel 
$\sigma^{\Lambda^\bullet(\pp^*\oplus \kk^*)\otimes 
E}q(g^{-1}\gamma\sigma(g),\sigma^{-1}Y,Y')$ on $\g\times \g$. 
By \eqref{condition3}, the function $\mathrm{Tr_s}^{\Lambda^\bullet(\pp^*\oplus \kk^*)\otimes E}[\sigma^{\Lambda^\bullet(\pp^*\oplus \kk^*)\otimes E}q(g^{-1}\gamma\sigma(g),\sigma^{-1}Y,Y)]$ is integrable on $\g$.

If 
$\sigma^{\mathcal{E}}q(g^{-1}\gamma\sigma(g))$ is trace class, by 
\cite[Proposition 3.1.1]{Duflo1972generalite}, we get
\begin{equation}\label{eq:3.37n}
	\begin{split}
		&\mathrm{Tr_s}^{\mathcal{E}}\big[\sigma^{\mathcal{E}}q(g^{-1}\gamma\sigma(g))\big]\\
		&=\int_{\g} \mathrm{Tr_s}^{\Lambda^\bullet(\pp^*\oplus 
		\kk^*)\otimes E}\Big[\sigma^{\Lambda^\bullet(\pp^*\oplus 
		\kk^*)\otimes 
		E}q\left(g^{-1}\gamma\sigma(g),\sigma^{-1}Y,Y\right)\Big]dY.
	\end{split}
\end{equation}
\begin{remark}\label{rem:traceclass}
	A sufficient condition for our operator to be trace class is 
	that the kernel together with its derivatives in $Y,Y'$ of 
	arbitrary orders lie in the Schwartz space of $\g\times\g$.
\end{remark}

By \eqref{condition3}, there exists $C_{\gamma\sigma}>0$ such that
\begin{equation}
	\begin{split}
		 &\bigg| \int_{\g} \mathrm{Tr_{s}}^{\Lambda^\bullet(\pp^*\oplus 
	\kk^*)\otimes E}\Big[\sigma^{\Lambda^\bullet(\pp^*\oplus \kk^*)\otimes 
	E}q(g^{-1}\gamma\sigma(g),\sigma^{-1}Y,Y)\Big]dY\bigg|\\
		&\hspace{60mm}\leq 
	C_{\gamma\sigma} \exp\left(- C' d^{2}(pg, \gamma\sigma pg)\right).
	\end{split}
	\label{eq:integration}
\end{equation}
By \eqref{lineardist}, the arguments in the proof of Proposition 
\ref{prop:integration} show that the left-hand side of 
\eqref{eq:integration} is integrable on 
$\pp^{\perp}_{\sigma}(\gamma)$.

We extend the notion of the twisted orbital integrals to the infinite 
dimensional case, which is a twisted version of \cite[Definition 4.3.3]{bismut2011hypoelliptic}.
\begin{definition}\label{defn_orbitalsupertrace}
	If $Q$ is given by $q\in \mathfrak{Q}^\sigma$, we define $\mathrm{Tr_s}^{[\gamma\sigma]}[Q]$ as follows,
	\begin{equation}\label{orbitaldef2}
	\begin{split}
	&\mathrm{Tr_s}^{[\gamma\sigma]}[Q]\\
	&=\int_{(Z_{\sigma}(\gamma)\backslash 
	G)\times\g}\mathrm{Tr_s}^{\Lambda^\bullet(\pp^*\oplus \kk^*)\otimes 
	E}\Big[\sigma^{\Lambda^\bullet(\pp^*\oplus \kk^*)\otimes E} 
	q(v^{-1}\gamma\sigma(v),Y,\sigma Y)\Big]dvdY\\
	&=\int_{\pp^{\perp}_{\sigma}(\gamma)\times \g} 
	\mathrm{Tr_s}^{\Lambda^\bullet(\pp^*\oplus \kk^*)\otimes 
	E}\Big[\sigma^{\Lambda^\bullet(\pp^*\oplus \kk^*)\otimes E} 
	q(e^{-f}\gamma e^{\sigma f},Y,\sigma Y)\Big]r(f)dfdY.
	\end{split}
	\end{equation}
	Expressions such as \eqref{orbitaldef2} are called twisted orbital supertraces.
\end{definition}

If $\sigma^{\mathcal{E}}q(g^{-1}\gamma\sigma(g))$ is trace class for 
$g\in G$, then we can rewrite 
\eqref{orbitaldef2} as
\begin{equation}\label{orbitalsupertrace}
	 \mathrm{Tr_s}^{[\gamma\sigma]}[Q]=\int_{\pp^{\perp}_{\sigma}(\gamma)} \mathrm{Tr_s}^{\mathcal{E}}[\sigma^{\mathcal{E}} q(e^{-f}\gamma e^{\sigma f})]r(f)df.
\end{equation}

\begin{proposition}\label{prop_3.4.3}
	If $Q,Q'\in\mathfrak{Q}^\sigma$, then
	\begin{equation}
	\mathrm{Tr_s}^{[\gamma\sigma]}\big[[Q,Q']\big]=0.
	\label{eq:4.3.15didotsud}
	\end{equation}
\end{proposition}

\begin{proof}
By the above constructions, the proof of our proposition is just a modification of the proof of Proposition \ref{prop_3.3.4}. 
\end{proof}
\subsection{Twisted trace formula for locally symmetric spaces}\label{section4.9}
Let $\Gamma$ be a cocompact torsion-free discrete subgroup of $G$ such that $\sigma(\Gamma)= \Gamma$. We still assume that 
$F$ is associated with a finite-dimensional unitary representation 
$(E,\rho^E)$ of $K^\sigma$. Put $Z=\Gamma\backslash X= \Gamma\backslash G/ K$. We use the 
notation in Subsection \ref{s1.9}. Recall that $\Sigma^\sigma$ acts 
isometrically on $Z$ and its action lifts to an action on the 
bundles $TZ$, $F$.

Let $dz$ be the volume element of $Z$ induced by the Riemannian metric. We still denote by $dg$ the volume element on $\Gamma\backslash G$ 
induced by $dg$.

Let $Q$ be an operator with kernel $q\in\mathcal{Q}^\sigma$. Then $Q$ descends to 
an operator $Q^Z$ acting on $C(Z,F)$. Let $q^Z(z,z')$, $z,z'\in Z$ be the 
continuous kernel of $Q^Z$ over $Z$. We also denote by $z,z'$ their 
arbitrary lifts in $X$.
Then
\begin{equation}\label{eq:3.71nn}
q^Z(z,z')=\sum_{\gamma\in\Gamma} \gamma q^X(\gamma^{-1}z,z')=\sum_{\gamma\in\Gamma} q^X(z,\gamma z')\gamma.
\end{equation}
The convergence of the above sums are guaranteed by the cocompactness 
of $\Gamma$ and the condition \eqref{kerneldecay} for $q$.

Note that $\sigma$ acts on $C^\infty(Z,F)$, we will denote it by 
$\sigma^Z$. Then $\sigma Q$ descends to $\sigma^Z Q^Z$.  
By \eqref{kernelinvariant}, \eqref{eq:3.71nn}, the kernel of $\sigma^Z Q^Z$ is given by
\begin{equation}\label{eq:3.73nn}
(\sigma^Z Q^Z)(z,z')=\sum_{\gamma\in \Gamma} q^X(z,\gamma\sigma(z'))\gamma\sigma.
\end{equation}

By \eqref{ZFsection2}, $(\sigma^Z Q^Z)(z,z')$ lifts to $G\times G$, so that
\begin{equation}
(\sigma^Z Q^Z)(g,g')=\sum_{\gamma\in\Gamma} q(g^{-1}\gamma\sigma(g'))\sigma^E\in \text{End}(E).
\label{eq:3.7.5nnn}
\end{equation}
Assume that $Q^Z$ is trace class, so is $\sigma^{Z}Q^{Z}$, then 
\begin{equation}
	\label{eq:3.7.6nnn}
		\mathrm{Tr}[\sigma^Z Q^Z]=\int_Z \mathrm{Tr}^F\big[(\sigma^Z 
		Q^Z)(z,z)\big]dz =\int_{\Gamma\backslash G} \mathrm{Tr}^E\big[(\sigma^Z 
		Q^Z)(g,g)\big]dg.
\end{equation}

By Lemma 
\ref{lm_semisimplecocompact}, if $\gamma\in\Gamma$, $\Gamma\cap Z_{\sigma}(\gamma)\backslash 
Z_{\sigma}(\gamma)$ is a compact smooth manifold. 
Recall that $[\Gamma]_\sigma$ is the set of $\sigma$-twisted 
conjugacy classes in $\Gamma$. The twisted trace formula is given as 
follows,
\begin{equation}\label{eq:tracebonn}
\mathrm{Tr}[\sigma^Z Q^Z]=\sum_{\underline{[\gamma]}_\sigma\in 
[\Gamma]_\sigma} \mathrm{Vol}\big(\Gamma\cap 
Z_{\sigma}(\gamma)\backslash X(\gamma\sigma)\big) \mathrm{Tr}^{[\gamma\sigma]}[Q].
\end{equation}

\section{A formula for semisimple twisted orbital integrals}\label{s4}
The purpose of this section is to present the main results in this 
paper. We get an explicit geometric formula for the twisted 
orbital integrals associated with heat kernels of the Casimir operator, which is an extension 
of Bismut's formula for semisimple orbital integrals \cite[Theorem 
6.1.1]{bismut2011hypoelliptic}. The proof of this formula is deferred 
to Section \ref{section:proof}. In the last 
	subsection, we will apply our formula and explain its 
	consequences in typical examples from cyclic base change theory.

\subsection{The $J$-function $J_{\gamma\sigma}(Y^{\kk}_{0})$ on 
$\kk_{\sigma}(\gamma)$ }\label{s4.1}

The function $\widehat{A}(x)$ is given by 
\begin{equation}
\widehat{A}(x)=\frac{x/2}{\sinh(x/2)}.
\end{equation}
Let $H$ be a finite-dimensional Hermitian vector space. If $B\in 
\mathrm{End}(H)$ is self-adjoint, then $\dfrac{B/2}{\sinh(B/2)}$ is a 
self-adjoint positive endomorphism. Put
\begin{equation}
\widehat{A}(B)=\det{}^{1/2}\left[\frac{B/2}{\sinh(B/2)}\right].
\end{equation}

If $\gamma\in G$, $\sigma\in \Sigma$ are such that 
$\gamma\sigma$ is semisimple, as in Sections \ref{s1} and \ref{section3}, we may and we will 
assume that $\gamma=e^{a}k^{-1}$ with $a\in \pp, 
k\in K$, and $\mathrm{Ad}(k)a=\sigma a$.

Set 
$\z_0=\z(a)=\pp_{0}\oplus \kk_{0}$, then $\z_{\sigma}(\gamma)$ is a 
Lie subalgebra of $\z_0$. Let $\z^{\perp}_0$, $\pp^{\perp}_0$, 
$\kk^{\perp}_0$ be the orthogonal vector spaces to $\z_0$, $\pp_0$, 
$\kk_0$ in $\g, \pp, \kk$, and let 
$\z^{\perp}_{\sigma,0}(\gamma)$, $\pp^{\perp}_{\sigma,0}(\gamma)$, $\kk^{\perp}_{\sigma,0}(\gamma)$ be the orthogonal spaces to $\z_{\sigma}(\gamma)$, $\pp_{\sigma}(\gamma)$, $\kk_{\sigma}(\gamma)$ in $\z_0$, $\pp_0$, $\kk_0$. We have
\begin{equation}
\z^{\perp}_{\sigma,0}(\gamma)=\pp^{\perp}_{\sigma,0}(\gamma)\oplus \kk^{\perp}_{\sigma,0}(\gamma).
\label{eq:5.6.7ugcd}
\end{equation}
For $Y_0^{\kk}\in \kk_{\sigma}(\gamma)$, $\mathrm{ad}(Y^{\kk}_0)$ 
preserves $\pp_{\sigma}(\gamma), \kk_{\sigma}(\gamma), 
\pp^{\perp}_{\sigma,0}(\gamma), \kk^{\perp}_{\sigma,0}(\gamma)$, and it is 
an antisymmetric endomorphism with respect to the scalar product.

As explain in \cite[pp.105]{bismut2011hypoelliptic}, the following function 
$A(Y^{\kk}_{0})$  in $Y^{\kk}_{0}\in \kk_{\sigma}(\gamma)$ has a natural square 
root that is analytic,
\begin{equation}
		A(Y^{\kk}_{0})=\frac{1}{\det 
		(1-\mathrm{Ad}(k^{-1}\sigma))|_{\z^{\perp}_{\sigma,0}(\gamma)}} \frac{\det \big(1-\exp(-i\mathrm{ad}(Y_0^\kk))\mathrm{Ad}(k^{-1}\sigma)\big)|_{\kk^{\perp}_{\sigma,0}(\gamma)}}{\det \big(1-\exp(-i\mathrm{ad}(Y_0^\kk))\mathrm{Ad}(k^{-1}\sigma)\big)|_{\pp^{\perp}_{\sigma,0}(\gamma)}}.
	\label{eq:AinJ}
\end{equation}
Its square root is denoted by
\begin{equation}
		\left[ \frac{1}{\det 
		(1-\mathrm{Ad}(k^{-1}\sigma))|_{\z^{\perp}_{\sigma,0}(\gamma)}} \cdot\frac{\det \big(1-\exp(-i\mathrm{ad}(Y_0^\kk))\mathrm{Ad}(k^{-1}\sigma)\big)|_{\kk^{\perp}_{\sigma,0}(\gamma)}}{\det \big(1-\exp(-i\mathrm{ad}(Y_0^\kk))\mathrm{Ad}(k^{-1}\sigma)\big)|_{\pp^{\perp}_{\sigma,0}(\gamma)}}     \right]^{1/2}.
	\label{eq:Aroot}
\end{equation}
If $Y_{0}^{\kk}=0$, this square root has the value $1/\det(1-\mathrm{Ad}(k^{-1}\sigma))|_{\pp^{\perp}_{\sigma,0}(\gamma)}$.

The following definition is a direct generalization of 
the function $J_\gamma$ defined by \cite[(5.5.5)]{bismut2011hypoelliptic}, we often call them the
$J$-functions.
\begin{definition}\label{def:3.1.1sss}
Let $J_{\gamma\sigma}(Y^{\kk}_{0})$ be the analytic function of 
$Y_0^{\kk}\in \kk_{\sigma}(\gamma)$ given by 
\begin{equation}\label{Jfunction}
\begin{split}
&J_{\gamma\sigma}(Y_0^{\kk})=\frac{1}{\big|\det 
(1-\mathrm{Ad}(\gamma\sigma))|_{\z^{\perp}_0}\big|^{1/2}} 
\frac{\widehat{A}\big(i\mathrm{ad}(Y_0^\kk)|_{\pp_{\sigma}(\gamma)}\big)}{\widehat{A}\big(i\mathrm{ad}(Y_0^\kk)|_{\kk_{\sigma}(\gamma)}\big)}\cdot\\
&\left[ \frac{1}{\det 
(1-\mathrm{Ad}(k^{-1}\sigma))|_{\z^{\perp}_{\sigma,0}(\gamma)}} 
\frac{\det 
\big(1-\exp(-i\mathrm{ad}(Y_0^\kk))\mathrm{Ad}(k^{-1}\sigma)\big)|_{\kk^{\perp}_{\sigma,0}(\gamma)}}{\det \big(1-\exp(-i\mathrm{ad}(Y_0^\kk))\mathrm{Ad}(k^{-1}\sigma)\big)|_{\pp^{\perp}_{\sigma,0}(\gamma)}}     \right]^{1/2}.
\end{split}
\end{equation}
\end{definition}

By \eqref{Jfunction}, there exist $c_{\gamma\sigma},\;C_{\gamma\sigma} > 0$ such that,
\begin{equation}
\big|J_{\gamma\sigma}(Y_0^\kk)\big|\leq c_{\gamma\sigma} 
\exp\big(C_{\gamma\sigma} |Y_0^\kk|\big).
\label{eq:5.1.10pps}
\end{equation}

\begin{remark}\label{rem:BtJ}
	If $t>0$, if we replace $B$ by $B/t$, the function 
	$J_{\gamma\sigma}$ is unchanged.
\end{remark}

\subsection{A formula for the twisted orbital integrals for the heat kernel}\label{s4.2}

Let $U\g$ be the universal enveloping algebra of $\g$. 
If we identify $\g$ to the vector space of left-invariant vector fields on $G$, 
then the enveloping algebra $U\g$ is identified with 
the algebra of left-invariant differential operators on $G$.

Let $C^\g\in U\g$ be the Casimir element of $G$ 
associated with the bilinear form $B$. 
If $e_1$, $\cdots$, $e_{m+n}$ is a basis of $\g$ and 
if $e^*_1$, $\cdots$, $e^*_{m+n}$ is the dual basis of $\g$ with respect to $B$, then
\begin{equation}\label{eq:3.3.1n}
C^\g=-\sum_{i=1}^{m+n} e^*_i e_i.
\end{equation} 
Also $C^\g$ lies in the center of $U\g$ and commutes 
with $\widetilde{G}$.

The scalar product $\langle\cdot,\cdot\rangle$ of $\g$ is given by 
$-B(\cdot,\theta\cdot)$. If $e_{1}$, $\cdots$, $e_{m}$ is an orthonormal basis of $\pp$, and if $e_{m+1}$, $\cdots$, $e_{m+n}$ is an orthonormal basis of $\kk$, by \eqref{eq:3.3.1n}, we have
\begin{equation}\label{eq:3.3.3gg}
C^\g=-\sum_{i=1}^{m} e^2_i+\sum_{i=m+1}^{m+n} e^2_i.
\end{equation} 

Set
\begin{equation}
    C^{\g,H}=-\sum_{i=1}^{m} e^2_i, \; C^\kk=\sum_{i=m+1}^{m+n} e_i^2.
    \label{eq:3.3.4gg}
\end{equation}
Note that $C^\kk\in U\kk$ is just the Casimir element of $K$ 
associated with $B|_{\kk}$.

By \eqref{eq:3.3.1n} - \eqref{eq:3.3.4gg}, we have
\begin{equation}\label{eq:3.2.5kolntt}
C^\g=C^{\g,H}+C^\kk.
\end{equation}

Let $F=G\times_{K} E$ be a homogeneous vector 
bundle on $X$ defined from a unitary finite-dimensional 
$K^{\sigma}$-representation $(E,\rho^{E})$. 

The operator $C^{\g}$ acts on $C^\infty(X,F)$ via the identification 
\eqref{eq:1.8.4bis}. Let $C^{\g,X}$ denote the action of $C^{\g}$ on 
$C^\infty(X,F)$, which commutes with $G^{\sigma}$. 

Let $\Delta^{H,X}$ denote the Bochner Laplacian acting on 
$C^\infty(X,F)$. Then $C^{\mathfrak{g},H}$ descends to $-\Delta^{H,X}$. Let $C^{\kk,E}\in\mathrm{End}(E)$ denote the 
action of $C^\kk$ on $E$ given by
\begin{equation}
    \label{eq:1.1.9n}
C^{\kk,E}=\sum_{i=m+1}^{m+n} \rho^{E,2}(e_i).
\end{equation}
Then $C^{\kk,E}$ descends to an invariant parallel section $C^{\kk,F}$ of 
$\mathrm{End}(F)$. By \eqref{eq:3.2.5kolntt},
\begin{equation}
C^{\g,X}=-\Delta^{H,X}+C^{\mathfrak{k},F}.
\label{eq:3.6.1pp}
\end{equation}

Let $\kappa^\g\in \Lambda^3(\g^*)$ be such that if $a,b,c\in \g$,
\begin{equation}\label{eq:4.4.2nn}
	\kappa^\g(a,b,c)=B([a,b],c).
\end{equation}
The form $\kappa^\g$ is invariant by the adjoint action of $G\rtimes \Sigma$. We can view $\kappa^\g$ as a closed left and right invariant $3$-form on $G$. 

Let $B^{*}$ be 
the bilinear form on $\Lambda^\bullet(\g^*)$ induced by $B$. 
Let $C^{\kk,\kk}\in\mathrm{End}(\kk)$, 
$C^{\kk,\pp}\in\mathrm{End}(\pp)$ be the actions of $C^{\kk}$ on 
$\kk$, $\pp$ respectively via the adjoint actions of $\kk$ as in 
\eqref{eq:1.1.9n}. By \cite[(2.6.4), (2.6.11)]{bismut2011hypoelliptic}, 
\begin{equation}\label{eq:4.4.5nm}
	B^*(\kappa^\g,\kappa^\g)=\frac{1}{6}
	\sum_{i,j=1}^{m+n} B\left([e_i,e_j],[e^*_i,e^*_j]\right)=\frac{1}{2}\mathrm{Tr^\pp}[C^{\kk,\pp}]
    +\frac{1}{6}\mathrm{Tr^\kk}[C^{\kk,\kk}].
\end{equation}

\begin{definition}\label{def:3.2.1sss}
Let $\mathcal{L}^X$ be the operator acting on $C^\infty(X,F)$,
\begin{equation}\label{ellipticoperator}
\mathcal{L}^X=\frac{1}{2}C^{\g,X}+\frac{1}{8} B^*(\kappa^\g,\kappa^\g).
\end{equation}
Then $\mathcal{L}^X$ commutes with $G^\sigma$.
\end{definition}

Let $A$ be a self-adjoint element of $\mathrm{End}(E)$ which commutes 
with the action of $K^\sigma$. Then $A$ descends to a self-adjoint parallel section of $\mathrm{End}(F)$ which commutes with $G^\sigma$.

\begin{definition}\label{def:LXA}
Let $\mathcal{L}^X_A$ be the operator acting on $C^\infty(X,F)$,
\begin{equation}\label{eq:LXA}
\mathcal{L}^X_A=\mathcal{L}^X+A.
\end{equation}
\end{definition}

It is clear that $\mathcal{L}^X_A$ is a Bochner-like Laplacian. For 
$t>0$, the heat operator $\exp(-t\mathcal{L}^X_A)$ has a smooth kernel $p^X_t(x,x')$ with respect to $dx$ on $X$.

\begin{proposition}\label{prop:kernelinQ}
For $t>0$, $p^X_t\in \mathcal{Q}^\sigma$.
\end{proposition}

\begin{proof}
This follows from \cite[Proposition 4.4.2]{bismut2011hypoelliptic} and from the fact that $\mathcal{L}^X$ commutes with the action of $\sigma$.
\end{proof}

It follows from Subsection \ref{section2-4} and Proposition 
\ref{prop:kernelinQ} that for $t>0$, the twisted orbital integral 
$\mathrm{Tr}^{[\gamma\sigma]}[\exp(-t\mathcal{L}^X_A)]$ is 
well-defined. Recall that $p=\dim \pp_{\sigma}(\gamma)$, $q=\dim 
\kk_{\sigma}(\gamma)$.

\begin{theorem}\label{thm_orbitalintegral}
For any $t>0$, the following identity holds:
\begin{equation}
	\begin{split}
	&\mathrm{Tr}^{[\gamma\sigma]}\big[\exp(-t \mathcal{L}_A^X)\big]=\frac{\exp(-|a|^2/2t)}{(2\pi t)^{p/2}}\\
		&\cdot\int_{\kk_\sigma(\gamma)} J_{\gamma\sigma}(Y^{\kk}_0) 
		\mathrm{Tr}^{E}\big[\rho^E(k^{-1}\sigma)\exp(-i\rho^E(Y^\kk_0)-t A)\big]e^{-|Y^\kk_0|^2/2t} \frac{dY^\kk_0}{(2\pi t)^{q/2}}.
	\end{split}
	\label{eq:4.2.1}
\end{equation}
\end{theorem}
\begin{proof}
	The proof of our theorem will be given in Section 
	\ref{section:proof}.  
\end{proof}

As we explained in Subsection \ref{section2-4}, our twisted orbital 
integral has an expression as an ordinary (un-twisted) orbital integral for a 
larger group $G^{\sigma}$ (cf. \eqref{eq3.35}), whose Lie algebra is 
the semi-direct sum of $\g$ and the Lie algebra of 
$\Sigma^{\sigma}$. One surprising point 
here is that in our formula \eqref{eq:4.2.1}, only the Lie 
subalgebras of $\g$ appears, specially for the cases where $\sigma$ is not 
of finite order. Indeed, in our setting, 
the twist $\sigma$ plays a role of an equivariant action on the vector bundles, 
so that when we apply the local index techniques to prove \eqref{eq:4.2.1}, 
the Lie algebra of $\Sigma^{\sigma}$, as we will see, is not involved 
through the 
computations.

In Section \ref{section5bonn}, we will look into some geometric operators on $X$, 
such as Laplacians for spinors and Hodge Laplacians for flat vector 
bundles. They all can be written as $\mathcal{L}^X_A$ with suitable $A$'s. 
Therefore, we can evaluate the corresponding equivariant heat traces 
via \eqref{eq:tracebonn} and \eqref{eq:4.2.1} for $Z=\Gamma\backslash X$.

\subsection{The twisted orbital integrals of wave operators}\label{section:5.3}
Let $\Delta^{\z_{\sigma}(\gamma)}$ be the standard Laplacian on 
$\z_{\sigma}(\gamma)$ with 
respect to the scalar product on $\z_\sigma(\gamma)$. For $t>0$, let $\exp(t\Delta^{\z_{\sigma}(\gamma)}/2)$ be the 
corresponding heat operator. Let $y$, $Y^\kk_0$ denote the elements in $\pp_{\sigma}(\gamma)$, $\kk_{\sigma}(\gamma)$ respectively. Let $dy dY^\kk_0$ be the Euclidean volume element of $\z_\sigma(\gamma)$, and let $\exp(t\Delta^{\z_{\sigma}(\gamma)}/2)((y,Y^{\kk}_{0}),(y',Y_{0}^{\kk'}))$ denote the smooth kernel of $\exp(t\Delta^{\z_{\sigma}(\gamma)}/2)$ with respect to $dy' dY^{\kk'}_0$.

Let $\delta_{y=a}$ be a distribution on 
$\z_{\sigma}(\gamma)=\pp_{\sigma}(\gamma)\oplus\kk_{\sigma}(\gamma)$ 
associated with the affine subspace $\{y=a\}$. Then 
$J_{\gamma\sigma}(Y^{\kk}_0) 
\rho^E(k^{-1}\sigma)\exp(-i\rho^E(Y^\kk_0))\delta_{y=a}$ is a 
distribution on $\z_{\sigma}(\gamma)$ with values in $\mathrm{End}(E)$. 
Applying the heat operator $\exp(t\Delta^{\z_{\sigma}(\gamma)}/2-t A)$ 
to this distribution, we get a smooth function on 
$\z_{\sigma}(\gamma)$ with values in $\mathrm{End}(E)$. It 
can be evaluated at $0\in\z_{\sigma}(\gamma)$. 
Then Theorem \ref{thm_orbitalintegral} can be rewritten as follows,
\begin{equation}
	\begin{split}
				&\mathrm{Tr}^{[\gamma\sigma]} 
		\big[\exp(-t\mathcal{L}^{X}_{A})\big]=\mathrm{Tr}^{E}\Big[\exp\left(t\Delta^{\z_{\sigma}(\gamma)}/2-tA\right)\\
		&\qquad\qquad\qquad\qquad\left[J_{\gamma\sigma}(Y^{\kk}_0) 
\rho^E(k^{-1}\sigma)\exp\big(-i\rho^E(Y^\kk_0)\big)\delta_{y=a}\right]\Big](0).
	\end{split}
	\label{eq:4.2.3}
\end{equation}

Let $\mathcal{S}(\R)$ be the Schwartz space of $\R$, let 
$\mathcal{S}^{\mathrm{even}}(\R)$ be the space of even functions in 
$\mathcal{S}(\R)$. The Fourier transform of $\mu\in 
\mathcal{S}(\R)$ is given by
\begin{equation}
	\widehat{\mu}(y)=\int_{\R} e^{-2 i \pi y x} \mu(x)dx.
	\label{eq:4.2.4n}
\end{equation}
Take $\mu\in \mathcal{S}^{\mathrm{even}}(\R)$, then $\widehat{\mu}\in 
\mathcal{S}^{\mathrm{even}}(\R)$. We now assume that there exists $C>0$ such that for any $k\in 
\mathbb{N}$, 
there exists $c_{k}>0$ such that
\begin{equation}
	|\widehat{\mu}^{(k)}(y)|\leq c_{k}\exp(-C|y|^{2}).
	\label{eq:4.2.5n}
\end{equation}

Then $\mu\Big(\sqrt{\mathcal{L}^{X}+A}\Big)$ is a self-adjoint operator with 
a smooth kernel. We denote its smooth kernel by 
$\mu\Big(\sqrt{\mathcal{L}^{X}+A}\Big)(x,x')\in \mathrm{Hom}(F_{x'},F_{x})$, 
$x,x'\in X$. As explained in \cite[Section 
6.2]{bismut2011hypoelliptic}, we have
\begin{equation}
\mu\Big(\sqrt{\mathcal{L}^{X}+A}\Big)\in \mathcal{Q}.
\label{eq:5.3.5didot}
\end{equation}
This is a consequence of the finite propagation speed for wave 
operators. We refer to \cite[Section 4.4]{MR618463} for more details.

Since $\sigma$ commutes with $\mathcal{L}^{X}+A$, we can get
$\mu\Big(\sqrt{\mathcal{L}^{X}+A}\Big)\in \mathcal{Q}^{\sigma}$. Then the twisted orbital integral 
$\mathrm{Tr}^{[\gamma\sigma]}\Big[\mu\Big(\sqrt{\mathcal{L}^{X}+A}\Big)\Big]$ is 
well-defined. Similarly, the kernel of 
$\mu\Big(\sqrt{-\Delta^{\z_{\sigma}(\gamma)}/2+A}\Big)$ on 
$\z_{\sigma}(\gamma)$ also has a Gaussian-like decay.

Using Theorem 
\ref{thm_orbitalintegral} and by \eqref{eq:5.1.10pps}, 
\eqref{eq:4.2.3}, \eqref{eq:4.2.5n}, a modification of the 
proof to \cite[Theorem 
6.2.2]{bismut2011hypoelliptic}, using essentially the denseness of 
$y^{2k}e^{-ty^{2}/2}$, $k\in \bN$ in $\mathcal{S}^{\mathrm{even}}(\R)$, shows the following result.

\begin{theorem}\label{thm:general}
	The following identity holds:
	\begin{equation}
		\begin{split}
			\mathrm{Tr}^{[\gamma\sigma]}\Big[\mu\Big(\sqrt{\mathcal{L}^{X}+A}\Big)\Big] = & \mathrm{Tr}^{E}\Big[ \mu \Big(\sqrt{-\Delta^{\z_{\sigma}(\gamma)}/2+A}\Big)J_{\gamma\sigma}(Y^{\kk}_{0}) \\
			&\hspace{2mm} 
			\rho^{E}(k^{-1}\sigma)\exp\big(-i\rho^{E}(Y^{\kk}_{0})\big)\delta_{y=a}\Big](0).
		\end{split}
		\label{eq:4.2.6n}
	\end{equation}
\end{theorem}

Let 
$\mathrm{Tr}^{[\gamma\sigma]}\Big[\cos\Big(s\sqrt{\mathcal{L}^X+A}\Big)\Big]$ be the even distribution on $s\in \R$ such that for any $\mu\in \mathcal{S}^{\mathrm{even}}(\R)$ with $\widehat{\mu}$ having compact support,
\begin{equation}
\label{eq:4.3.7bonn}
\mathrm{Tr}^{[\gamma\sigma]}\Big[\mu\Big(\sqrt{\mathcal{L}^X+A}\Big)\Big]=\int_\R \widehat{\mu}(s)\mathrm{Tr}^{[\gamma\sigma]}\Big[\cos\Big(s\sqrt{\mathcal{L}^X+A}\Big)\Big]ds.
\end{equation}

Let $P^\perp_{\sigma}(\gamma)\subset X$ be the image of 
$\pp^{\perp}_{\sigma}(\gamma)$ by the map $f\rightarrow pe^{f}$. Put
\begin{equation}
	\Delta^{\gamma\sigma}_{X}=\{(x,\gamma\sigma(x))\;:\; x\in 
	P^\perp_{\sigma}(\gamma)\}.
	\label{eq:sousvariete}
\end{equation}
Then $\Delta^{\gamma\sigma}_{X}$ is a submanifold of $X\times X$. We view $\R\times \Delta^{\gamma\sigma}_{X}$ as a distribution on $\R\times X\times X$. 
By 
analyzing the wave front sets for both 
$\cos\Big(s\sqrt{\mathcal{L}^X+A}\Big)$ and $\R\times 
\Delta^{\gamma\sigma}_{X}$ (\cite[Theorem 8.2.10]{MR1996773}, 
\cite[Theorem 23.1.4]{hormander2007analysis}), we get that the 
distribution $\gamma\sigma\cos\Big(s\sqrt{\mathcal{L}^X+A}\Big)(\R\times 
\Delta^{\gamma\sigma}_{X})$ is well-defined on $\R\times X\times X$. 
Using again the finite propagation speed of 
$\cos\Big(s\sqrt{\mathcal{L}^X+A}\Big)$, we see that the push-forward of 
$\mathrm{Tr}^F\Big[\gamma\sigma\cos\Big(s\sqrt{\mathcal{L}^X+A}\Big)\Big](\R\times \Delta^{\gamma\sigma}_{X})$ by the projection $\R\times X\times X\rightarrow \R$ is well-defined, which will be denoted by
\begin{equation}
\label{eq:4.3.13bonn}
\int_{\Delta^{\gamma\sigma}_{X}} \mathrm{Tr}^F\Big[\gamma\sigma 
\cos\Big(s\sqrt{\mathcal{L}^X+A}\Big)\Big].
\end{equation}

By \eqref{orbitaldef1}, \eqref{eq:4.3.7bonn}, \eqref{eq:sousvariete}, we have the identity of even distributions on $\R$,
\begin{equation}
\label{eq:4.3.14bonn}
\mathrm{Tr}^{[\gamma\sigma]}\Big[\cos\Big(s\sqrt{\mathcal{L}^X+A}\Big)\Big]=\int_{\Delta^{\gamma\sigma}_{X}} \mathrm{Tr}^F\Big[\gamma\sigma \cos\Big(s\sqrt{\mathcal{L}^X+A}\Big)\Big].
\end{equation}

The even 
distribution on $\R$,
\begin{equation}
	\mathrm{Tr}^{E}\Big[ \cos 
	\Big(s\sqrt{-\Delta^{\z_{\sigma}(\gamma)}/2+A}\Big)J_{\gamma\sigma}(Y^{\kk}_{0})\rho^{E}(k^{-1}\sigma)\exp\big(-i\rho^{E}(Y^{\kk}_{0})\big)\delta_{y=a}\Big](0)
	\label{eq:4.3.14hh}
\end{equation}
is defined by
\begin{equation}
	\begin{split}
			\mathrm{Tr}^{E}\Big[ \mu 
	\Big(\sqrt{-\Delta^{\z_{\sigma}(\gamma)}/2+A}\Big)J_{\gamma\sigma}(Y^{\kk}_{0})\rho^{E}(k^{-1}\sigma)\exp\big(-i\rho^{E}(Y^{\kk}_{0})\big)\delta_{y=a}\Big](0)&\\
	=\int_{\R}\widehat{\mu}(s)	\mathrm{Tr}^{E}\Big[ \cos 
	\Big(2\pi 
	s\sqrt{-\Delta^{\z_{\sigma}(\gamma)}/2+A}\Big)J_{\gamma\sigma}(Y^{\kk}_{0})&\\
	\rho^{E}(k^{-1}\sigma)\exp\big(-i\rho^{E}(Y^{\kk}_{0})\big)\delta_{y=a}\Big](0).&
	\end{split}
	\label{eq:4.3.15hhs}
\end{equation}

Let 
$(a,\kk_{\sigma}(\gamma))$ denote the affine subspace of 
$\z_{\sigma}(\gamma)=\pp_{\sigma}(\gamma)\oplus 
\kk_{\sigma}(\gamma)$. Set
\begin{equation}
	H^{\gamma}_{\sigma}= \{0\}\times (a,\kk_{\sigma}(\gamma))\subset 
	\z_{\sigma}(\gamma)\times\z_{\sigma}(\gamma).
	\label{eq:Haffine}
\end{equation}
Then we have the tautological 
identification of even distributions on $\R$,
\begin{equation}
	\begin{split}			
	&\mathrm{Tr}^{E}\Big[ \cos 
	\Big(s\sqrt{-\Delta^{\z_{\sigma}(\gamma)}/2+A}\Big)J_{\gamma\sigma}(Y^{\kk}_{0})\rho^{E}(k^{-1}\sigma)\exp\big(-i\rho^{E}(Y^{\kk}_{0})\big)\delta_{y=a}\Big](0)\\
	&=\int_{H_{\sigma}^{\gamma}} \mathrm{Tr}^{E}\Big[ \cos 
	\Big(s\sqrt{-\Delta^{\z_{\sigma}(\gamma)}/2+A}\Big)J_{\gamma\sigma}(Y^{\kk}_{0})\rho^{E}(k^{-1}\sigma)\exp\big(-i\rho^{E}(Y^{\kk}_{0})\big)\Big].
	\end{split}
	\label{eq:4.3.16hhs}
\end{equation}
This is an analogue of \eqref{eq:4.3.14bonn}.

Following the above constructions, we extend 
\cite[Theorem 6.3.2]{bismut2011hypoelliptic} for the twisted orbital 
integrals, where the supports and singular supports of the above 
distributions are obtained as in \cite[Proposition 
6.3.1]{bismut2011hypoelliptic}.
\begin{theorem}\label{thm:waveoperator}
	We have the identity of even distributions on $\mathbb{R}$ 
	supported on $\{s\in\mathbb{R}\,:\, |s|\geq \sqrt{2}|a|\}$ with 
	singular support included in $\pm \sqrt{2}|a|$,
	\begin{equation}
	\begin{split}
		&\int_{\Delta^{\gamma\sigma}_{X}}\mathrm{Tr}^{F}\Big[\gamma\sigma 
		\cos\Big(s\sqrt{\mathcal{L}^{X}+A}\Big)\Big]\\
		&\;\;\;\;=\int_{H^{\gamma}_{\sigma}} 
		\mathrm{Tr}^{E}\Big[\cos\Big(s\sqrt{-\Delta^{\z_{\sigma}(\gamma)}/2+A}\Big)J_{\gamma\sigma}(Y^{\kk}_{0})\rho^{E}(k^{-1}\sigma)\exp\big(-i\rho^{E}(Y^{\kk}_{0})\big)\Big].
	\end{split}
		\label{eq:waveop}
	\end{equation}
\end{theorem}

\subsection{Representation of $K^{\sigma}$ and vanishing of twisted 
orbital integrals}\label{newsection}

In Subsections \ref{section4.1} and \ref{section2-4}, for the 
twisted orbital integral, we always start with a 
$K^{\sigma}$-representation $\rho^{E}$. Now we 
study the irreducible representations of $K^{\sigma}$ and show that 
only $\sigma$-stable irreducible representations of $K$ give the
non-vanishing twisted orbital integrals. Let 
	$\mathrm{Irr}(\cdot)$ denote the set of equivalent classes of 
	irreducible (complex) representations of a compact Lie group.

\begin{proposition}\label{lm:Krepbonn}
If $(E,\rho^E)\in\mathrm{Irr}(K^{\sigma})$ and if the restriction of 
$(E,\rho^E)$ to $K$ is not 
irreducible, then for $k\in K$, we have
\begin{equation}
\label{eq:Krepbonn}
\mathrm{Tr}^E[\rho^E(\sigma)\rho^E(k)]=0.
\end{equation}
\end{proposition}
\begin{proof}
	At first, we assume that $K$ is semisimple. Let $\mathrm{Inn}(K)$ 
	denote the inner automorphism group of $K$. The outer 
	automorphism group of $K$ is
\begin{equation}
\label{eq:5.1.29bonn}
\mathrm{Out}(K)=\mathrm{Aut}(K)/\mathrm{Inn}(K).
\end{equation}
By fixing a maximal torus $T$ of $K$ and an associated positive root 
system $R^+$, $\mathrm{Out}(K)$ can be realized as a finite 
subgroup of $\mathrm{Aut}(K)$ whose elements preserve $T$ and $R^+$ \cite[Chapter VIII, \S 4.4 and Chapter 
IX, \S 4.10]{bourbaki2004lie}. 
Moreover,
\begin{equation}
\label{eq:5.1.30bonn}
\mathrm{Aut}(K)=\mathrm{Inn}(K)\rtimes \mathrm{Out}(K).
\end{equation}

Take $k_0\in K$, $\tau\in \mathrm{Out}(K)$ such that for $k\in K$,
\begin{equation}
\label{eq:5.1.31bonn}
\sigma(k)=k_0\tau(k)k^{-1}_{0}.
\end{equation}
Let $K^\tau$ be the subgroup of $K\rtimes\mathrm{Out}(K)$ generated 
by $K$ and $\tau$.
We claim that there exists $c_\tau\in \bbC$ such that if set
\begin{equation}\label{eq:tauvalue}
\rho^{E,\prime}(\tau)=c_\tau \rho^E(k_0^{-1})\rho^E(\sigma),\; \; 
\rho^{E,\prime}(k)=\rho^E(k),
\end{equation}
then $(E,\rho^{E,\prime})$ is an irreducible representation of $K^\tau$. Note that such number $c_\tau$ is not unique, it depends on the order of $\tau$ and the choice of $k_0$.

Indeed, set
\begin{equation}\label{eq:Ataubonn}
A=\rho^E(k_0^{-1})\rho^E(\sigma)\in\mathrm{End}(E).
\end{equation}
Let $N_0\geq 1$ be the order of $\tau$ in $\mathrm{Out}(K)$.
Set
\begin{equation}\label{eq:5.1.34bonn}
\widehat{k}=k_0\tau(k_0)\cdots\tau^{N_0-1}(k_0)\in K.
\end{equation}
Then 
\begin{equation}
\label{eq:5.1.35bonn}
\sigma(\widehat{k})=\widehat{k}\in K,\; 
\sigma^{N_0}=\mathrm{Ad}(\widehat{k})\in \mathrm{Inn}(K).
\end{equation}

Also we have
\begin{equation}
\label{eq:5.1.36bonn}
A^{N_0}=\rho^{E}(\widehat{k}^{-1})\rho^{E}(\sigma^{N_0}).
\end{equation}
We can verify directly that $A^{N_0}$ commutes with $K^\sigma$. Since 
$(E,\rho^E)$ is irreducible as $K^\sigma$-representation, then 
$A^{N_0}$ is a non-zero scalar endomorphism of $E$, then we take 
$c_\tau\in \bbC^{*}$ such that $c_\tau^{N_0}A^{N_0}=\mathrm{Id}_E$.

We define $\rho^{E,\prime}$ as in \eqref{eq:tauvalue}. Then for $k\in 
K$, 
\begin{equation}\label{eq:5.1.38bonn}
\rho^{E,\prime}(\tau)\rho^{E,\prime}(k)\rho^{E,\prime}(\tau^{-1})=\rho^{E,\prime}(\tau(k)).
\end{equation}
Therefore, $(E,\rho^{E,\prime})$ become an irreducible representation 
of $K^\tau$.

By \eqref{eq:tauvalue}, for any $k\in K$, we have
\begin{equation}
	\mathrm{Tr}^{E}[\rho^{E,\prime}(\tau)\rho^{E,\prime}(k)]=\mathrm{Tr}^{E}[c_{\tau}\rho^{E}(k^{-1}_{0})\rho^{E}(\sigma)\rho^{E}(k)]=c_{\tau}\mathrm{Tr}^{E}[\rho^{E}(\sigma)\rho^{E}(kk^{-1}_{0})].
\end{equation}	
Note that $c_{\tau}\neq 0$. Then for proving \eqref{eq:Krepbonn}, it is equivalent to prove that for all $k\in 
K$, 
\begin{equation}\label{eq:5.1.39bonn}
\mathrm{Tr}^E[\rho^{E,\prime}(\tau)\rho^{E,\prime}(k)]=0.
\end{equation}

In the sequel, we prove \eqref{eq:5.1.39bonn}. Let $P_{++}$ be the dominant weights for the pair $(K,T)$ with 
respect to $R^+$. Then $\tau$ acts on $P_{++}$. If $\lambda\in 
P_{++}$, let $V_{\lambda}\in\mathrm{Irr}(K)$ denote the one with the highest weight $\lambda$.

Now we take a dominant weight $\lambda\in P_{++}$ such that 
$V_\lambda$ embeds into 
$(E,\rho^E)$ as a $K$-subrepresentation. Let 
$\{\tau^{i}(\lambda)\}_{i=0}^{d-1}\subset P_{++}$ be the orbit of 
$\lambda$ under the action of $\tau$.  Note that $d\geq 1$ is the 
length of the orbit and $d\,|\,N_0$. By the description of all the 
irreducible representations of non-connected compact Lie groups in 
\cite[Corollary 4.13.2 and Proposition 4.13.3]{Duistermaat_2000}, 
we get that the representation $(E,\rho^{E,\prime})$ restricting on $K$ is of the form 
\begin{equation}
\label{eq:decompEtau}
\oplus_{i=0}^{d-1} V_{\tau^{i}(\lambda)}.
\end{equation}
Moreover, the action $\rho^{E,\prime}(\tau)$ on $E$ sends the 
component $V_{\tau^{i}(\lambda)}$ to $V_{\tau^{i+1}(\lambda)}$. As a 
consequence, we get \eqref{eq:5.1.39bonn}.

	If $K$ is not semisimple (but always reductive), let $Z_{K}^{0}$ be the 
	identity component of the center of $K$, and let 
	$K_{\mathrm{ss}}$ be the analytic subgroup of $K$ associated with $\kk_{\mathrm{ss}}=[\kk,\kk]$. Then 
	$Z_{K}^0\times K_{\mathrm{ss}}$ is a finite cover of $K$. Note 
	that $Z_{K}^{0}$ is a torus, the 
	action of $\sigma$ on it is of finite order. Then if we 
	proceed as in the above for $K_{\mathrm{ss}}$, we can still apply \cite[Corollary 4.13.2 and Proposition 
	4.13.3]{Duistermaat_2000} to get \eqref{eq:decompEtau} and then 
	\eqref{eq:5.1.39bonn}. This completes the proof of our proposition.
\end{proof}

By our formula in Theorem \ref{thm_orbitalintegral}, the integrand 
contains a term $\mathrm{Tr}^{E}[\rho^{E}(\sigma)\cdots]$. Then we get the following result.
\begin{corollary}
	Let $F$ be the Hermitian vector bundle on $X$ defined from 
	$(E,\rho^E)\in\mathrm{Irr}(K^{\sigma})$ which is not irreducible 
	as $K$-representation. Then for semisimple $\gamma\sigma$ as 
	before, and for $t>0$,
	\begin{equation}
		\mathrm{Tr}^{[\gamma\sigma]}[\exp(-t \mathcal{L}_A^X)]=0.
		\label{eq:3.4.16evian}
	\end{equation}
	
	Moreover, if $\mu\in \mathcal{S}^{\mathrm{even}}(\R)$ is such that 
	\eqref{eq:4.2.5n} holds, then
	\begin{equation}
		\mathrm{Tr}^{[\gamma\sigma]}\Big[\mu\Big(\sqrt{\mathcal{L}^{X}+A}\Big)\Big] =0.
	\end{equation}
\end{corollary}

\begin{remark}
	In \cite[Section 4.13]{Duistermaat_2000}, a Weyl character 
	formula for the non-connected compact Lie group (such as 
	$K^{\tau}$) was established. Then,
	via \eqref{eq:tauvalue}, the 
	trace term $\mathrm{Tr}^{E}[\rho^{E}(\sigma)\cdots]$ in 
	\eqref{eq:4.2.1} can be written in terms of $\lambda$ and the root data 
	associated with $(K,T)$. In Subsection \ref{ss:5.2elliptic}, we use this observation to 
	evaluate the twisted orbital integrals more explicitly in the 
	geometric context. 
\end{remark}

The proof of Proposition \ref{lm:Krepbonn} also gives a 
correspondence between 
$\mathrm{Irr}(K^{\sigma})$ and $\tau$-orbits in $P_{++}$. For simplicity, we assume $K$ to be 
semisimple. Note that $\tau$ generates a finite group 
$\langle\tau\rangle$ in $\mathrm{Out}(K)$. The set 
$\mathrm{Irr}(\langle\tau\rangle)$ can be viewed as a finite abelian group 
($\simeq \Z/{N_{0}}\Z$), it acts on $\mathrm{Irr}(K^{\tau})$ by 
tensor product of representations. By \cite[Corollary 4.13.2 and 
Proposition 4.13.3]{Duistermaat_2000}, we have the canonical 
bijection,
	\begin{equation}
	 \mathrm{Irr}(\langle 
		\tau\rangle)\backslash 
		\mathrm{Irr}(K^{\tau})\simeq  \langle 
		\tau\rangle\backslash P_{++}.
		\label{eq:3.4.14paris}
	\end{equation}

Similarly, 
$\mathrm{Irr}(\Sigma^{\sigma})$ acts on $\mathrm{Irr}(K^{\sigma})$. 
Then the construction given by \eqref{eq:tauvalue} implies an 
injective map
	\begin{equation}
		\mathrm{Irr}(\Sigma^{\sigma})\backslash 
		\mathrm{Irr}(K^{\sigma})\rightarrow \mathrm{Irr}(\langle 
		\tau\rangle)\backslash 
		\mathrm{Irr}(K^{\tau}).
		\label{eq:3.4.15paris}
	\end{equation}
Now we explain that the map in \eqref{eq:3.4.15paris} is also a 
bijection. By \eqref{eq:3.4.14paris}, we consider a $\tau$-orbit in 
$P_{++}$, and let $\lambda$ be one element in this orbit. Let 
$\mathrm{Ind}^{K^{\sigma}}_{K}(V_{\lambda})$ be the induced 
$K^{\sigma}$-representation, and let $(E,\rho^{E})$ be a 
$K^{\sigma}$-irreducible component of it, which contains a 
$V_{\lambda}$-component when restricting to $K$. Then, by the 
arguments as in \eqref{eq:tauvalue} - \eqref{eq:decompEtau}, we get 
the representation $(E,\rho^{E})\in \mathrm{Irr}(K^{\sigma})$ 
corresponds exactly to the $\tau$-orbit of $\lambda$ in $P_{++}$. Therefore, the 
map in \eqref{eq:3.4.15paris} is surjective, then a bijection.

\begin{proposition}
	If $(E,\rho^{E})$ is a finite dimensional unitary 
	$K$-representation, then it can extend to an irreducible 
	representation of $K^{\sigma}$ if and only if the highest weights 
	of its $K$-irreducible components form exactly one $\tau$-orbit 
	in $P_{++}$. Therefore, an irreducible $K^{\sigma}$-representation 
	is also $K$-irreducible if and only if its highest weight is 
	fixed by $\tau$.
\end{proposition}

\subsection{Examples from cyclic base change 
theory}\label{section:basechange}
The twisted orbital integral plays an important role in the cyclic base change 
theory, where $\sigma$ is of finite 
order. The typical examples are as follows,
\begin{itemize}
\item a connected semisimple complex linear Lie group $G_{\bbC}$, where $\sigma$ is taken to be 
the conjugation of a matrix and its fixed point set is just the the real 
matrix subgroup $G_{\R}$;
\item the product case where $G=G^{\ell}_{0}$ is given as $\ell$-copies of a 
connected real semisimple Lie group $G_{0}$ and $\sigma$ is given as 
the cyclic permutation of the copies. The simplest case is $\ell=2$.
\end{itemize}

In this subsection, we focus on such examples and explain how 
to make use of our formula in Theorem \ref{thm_orbitalintegral}. In 
particular, we show via elementary computations how the twisted orbital 
integrals (for $G_{\bbC}$ or $G^{\ell}_{0}$) relate to the ordinary 
orbital integrals (for $G_{\R}$ or $G_{0}$). Note that we have no 
any regularity condition on the semisimple element $\gamma\sigma$.

\begin{example}[Complex semisimple Lie group and matrix conjugation]
	Let $N\in\bN$ be large integer. Let $G=G_{\bbC}\subset \mathrm{GL}(N,\bbC)$ be 
	a connected and simply connected semisimple linear algebraic group which is invariant 
	under transpose and conjugation. Then the Cartan involution is 
	given as $\theta(A)=(\bar{A}^{T})^{-1}$, where $(\cdot)^{T}$ 
	denotes the matrix transpose. 
	
	We also view $G_{\bbC}$ as a real semisimple Lie group with (real) Lie 
	algebra $\g$, the 
	bilinear form $B$ on $\g$ is taken to be the real trace form, 
	which is equal to the real part of the complex Killing form on 
	$\g_{\bbC}$ up a positive multiple. Let 
	$\sigma\in\mathrm{Aut}(G_{\bbC})$ be such that $\sigma(A)=\bar{A}$. 
	Then its fixed points are exactly the real points of 
	$G_{\bbC}$, denoted by $G_{\R}$, the subgroup of real 
	matrices in $G_{\bbC}$. 
	Alternatively speaking, $G_{\R}$ is a split real form of 
	$G_{\bbC}$, and $G_{\bbC}$ is the complexification 
	of $G_{\R}$. We will put $X_{\bbC}=G_{\bbC}/K$, 
	$X_{\R}=G_{\R}/K_{\R}$.
	
	Let $\g_{\R}=\pp_{\R}\oplus\kk_{\R}$ denote the Cartan 
	decomposition of Lie algebra of $G_{\R}$, and let $K_{\R}\subset 
	G_{\R}$ be the maximal compact subgroup corresponding $\kk_{\R}$.
	Then
	\begin{equation}
		\g=\g_{\R}\oplus i\g_{\R},
	\end{equation}
	and the Cartan decomposition is given by
	\begin{equation}
		\g=\pp\oplus\kk,\; \pp=\pp_{\R}\oplus i\kk_{\R},\; 
		\kk=\kk_{\R}\oplus i\pp_{\R}.
	\end{equation}
	
	Then the maximal compact subgroup $K$ (with Lie algebra $\kk$) of 
	$G_{\bbC}$ is just the compact real form of $G_{\bbC}$, and also the 
	compact form of $G_{\R}$. Moreover, it 
	is simply connected. For $\gamma\in G_{\bbC}$, if $\gamma\sigma$ is 
	semisimple if and only if $\gamma\sigma(\gamma)\in G_{\bbC}$ is 
	semisimple (cf. \cite[Lemme 2.2]{Clozel1982}). Here, we consider 
	the elliptic element $\sigma$ itself, for which the associated twisted 
	orbital integral $\mathrm{Tr}^{[\sigma]}[\cdot]$ has been studied 
vastly (cf. \cite[\S 
8]{CM_1991__80_2_197_0}, \cite{CM_1992__81_3_261_0}, 
\cite{BeLip2017}, etc).
	In previous notation, we have 
	$Z_{\sigma}(1)=G_{\R}$, $K_{\sigma}(1)=K_{\R}$. Set $p=\dim 
	\pp_{\R}$, $q=\dim \kk_{\R}$.
	
	We consider the following representations of $G_{\bbC}$. Let $(E_{0},\rho_{0})$ be a finite 
	dimensional holomorphic representation of $G_{\bbC}$, the unitary 
	trick implies that the restriction of $\rho_{0}$ to $K$ 
	or $G_{\R}$ determines uniquely $\rho_{0}$. Let 
	$(E_{0}^{\sigma}:=E_{0},\rho^{\sigma}_{0})$ be the representation of 
	$G_{\bbC}$ twisted by $\sigma$, i.e., 
	$\rho^{\sigma}_{0}(g)=\rho_{0}(\sigma(g))$, $g\in G_{\bbC}$. For 
	$v_{1},v_{2}\in E_{0}$, set
	\begin{equation}
		\rho^{E}(\sigma)(v_{1}\otimes v_{2})=v_{2}\otimes v_{1}\in 
		E_{0}\otimes E^{\sigma}_{0}.
	\end{equation}
	This way, we get a representation $(E,\rho^{E}):=(E_{0}\otimes 
	E^{\sigma}_{0},\rho_{0}\otimes \rho^{\sigma}_{0})$ of 
	$(G_{\bbC})^{\sigma}=G_{\bbC}\rtimes \{1,\sigma\}$. Taking a 
	$K$-invariant Hermitian metric on $E_{0}$, we make $(E,\rho^{E})$ as a unitary representation of $K^{\sigma}=K\rtimes \{1,\sigma\}$. We 
	consider the Laplacian $\mathcal{L}^{X_{\bbC},F}$ acting on 
	$C^{\infty}(X_{\bbC},F=G_{\bbC}\times_{K}E)$ defined in \eqref{ellipticoperator}.
	
	For $Y\in \kk_{\R}$, the $J$-function $J^{G_{\R}}_{1}$ for the 
	identity element $1\in G_{\R}$ is 
	\begin{equation}
		J^{G_{\R}}_{1}(Y)=\frac{\widehat{A}\big(i\mathrm{ad}(Y)|_{\pp_{\R}}\big)}{\widehat{A}\big(i\mathrm{ad}(Y)|_{\kk_{\R}}\big)}.
	\end{equation}
	Note that one should not confuse the imaginary unit $i$ appearing in 
	the $J$-functions with the one in the Lie algebra $\g$.
	
	An elementary computation shows that as a function in $Y\in 
	\kk_{\R}$, 
	\begin{equation}
		\widehat{A}\big(i\mathrm{ad}(Y)|_{\pp_{\R}}\big)\left[\frac{1}{\det(1+e^{-i\mathrm{ad}(Y)})|_{\pp_{\R}}}\right]^{1/2}=\frac{1}{2^{p/2}}\widehat{A}\big(i\mathrm{ad}(2Y)|_{\pp_{\R}}\big).
		\label{eq:3.4.5tt}
	\end{equation}
	Similar for $\widehat{A}\big(i\mathrm{ad}(Y)|_{\kk_{\R}}\big)$. The twist 
	$\sigma$ acts on $i\pp_{\R}\oplus i\kk_{\R}$ as $-1$. Let $J_{\sigma}$ be the $J$-function associated with $\sigma$ and 
	$G_{\bbC}$, then for $Y\in \kk_{\R}=\kk_{\sigma}(1)$,
	\begin{equation}
		J_{\sigma}(Y)=\frac{1}{2^{p}} 
		J^{G_{\R}}_{1}(2Y)\left[\frac{\det(1+e^{-i\mathrm{ad}(Y)})|_{\pp_{\R}}}{\det(1+e^{-i\mathrm{ad}(Y)})|_{\kk_{\R}}}\right].
	\end{equation}
	
	For the trace of $\rho^{E}$, we have, for $Y\in \kk_{\R}$,
	\begin{equation}
		\mathrm{Tr}^{E}\big[\rho^{E}(\sigma)\exp(-i\rho^{E}(Y))\big]=\mathrm{Tr}^{E_{0}}\big[\exp(-i\rho^{E_{0}}(2Y))\big].
	\end{equation}

	By \eqref{eq:4.2.1} in our theorem, we have, for $t>0$,
	\begin{equation}
		\begin{split}
				\mathrm{Tr}^{[\sigma]}[\exp(-t\mathcal{L}^{X_{\bbC},F})]
				=\frac{1}{(8\pi 
				t)^{p/2}}\int_{\kk_{\R}}J^{G_{\R}}_{1}(2Y)\left[\frac{\det(1+e^{-i\mathrm{ad}(Y)})|_{\pp_{\R}}}{\det(1+e^{-i\mathrm{ad}(Y)})|_{\kk_{\R}}}\right]&\\
				\cdot\mathrm{Tr}^{E_{0}}\big[e^{-i\rho^{E_{0}}(2Y)}\big]e^{-\frac{|Y|^{2}}{2t}}\frac{dY}{(2\pi t)^{q/2}}.&
		\end{split}
		\label{eq:3.4.6new}
	\end{equation}
	
	As we will see in Section \ref{section5bonn}, after twisting 
	$(E,\rho^{E})$ with the graded (virtual) 
	$K^{\sigma}$-representations $\rho^{\Lambda^\bullet(\pp^{*})}$ on
	$\sum_{j}(-1)^{j}\Lambda^{j}(\pp^{*})$ or 
	$\sum_{j}(-1)^{j}j\Lambda^{j}(\pp^{*})$, the denominator 
	$\det(1+e^{-i\mathrm{ad}(Y)})|_{\kk_{\R}}$ can be canceled out 
	properly. Such constructions, in geometric setting, appear in 
	the evaluations of Lefschetz numbers or equivariant real analytic 
	torsions. We will use $\mathrm{Tr_{s}}^{[\bullet]}[\cdots]$ with 
	the subscript $\mathrm{s}$ to
	denote the (twisted) orbital integrals which take the supertrace 
	of the endomorphisms of the $\Z_{2}$-graded vector bundles.

	If $\kt_{\R}$ is a Cartan subalgebra of $\kk_{\R}$, put 
	$$\kb_{\R}=\{f\in\pp_{\R}\;:\; [f,v]=0, \;\mathrm{for\;all\;} v\in\kt_{\R}\}.$$ 
	Then $\kt_{\R}\oplus\kb_{\R}$ is Cartan subalgebra of $\g_{\R}$ 
	(\cite[pp.129]{KnappRep1986}), 
	and the fundamental rank $\delta(G_{\R})$ is defined as 
	$\dim_{\R}\kb_{\R}$. Let $N^{\Lambda^{\bullet}(\pp^{*})}$ denote the number operator on 
	$\Lambda^{\bullet}(\pp^{*})$ which 
	acts on $\Lambda^{j}(\pp^{*})$ as multiplication by $j$. Then we have the following identities for 
	$Y\in\kk_{\R}$,
	\begin{equation}
		\begin{split}
\left[\frac{\det(1+e^{-i\mathrm{ad}(Y)})|_{\pp_{\R}}}{\det(1+e^{-i\mathrm{ad}(Y)})|_{\kk_{\R}}}\right]&\mathrm{Tr_{s}}^{\Lambda^{\bullet}(\pp^{*})}\big[\rho^{\Lambda^\bullet(\pp^{*})}(\sigma)e^{-i\rho^{\Lambda^\bullet(\pp^{*})}(Y)}\big] \\
&= \begin{cases}
\mathrm{Tr_{s}}^{\Lambda^{\bullet}(\pp_{\R}^{*})}\big[e^{-i\rho^{\Lambda^\bullet(\pp_{\R}^{*})}(2Y)}\big] &\text{if $\delta(G_{\R})=0$};\\
0 &\text{if $\delta(G_{\R})\geq 1$},
\end{cases}
\end{split}
\label{eq:3.4.9volvic}
\end{equation}
and
\begin{equation}
\begin{split}
\left[\frac{\det(1+e^{-i\mathrm{ad}(Y)})|_{\pp_{\R}}}{\det(1+e^{-i\mathrm{ad}(Y)})|_{\kk_{\R}}}\right]&\mathrm{Tr_{s}}^{\Lambda^{\bullet}(\pp^{*})}\big[N^{\Lambda^{\bullet}(\pp^{*})}\rho^{\Lambda^\bullet(\pp^{*})}(\sigma)e^{-i\rho^{\Lambda^\bullet(\pp^{*})}(Y)}\big] \\
& = \begin{cases}
\big(\dfrac{p+q}{2}\big)\mathrm{Tr_{s}}^{\Lambda^{\bullet}(\pp_{\R}^{*})}\big[e^{-i\rho^{\Lambda^\bullet(\pp_{\R}^{*})}(2Y)}\big] &\text{if $\delta(G_{\R})=0$};\\
2\mathrm{Tr_{s}}^{\Lambda^{\bullet}(\pp_{\R}^{*})}\big[N^{\Lambda^{\bullet}(\pp_{\R}^{*})}e^{-i\rho^{\Lambda^\bullet(\pp_{\R}^{*})}(2Y)}\big] &\text{if $\delta(G_{\R})=1$};\\
0 &\text{if $\delta(G_{\R})\geq 2$}.
\end{cases}		
\end{split}
\label{eq:3.4.9paris}
\end{equation}
We briefly explain how to obtain the above identities. Note that if $g$ is an 
isometry of a finite dimensional Euclidean space $V$, then 
\begin{equation}
	\begin{split}
		&\mathrm{Tr_{s}}^{\Lambda^{\bullet}(V^{*})}[g]=\det(1-g^{-1})|_{V},\\
		&\mathrm{Tr_{s}}^{\Lambda^{\bullet}(V^{*})}\big[N^{\Lambda^{\bullet}(V^{*})}g\big]=\frac{\partial}{\partial s}|_{s=0} \det(1-g^{-1}e^{s})|_{V}.
	\end{split}
	\label{eq:4.5.10volvic}
\end{equation}
Moreover, if $V$ is even-dimensional and $g$ preserves the 
orientation, or if $V$ is odd-dimensional and $g$ reverses the 
orientation, then
\begin{equation}
	\mathrm{Tr_{s}}^{\Lambda^{\bullet}(V^{*})}\Big[\big(N^{\Lambda^{\bullet}(V^{*})}-\frac{\dim V}{2}\big)g\Big]=0.
	\label{eq:4.5.11volvic}
\end{equation}
Due to 
the invariance by adjoint action of $K_{\R}$, we only need to prove 
\eqref{eq:3.4.9volvic}, \eqref{eq:3.4.9paris}
for $Y\in \kt_{\R}$. In this case, $\mathrm{ad}(Y)$ acts $\kb_{\R}$ 
as zero. Note that $\pp=\pp_{\R}\oplus i\kk_{\R}$, then the first part of \eqref{eq:3.4.9volvic} follows directly 
from the first identity in \eqref{eq:4.5.10volvic}. Using further 
\eqref{eq:4.5.11volvic}, we get the first case ($\delta(G_{\R})=0$) 
in \eqref{eq:3.4.9paris}. The case where $\delta(G_{\R})=\dim_{\R} \kb_{\R}\geq 
2$ follows from the second identity in \eqref{eq:4.5.10volvic}. 
Finally, when $\delta(G_{\R})=1$, $\kb_{\R}$ is a real line, then, 
by taking the orthogonal splitting 
$\pp_{\R}=\kb_{\R}\oplus\kb_{\R}^{\perp}$, the corresponding result in 
\eqref{eq:3.4.9paris} follows from 
\begin{equation}
	\mathrm{Tr_{s}}^{\Lambda^{\bullet}(\pp_{\R}^{*})}\big[N^{\Lambda^{\bullet}(\pp_{\R}^{*})}e^{-i\rho^{\Lambda^\bullet(\pp_{\R}^{*})}(Y)}\big]=-\det(1-e^{-i\mathrm{ad}(Y)})|_{\kb_{\R}^{\perp}}.
\end{equation}

As a consequence, if $\delta(G_{\R})=0$,
\begin{equation}
\mathrm{Tr_{s}}^{[\sigma]}\big[\exp(-t\mathcal{L}^{X_{\bbC},\Lambda^{\bullet}(T^{*}X_{\bbC})\otimes F})\big]=\mathrm{Tr_{s}}^{[1]}\big[\exp(-4t\mathcal{L}^{X_{\R},\Lambda^{\bullet}(T^{*}X_{\R})\otimes F_{0}})\big],
	\label{eq:3.4.10paris}
\end{equation}
and if $\delta(G_{\R})=1$,
\begin{equation}
	\begin{split}
		&\mathrm{Tr_{s}}^{[\sigma]}\big[N^{\Lambda^{\bullet}(T^{*}X_{\bbC})}\exp(-t\mathcal{L}^{X_{\bbC},\Lambda^{\bullet}(T^{*}X_{\bbC})\otimes F})\big]\\
		&\qquad\qquad=2\mathrm{Tr_{s}}^{[1]}\big[N^{\Lambda^{\bullet}(T^{*}X_{\R})}\exp(-4t\mathcal{L}^{X_{\R},\Lambda^{\bullet}(T^{*}X_{\R})\otimes F_{0}})\big].
	\end{split}
	\label{eq:3.4.11paris}
\end{equation}
The other identities in \eqref{eq:3.4.9paris} will imply the vanishing of 
Lefschetz numbers or equivariant analytic torsions, we refer to 
Section
\ref{section5bonn} and also \cite[Theorem 3.3.2]{LIU2021109117} for such results.
\end{example}

\begin{example}[Product case]
	Let $(G_{0},K_{0},\theta_{0},B_{0})$ be a connected real 
	reductive Lie group. Put
	\begin{equation}
		(G,K,\theta)=(G_{0},K_{0},\theta_{0})\times 
		(G_{0},K_{0},\theta_{0}).
		\label{eq:3.4.100s}
	\end{equation}
	Let $\g_{0}=\kk_{G_{0}}\oplus \pp_{G_{0}}$ denote the Cartan 
	decomposition of the Lie algebra of 
	$G_{0}$. Then 
	\begin{equation}
		\g=\g_{0}\oplus \g_{0},\; \kk=\kk_{G_{0}}\oplus\kk_{G_{0}}, 
		\pp=\pp_{G_{0}}\oplus\pp_{G_{0}}.
		\label{eq:3.4.101s}
	\end{equation}
	We define the bilinear form on $\g$ by
	\begin{equation}
		B=B_{0}\oplus B_{0}.
		\label{eq:3.4.102s}
	\end{equation}
	The symmetric space $X$ is identify with $X_{0}\times X_{0}$, 
	where $X_{0}=G_{0}/K_{0}$. 
	
	The twist $\sigma$ is defined as follows, for $(g_{1},g_{2})\in 
	G=G_{0}\times G_{0}$,
	\begin{equation}
		\sigma(g_{1},g_{2})=(g_{2},g_{1}).
		\label{eq:3.4.103s}
	\end{equation}
	The fixed point set of $\sigma$, i.e. the $\sigma$-twisted 
	centralizer $Z_{\sigma}(1)$ of $1\in G$, is exactly the diagonal 
	of the product $G_{0}\times G_{0}$. Then $Z_{\sigma}(1)\simeq G_{0}$ 
	canonically, and the induced Cartan involution on $Z_{\sigma}(1)$ 
	from $\theta$ is just $\theta_{0}$. By \eqref{eq:3.4.102s}, the bilinear form $B$ restricting to $\z_{\sigma}(1)\simeq \g_{0}$ coincides with $2B_{0}$.

	Let $(E_{0},\rho_{0})$ be a unitary representation of $K_{0}$. 
	Set $(E,\rho^{E})=(E_{0},\rho_{0})\otimes (E_{0},\rho_{0})$, a 
	unitary representation of $K$. For $v_{1}, v_{2}\in E_{0}$, set 
	$\rho^{E}(\sigma)(v_{1}\otimes v_{2})=v_{2}\otimes v_{1}$. Then 
	$(E,\rho^{E})$ extends as a representation of $K^{\sigma}$. We 
	define the vector bundles $F$, $F_{0}$ on $X$, $X_{0}$ 
	respectively.  Let 
	$\mathcal{L}^{X,F}$, $\mathcal{L}^{X_{0},F_{0}}$ denote the operators as in \eqref{ellipticoperator} 
	acting on  
	$C^{\infty}(X,F)$, $C^{\infty}(X_{0}, F_{0})$ respectively. In 
	particular, we have
	\begin{equation}
		\mathcal{L}^{X,F}=\mathcal{L}^{X_{0},F_{0}}\otimes 1+1\otimes 
		\mathcal{L}^{X_{0},F_{0}}.
		\label{eq:3.4.105s}
	\end{equation}
	
	In \cite[\S 8]{Langlands1988base}, under the above setting, Langlands deduced an identity 
	between the 
	$\sigma$-twisted orbitals integrals and the ordinary orbital 
	integrals, where the matching functions are given via 
	convolution. We specialize his result in our simple example 
	here. For $\gamma_1,\gamma_{2}\in G_{0}$, take 
	$\gamma=(\gamma_1, \gamma_{2})\in G$ such that $\gamma\sigma$ 
	semisimple. We may assume that 
	\begin{equation}
		\gamma_{1}=e^{a_{1}}k_{1}^{-1},\; 
		\gamma_{2}=e^{a_{2}}k_{2}^{-1}\; a_{1},a_{2}\in\pp_{G_{0}}, 
		k_{1},k_{2}\in K_{0},
	\end{equation}
	and by Theorem \ref{thm_keythm}, $\mathrm{Ad}(k_{1}^{-1})a_{2}=a_{1}, 
	\mathrm{Ad}(k_{2}^{-1})a_{1}=a_{2}$. 
	The norm of $\gamma$ is defined as 
	$N\gamma=\gamma_{1}\gamma_{2}\in G_{0}$, which has the form
	\begin{equation}
		\gamma_{1}\gamma_{2}=e^{a}k^{-1}, a=2a_{1}, k=k_{2}k_{1}, 
		\mathrm{Ad}(k)a=a.
		\label{eq:3.4.108s}
	\end{equation}
	Then $\gamma_{1}\gamma_{2}$ is a semisimple element in 
	$G_{0}$.
	
	Let $Z_{0}(N\gamma)$ be the centralizer of $N\gamma$ in $G_{0}$ 
	with Lie algebra $\z_{G_0}(N\gamma)=\pp_{G_{0}}(N\gamma)\oplus 
	\kk_{G_{0}}(N\gamma)\subset \g_{0}$. Then by \eqref{prop_centralizer}, we have
	\begin{equation}
		Z_{\sigma}(\gamma)=\{(g,k_{1}gk_{1}^{-1})\in G\;:\; g\in 
		Z_{0}(N\gamma)\}\simeq Z_{0}(N\gamma).
	\end{equation}
	The diffeomorphism $(g_{1},g_{2})\in 
	G\mapsto (g_{1},g_{2}^{-1}\gamma_{2}g_{1})\in G$ induces an 
	identification 
	\begin{equation}
		Z_{\sigma}(\gamma)\backslash G\simeq 
		(Z_{0}(N\gamma)\backslash G_{0})\times G_{0}.
		\label{eq:3.4.109s}
	\end{equation}
	The result in \cite[\S 8]{Langlands1988base}, as a consequence of 
	\eqref{eq:3.4.109s}, says that for 
	$t>0$,
	\begin{equation}
\mathrm{Tr}^{[\gamma\sigma]}[\exp(-t 
\mathcal{L}^{X,F})]=\frac{1}{2^{p/2}}\mathrm{Tr}^{[N\gamma]}[\exp(-2t 
\mathcal{L}^{X_{0},F_{0}})],
\label{eq:3.4.106s}
	\end{equation}
	where the right-hand side is the ordinary orbital integrals for 
	$(G_{0},B_{0})$, and the factor $2^{p/2}$ comes from the volume 
	conventions with 
	$B|\z_{\sigma}(\gamma)=2B_{0}|_{\z_{G_{0}}(N\gamma)}$.
	
	Now we explain how our formula \eqref{eq:4.2.1} is compatible 
	with \eqref{eq:3.4.106s}. We start with the $J$-function 
	$J_{\gamma\sigma}$. Set $p=\dim\pp_{G_{0}}(N\gamma), 
	q=\dim\kk_{G_{0}}(N\gamma)$. For $Y\in \kk_{G_{0}}(N\gamma)$, 
	$(Y,\mathrm{Ad}(k_{1})Y)\in \kk_{\sigma}(\gamma)$. In this case, 
	$\z_{0}=\z((a_{1},a_{2}))=\z_{G_{0},0}\oplus\mathrm{Ad}(k_{1})\z_{G_{0},0}$, where $\z_{G_{0},0}=\z_{G_{0}}(a)\subset \g_{0}$. Then we have
	\begin{equation}
		\kk^{\perp}_{\sigma,0}(\gamma)\simeq 
		\big(\kk^{\perp}_{G_{0},0}(N\gamma),\mathrm{Ad}(k_{1})\kk^{\perp}_{G_{0},0}(N\gamma)\big) \oplus \kk_{G_{0}}(N\gamma).
		\label{eq:3.4.105}
	\end{equation}
	As a consequence, we get
	\begin{equation}
		\begin{split}
			&\det\big(1-\exp(-i\mathrm{ad}(Y,\mathrm{Ad}(k_{1})Y))\mathrm{Ad}((k_{1},k_{2})^{-1}\sigma)\big)|_{\kk^{\perp}_{\sigma,0}(\gamma)}\\
			&=\det 
			\big(1-\exp(-i\mathrm{ad}(2Y))\mathrm{Ad}(k^{-1})\big)|_{\kk^{\perp}_{G_{0},0}(N\gamma)}\\
			&\quad\cdot\det\big(1+\exp(-i\mathrm{ad}(Y))\big)|_{\kk_{G_{0}}(N\gamma)}.
		\end{split}
		\label{eq:3.4.110s}
	\end{equation}
	Similar computations hold for $\pp^{\perp}_{\sigma,0}(\gamma)$ 
	and $\z^{\perp}_{\sigma,0}(\gamma)$. Then by \eqref{Jfunction} 
	and \eqref{eq:3.4.5tt},
	\begin{equation}
		J_{\gamma\sigma}(Y,\mathrm{Ad}(k_{1})Y)=\frac{1}{2^{p}} 
		J^{G_{0}}_{N\gamma}(2Y),
	\end{equation}
	where $J^{G_{0}}_{N\gamma}$ is the corresponding $J$-function defined 
	with $N\gamma$ and $G_{0}$.
	
Moreover, a direct computation shows,
	\begin{equation}
		\begin{split}
					&\mathrm{Tr}^{E}\left[\rho^{E}((k_{1},k_{2})^{-1}\sigma)\exp\left(-i\rho^{E}(Y,\mathrm{Ad}(k_{1})Y)\right)\right]\\
					&=\mathrm{Tr}^{E_{0}}\left[\rho_{0}(k^{-1})\exp\left(-i2\rho_{0}(Y)\right)\right].
		\end{split}
	\end{equation}
Note that $|(Y,\mathrm{Ad}(k_{1})Y)|^{2}_{B}=2|Y|^{2}_{B_{0}}$, where 
the subscripts indicate the corresponding norms. Then by 
\eqref{eq:4.2.1}, we get 
\begin{equation}
	\begin{split}
	&\mathrm{Tr}^{[\gamma\sigma]}[\exp(-t 
	\mathcal{L}^{X,F})]=\frac{1}{2^{p/2}}\frac{\exp(-|a|_{B_{0}}^2/4t)}{(4\pi t)^{p/2}}\\
		&\cdot\int_{\kk_{G_{0}}(N\gamma)} J^{G_{0}}_{N\gamma}(2Y) 
		\mathrm{Tr}^{E_{0}}\left[\rho_{0}(k^{-1})\exp(-i\rho_{0}(2Y))\right]e^{-2|Y|_{B_{0}}^2/2t} \frac{2^{q}|dY|_{B_{0}}}{(4\pi t)^{q/2}}.
	\end{split}
	\label{eq:4.2.1new}
\end{equation}
After the coordinate change $2Y\rightarrow Y$ in the above integral, 
we get exactly the quantity
$\dfrac{1}{2^{p/2}}\mathrm{Tr}^{[N\gamma]}[\exp(-2t 
\mathcal{L}^{X_{0},F_{0}})]$.

One can consider generally 
$\ell\geq 2$ copies of $G_{0}$ with cyclic permutation $\sigma$, the 
above computations are still applicable with suitable change. Using 
the formula \eqref{eq:waveop} for wave operators, one can also verify the 
identity \eqref{eq:3.4.106s} for a general class of integral kernel functions.
\end{example}

\section{The hypoelliptic Laplacian on 
$X$}\label{ch3}
The purpose of this section is to recall the construction of the hypoelliptic 
Laplacian of Bismut \cite[Chapter 
2]{bismut2011hypoelliptic}.

\subsection{Clifford algebras}\label{section3.1}
Let $V$ be a real vector space of dimension $m$ equipped with a 
real-valued nondegenerate symmetric bilinear form $B$. The Clifford 
algebra $c(V)$ of $V$ with respect to $B$ is the algebra generated by 
$1$ and $a\in V$ and the relations, 
\begin{equation}
	ab+ba=-2B(a,b),\; a,b\in V.
	\label{eq:clifford}
\end{equation}
We will denote by $\widehat{c}(V)$ the Clifford algebra of $V$ associated 
with $-B$. Also they are 
$\Z_{2}$-graded algebras, we write
\begin{equation}
	c(V)=c_{+}(V)\oplus c_{-}(V),\;\; 
	\widehat{c}(V)=\widehat{c}_{+}(V)\oplus \widehat{c}_{-}(V).
	\label{eq:mp}
\end{equation}

Since $B$ is nondegenerate, it induces an 
isomorphism $\varphi$ between $V$ and $V^{*}$ such that if $a, b\in V$, 
then
\begin{equation}
	\langle \varphi(a), b\rangle=B(a,b).
	\label{eq:varphi}
\end{equation}
Let $B^{*}$ be the corresponding bilinear form on $V^{*}$, which also extends to a nondegenerate symmetric bilinear form on 
$\Lambda^\bullet(V^{*})$.

If $\alpha\in V^{*}$, $a\in V$, let $\alpha\,\wedge$ denote 
the exterior product of $\alpha$ acting on 
$\Lambda^{\bullet}(V^{*})$, and let $i_{a}$ denote the interior 
product (or the contraction) of $a$ acting on 
$\Lambda^{\bullet}(V^{*})$.
If $a\in V$, let $c(a)$, $\widehat{c}(a)\in 
\mathrm{End}(\Lambda^{\bullet}(V^{*}))$ be given by
\begin{equation}
	c(a)=\varphi(a)\wedge - \;i_{a},\; 
	\widehat{c}(a)=\varphi(a)\wedge + \;i_{a}.
	\label{eq:actionscc}
\end{equation}
Then $c(a)$, $\widehat{c}(a)$ are odd operators, which are 
respectively antisymmetric, symmetric with respect to $B^{*}$. If $a, b\in V$, then
\begin{equation}
	[c(a), c(b)]=-2B(a,b),\; [\widehat{c}(a), 
	\widehat{c}(b)]=2B(a,b),\; [c(a), \widehat{c}(b)]=0.
	\label{eq:bracketrelations}
\end{equation}
By \eqref{eq:bracketrelations}, $\Lambda^\bullet(V^{*})$ is a 
$c(V)\widehat{\otimes}\widehat{c}(V)$-module. If $D\in c(V)$ or $\widehat{c}(V)$, 
then we denote by $c(D)$ or $\widehat{c}(D)$ the corresponding actions 
on $\Lambda^\bullet(V^{*})$ defined by \eqref{eq:actionscc}.

Let $e_{1}$, $\cdots$, $e_{m}$ be a basis of $V$, and let $e_{1}^*$, 
$\cdots$, $e_{m}^{*}$ be the dual basis of $V$ with respect to $B$, 
so that $B(e_{i}, e^{*}_{j})=\delta_{ij}$. 
Let $e^{1}$, $\cdots$, $e^{m}$ be the basis of $V^{*}$ which is dual 
to the basis $e_{1}$, $\cdots$, $e_{m}$. Then 
$e^{i}=\varphi(e^{*}_{i})$.

Note that $1\in \R=\Lambda^{0}(V^{*})$. The symbol map 
$\bm{\sigma}:D\in\widehat{c}(V)\mapsto \widehat{c}(D)\cdot 
1\in\Lambda^\bullet(V^{*})$ is an isomorphism 
of $\Z_{2}$-graded vector spaces. If $\alpha\in \Lambda^{p}(V^{*})$, then the inverse map of 
$\bm{\sigma}$ is given by
\begin{equation}
	\widehat{c}(\alpha)= \frac{1}{p!}\sum_{1\leq i_{1},\cdots,i_{p}\leq 
	m}\alpha(e^{*}_{i_{1}},\cdots, e^{*}_{i_{p}}) \widehat{c}(e_{i_{1}})\cdots 
	\widehat{c}(e_{i_{p}})\in \widehat{c}(V).
	\label{eq:calpha}
\end{equation}
If $A\in\mathrm{End}(V)$ is antisymmetric with respect to $B$, set
\begin{equation}
	\widehat{c}(A)=-\frac{1}{4}\sum_{i,j} B(e^{*}_{i}, 
	e^{*}_{j})\widehat{c}(e_{i})\widehat{c}(e_{j}).
	\label{eq:chatA}
\end{equation}

\begin{definition}
	The number operator $N^{\Lambda^\bullet(V^*)}$ on $\Lambda^\bullet(V^*)$ is such that, if $\alpha\in \Lambda^p(V^*)$, then
	\begin{equation}
	N^{\Lambda^\bullet(V^*)}\alpha=p\alpha.
	\end{equation}
	One verifies easily that
	\begin{equation}
		N^{\Lambda^\bullet(V^*)}=\frac{1}{2}\sum_{i=1}^{m}c(e^{*}_{i})\widehat{c}(e_{i})+\frac{m}{2}.
	\label{eq:4.1.9guo}
	\end{equation}
\end{definition}

We refer to \cite[Chapter I]{Lawson1989spin}, \cite[Chapter 
3]{berline2003heat} for more detailed discussions 
on Clifford algebras. 

\subsection{Harmonic oscillators}\label{section3.3ssd}
Now we consider the Lie algebra $\g$ of $G$ equipped with the 
bilinear form $B$ introduced in Subsection \ref{section1-1}.
Let $c(\g)$, 
$\widehat{c}(\g)$ be the Clifford algebras associated with $(\g,B)$, 
$(\g,-B)$. By restricting $B$ to $\pp$, $\kk$, we get the Clifford algebras 
$c(\pp)$, $\widehat{c}(\pp)$, $c(\kk)$, $\widehat{c}(\kk)$. By \eqref{cartandecom1}, 
\begin{equation}
    c(\g)=c(\pp)\widehat{\otimes}c(\kk),\qquad 
    \widehat{c}(\g)=\widehat{c}(\pp)\widehat{\otimes}\widehat{c}(\kk).
    \label{eq:cplusg}
\end{equation}

If $a\in\g$, let $\nabla_{a}$ denote the corresponding 
differentiation operator along $\g$.
Let $e_{1}$, $\cdots$, $e_{m}$ be an orthonormal basis of $\pp$, 
and let $e_{m+1}$, $\cdots$, $e_{m+n}$ be an orthonormal basis of $\kk$. If $Y\in\g$, we split $Y$ in the form $Y=Y^\pp+Y^\kk$ with $Y^\pp\in\pp$, $Y^\kk\in \kk$.
Set
\begin{equation}
\label{eq:3.2.9zz}
\begin{split}
	&\mathcal{D}^\pp=\sum_{j=1}^m c(e_j)\nabla_{e_j},\; \; 
	\mathcal{E}^\pp=\widehat{c}(Y^\pp),\\
	&\mathcal{D}^\kk=-\sum_{j=m+1}^{m+n} c(e_j)\nabla_{e_j},\; \; \mathcal{E}^\kk=\widehat{c}(Y^\kk).
\end{split}
\end{equation}

Since $K$ preserves the scalar products on $\pp$ and $\kk$, the above 
constructions are $K$-equivariant. The operators $\mathcal{D}^\pp$, $\mathcal{E}^\pp$, $\mathcal{D}^{\kk}$, $\mathcal{E}^{\kk}$ are linear differential operators acting on $\Lambda^\bullet(\g^*)\otimes C^\infty(\g)$. Moreover,
\begin{equation}
[\mathcal{D}^\pp+\mathcal{E}^\pp, -i\mathcal{D}^\kk+i\mathcal{E}^\kk]=0.
\label{eq:3.2.16zz}
\end{equation}
Let $\Delta^\g$ be the Euclidean Laplacian of $(\g, 
\langle\cdot,\cdot\rangle)$. Then by \eqref{eq:bracketrelations}, 
\eqref{eq:4.1.9guo}, we get
\begin{equation}\label{eq:2.2.9bonn}
\begin{split}
\frac{1}{2}\big(\mathcal{D}^\pp+\mathcal{E}^\pp -i\mathcal{D}^\kk+i\mathcal{E}^\kk\big)^2= \frac{1}{2}\big(-\Delta^\g+|Y|^2-(m+n)\big)+N^{\Lambda^\bullet(\g^*)}.
\end{split}
\end{equation}
The kernel of the unbounded operator in \eqref{eq:2.2.9bonn} is 
one-dimensional line spanned by the function $\exp(-|Y|^2/2)/\pi^{(m+n)/4}$.

\subsection{The Dirac operator of 
Kostant}\label{section3.3}
Recall that $C^{\g}\in U\g$ is defined in 
\eqref{eq:3.3.1n} and that $\kappa^{\g}\in\Lambda^{3}(\g^{*})$ is 
defined in \eqref{eq:4.4.2nn}. Let $\kappa^\kk\in \Lambda^3(\kk^{*})$ be the element defined by the 
same formula as in \eqref{eq:4.4.2nn} with respect to $(\kk,B|_{\kk})$. Then by 
\eqref{eq:4.4.5nm}, we get
\begin{equation}
    B^*(\kappa^\kk,\kappa^\kk)=\frac{1}{6}\mathrm{Tr^\kk}[C^{\kk,\kk}].
    \label{eq:4.4.6nn}
\end{equation}

The Clifford elements $c(\kappa^\g)$, 
 $\widehat{c}(-\kappa^\g)$, 
  $c(\kappa^\kk)$, $\widehat{c}(-\kappa^\kk)$ are defined as in 
 \eqref{eq:calpha}. If $e\in \kk$, let $\mathrm{ad}(e)|_{\pp}$ be the 
 restriction of $\mathrm{ad}(e)$ to $\pp$. Then 
 $\widehat{c}(\mathrm{ad}(e)|_{\pp})\in \widehat{c}(\pp)$. By 
 \cite[(2.7.4)]{bismut2011hypoelliptic}, we have
 \begin{equation}
     \widehat{c}(-\kappa^\g)=-2 \sum_{i=m+1}^{m+n} 
     \widehat{c}(e_{i})\widehat{c}(\mathrm{ad}(e_{i})|_{\pp})+\widehat{c}(-\kappa^\kk).
     \label{eq:3.3.11gg}
 \end{equation}

\begin{definition}\label{def:kostant}
    Let $\widehat{D}^\g\in 
    \widehat{c}(\g)\otimes U\g$ be the Dirac operator,
    \begin{equation}
	   \widehat{D}^\g=\sum_{i=1}^{m+n} 
	    \widehat{c}(e^{*}_{i})e_{i}+\frac{1}{2}\widehat{c}(-\kappa^\g).
	\label{eq:3.3.10gg}
    \end{equation}
    The operator $\widehat{D}^\g$ is called the Dirac operators of Kostant \cite{KOSTANT1997275}.
\end{definition}

By \cite{KOSTANT1997275} (cf. \cite[Theorem 
2.7.2]{bismut2011hypoelliptic}), we have
\begin{equation}
\widehat{D}^{\g,2}=-C^\g-\frac{1}{4}B^*(\kappa^\g, \kappa^\g).
\label{eq:3.4.13ugc}
\end{equation}

\subsection{The operator $\mathfrak{D}^X_b$}\label{section3.5s}
As we saw in Subsection \ref{section3.3}, $\widehat{D}^\g$ acts on $C^\infty(G,\Lambda^\bullet(\g^*))$. 
Recall that 
$\mathcal{D}^\mathfrak{p}+\mathcal{E}^\mathfrak{p}-i\mathcal{D}^\mathfrak{k}+
i \mathcal{E}^\mathfrak{k}$ 
is a differential operator acting on $C^\infty(\g,\Lambda^\bullet(\g^*))$.

\begin{definition}
For $b>0$, let $\mathfrak{D}_b$ be the differential operator,
\begin{equation}
\mathfrak{D}_b=\widehat{D}^\mathfrak{g}+i c([Y^\mathfrak{k},Y^\mathfrak{p}])+\frac{1}{b}(\mathcal{D}^\mathfrak{p}+\mathcal{E}^\mathfrak{p}-i\mathcal{D}^\mathfrak{k}+i \mathcal{E}^\mathfrak{k}).
\label{eq:3.4.2zz}
\end{equation}
Then $\mathfrak{D}_b$ acts on $C^{\infty}(G\times 
\mathfrak{g},\Lambda^\bullet(\mathfrak{g^*}))$.
\end{definition}

If $Y\in\mathfrak{g}$, let $\underline{Y^\mathfrak{p}}$, $\underline{Y^\mathfrak{k}}$ 
denote the tangent vector fields on $G$ 
associated with $Y^\mathfrak{p}$, $Y^\mathfrak{k}\in\mathfrak{g}$. 
Let $\nabla^{\g}_{[Y^{\kk},Y^{\pp}]}$ denote the differentiation 
operator in the direction $[Y^{\kk},Y^{\pp}]\in\pp$ along the vector 
space $\g$.
The following identity is obtained in \cite[Section 
2.11]{bismut2011hypoelliptic}.
\begin{theorem}\label{thm:3.5.2}
We have the following formula for $\mathfrak{D}^2_b$,
\begin{equation}
\begin{split}
\frac{\mathfrak{D}_b^2}{2}=\frac{\widehat{D}^{\mathfrak{g},2}}{2}+\frac{1}{2}\big|[Y^\mathfrak{k},Y^\mathfrak{p}]\big|^2+\frac{1}{2 b^2}\big(-\Delta^{\mathfrak{p}\oplus\mathfrak{k}}+|Y|^2-m-n\big)+\frac{N^{\Lambda^\bullet(\mathfrak{g}^*)}}{b^2}&\\
                         +\frac{1}{b}\Big(\underline{Y^{\mathfrak{p}}}+i\underline{Y^{\mathfrak{k}}}-i\nabla^{\mathfrak{g}}_{[Y^\mathfrak{k},Y^\mathfrak{p}]}+\widehat{c}(\mathrm{ad}(Y^\mathfrak{p}+iY^\mathfrak{k}))&\\
                         +2 i c 
						 (\mathrm{ad}(Y^\mathfrak{k})|_\mathfrak{p})-c(\mathrm{ad}(Y^\mathfrak{p}))\Big).&
\end{split}
\label{eq:3.4.3zz}
\end{equation}
\end{theorem}

Recall that $(E,\rho^E)$ is a unitary representation of $K^\sigma$. 
Let $C_K^\infty(G\times \mathfrak{g},\Lambda^\bullet(\mathfrak{g}^*)\otimes E)$ 
denote the set of $K$-invariant sections. Recall that $\widehat{\pi}:\widehat{\mathcal{X}}\rightarrow X$ is the total 
space of $TX\oplus N$. Then we have
\begin{equation}
	C_K^\infty\big(G\times 
	\mathfrak{g},\Lambda^\bullet(\mathfrak{g}^*)\otimes E\big)= 
	C^\infty\big(\widehat{\mathcal{X}}, 
	\widehat{\pi}^*(\Lambda^\bullet(T^*X\oplus N^*)\otimes F)\big).
	\label{eq:4.44.44.4}
\end{equation}

Let $Y=Y^{TX}+Y^N$, $Y^{TX}\in TX$, $Y^N\in N$ be the tautological section of $\widehat{\pi}^*(TX\oplus N)$ over $\widehat{\mathcal{X}}$. 
\begin{definition}
Let $\mathcal{H}$ be the vector space of smooth sections over $X$ of 
the vector bundle $C^\infty\big(TX\oplus N, 
\widehat{\pi}^*(\Lambda^\bullet(T^*X\oplus N^*)\otimes F)\big)$.
\end{definition}

We can identify $\mathcal{H}$ with 
$C^\infty\big(\widehat{\mathcal{X}},\widehat{\pi}^*(\Lambda^\bullet(T^*X\oplus N^*)\otimes F)\big)$.
Let $\nabla^\mathcal{H}$ be the connection on $\mathcal{H}$ induced by the connection form $\omega^\mathfrak{k}$ on $X$.

Let $e\in\mathfrak{k}$, then the vector field $[e,Y]$ on $\mathfrak{g}$ is 
a Killing vector field. Let $L^V_{[e,Y]}$ be the Lie derivative 
acting on $C^\infty(\g, \Lambda^\bullet(\g^*))$. Then by \cite[(2.12.4)]{bismut2011hypoelliptic},
\begin{equation}
L^V_{[e,Y]}=\nabla_{[e,Y]}-(c+\widehat{c})(\mathrm{ad}(e)).
\label{eq:3.5.9ugc}
\end{equation}

Note that $\widehat{D}^\mathfrak{g}$ is $K$-invariant. Let 
$\widehat{D}^{\g,X}$ be the corresponding differential operators on 
the smooth sections of $\mathcal{H}$. By \cite[Theorem 2.12.2]{bismut2011hypoelliptic}, 
\begin{equation}
\widehat{D}^{\mathfrak{g},X}=\sum_{j=1}^m 
\widehat{c}(e_i)\nabla^\mathcal{H}_{e_i}-\sum_{j=m+1}^{m+n}\widehat{c}(e_j)\big(L^V_{[e_j,Y]}+\widehat{c}(\text{ad}(e_j)|_\mathfrak{p})-\rho^E(e_j)\big)+\frac{1}{2}\widehat{c}(-\kappa^\mathfrak{k}).
\label{eq:3.5.11ugc}
\end{equation}

Let $\mathcal{D}^{TX}$, $\mathcal{E}^{TX}$, $\mathcal{D}^N$, $\mathcal{E}^N$
be the operators acting on $\widehat{\pi}^*(\Lambda^\bullet(T^*X\oplus N^*)\otimes F)$ 
along the fibre $\widehat{\mathcal{X}}$ 
induced by $\mathcal{D}^{\pp}$, $\mathcal{E}^{\pp}$, $\mathcal{D}^\kk$, $\mathcal{E}^\kk$. 
Then $\mathfrak{D}_b$ defined in \eqref{eq:3.4.2zz} descends to 
an operator $\mathfrak{D}^X_b$ on 
$C^\infty(TX\oplus N, \widehat{\pi}^*(\Lambda^\bullet(T^*X\oplus N^*)\otimes F))$. 
Then
\begin{equation}
\mathfrak{D}^X_b=\widehat{D}^{\mathfrak{g},X}+ic([Y^N,Y^{TX}])
+\frac{1}{b}(\mathcal{D}^{TX}+\mathcal{E}^{TX}-i\mathcal{D}^N+i\mathcal{E}^N).
\label{eq:3.5.12ugc}
\end{equation}

\subsection{The hypoelliptic Laplacian}
\label{s3.7}

Recall that $A\in \mathrm{End}(E)$ is a $K^{\sigma}$-invariant such 
that it gives a parallel section of $\mathrm{End}(F)$ on $X$. Recall 
that for $t>0$, $p^{X}_{t}(x,x')$ is the heat kernel of 
$\mathcal{L}^{X}_{A}$.

Let $(\cdot,\cdot)$ denote the Hermitian metric on 
$\Lambda^\bullet(T^{*}X\oplus N^{*})\otimes F$ associated with 
$\langle\cdot,\cdot\rangle$ and 
$g^F$. The Cartan involution $\theta$ acts on 
$\widehat{\mathcal{X}}$, so that 
\begin{equation}
    \theta(Y^{TX}+Y^{N})= -Y^{TX}+Y^N.
    \label{eq:3.6.3ugc}
\end{equation}
Let $dv_{\widehat{\mathcal{X}}}$ be the volume form on 
$\widehat{\mathcal{X}}$ coming from the Riemann metric on $X$ 
and the Euclidean scalar product on $TX\oplus N$. Let $\eta(\cdot,\cdot)$ be the Hermitian form on the space of smooth 
compactly supported sections of 
$\widehat{\pi}^{*}(\Lambda^\bullet(T^{*}X\oplus N^{*})\otimes F)$ over 
$\widehat{\mathcal{X}}$,
\begin{equation}
    \eta(s,s')=\int_{\widehat{\mathcal{X}}} (s\circ \theta, 
    s')dv_{\widehat{\mathcal{X}}}.
    \label{eq:3.6.4mk2}
\end{equation}

As in \cite[Sections 2.12 and 2.13]{bismut2011hypoelliptic}, for 
$b>0$, we put
\begin{equation}
\label{eq:4.4.12hh}
\mathcal{L}^X_b=-\frac{1}{2}\widehat{D}^{\g,X,2}+\frac{1}{2} \mathfrak{D}^{X,2}_b.
\end{equation}
It acts on 
$C^\infty\big(\widehat{\mathcal{X}}, \widehat{\pi}^*(\Lambda^\bullet 
(T^*X\oplus N^*)\otimes F)\big)$, whose formula is given as follows, 
\begin{equation}\label{hypoopX}
\begin{split}
\mathcal{L}^X_b = \frac{1}{2}\left|[Y^N,Y^{TX}]\right|^2+\frac{1}{2b^2} 
\left(-\Delta^{TX\oplus N}+|Y|^2-m-n\right)
+\frac{N^{\Lambda^\bullet(T^*X\oplus N^*)}}{b^2}&\\
 +\frac{1}{b}\Big( \nabla^{\mathcal{H}}_{Y^{TX}}+ 
\widehat{c}\big(\mathrm{ad}(Y^{TX})\big)- c\big(\mathrm{ad}(Y^{TX})+i\theta 
\mathrm{ad}(Y^N)\big)-i\rho^E(Y^N)\Big).&
\end{split}
\end{equation} 
By H\"{o}rmander's theorem \cite{Hormander1967}, both 
$\mathcal{L}^X_{b}$ and $\frac{\partial}{\partial t}+ 
\mathcal{L}^X_{b}$ are hypoelliptic.
The operator $\mathcal{L}^X_b$ is called the
hypoelliptic Laplacian associated with $(G,K)$. Moreover, it is formally self-adjoint with respect to $\eta(\cdot,\cdot)$.

By \cite[Proposition 2.15.1]{bismut2011hypoelliptic}, we have the identity
\begin{equation}
[\mathfrak{D}_b^X, \mathcal{L}^X_b]=0.
\label{eq:3.6.5ugc}
\end{equation}

Since $\sigma$ preserves $B$ and the splitting \eqref{cartandecom1}, both $\widehat{D}^{\g,X}$ and $\mathfrak{D}^X_b$ commute 
with $G^\sigma$, so that 
$\mathcal{L}^X_b$ commutes with $G^\sigma$. The section $A$ lifts to $\widehat{\mathcal{X}}$. Let 
$\mathcal{L}^X_{A,b}$ be the operator acting on 
$C^\infty(\widehat{\mathcal{X}},\hat{\pi}^*(\Lambda^\bullet (T^*X\oplus N^*)\otimes F))$ given by
\begin{equation}
\mathcal{L}^X_{A,b}=\mathcal{L}^X_b+A.
\label{eq:3.7.2ugcd}
\end{equation}
In \cite[Sections 4.5 and 
11.8]{bismut2011hypoelliptic}, the heat operator 
$\exp(-t\mathcal{L}^X_{A,b})$) is well-defined for $b>0, t>0$ with a smooth kernel $q^X_{b,t}((x,Y),(x',Y'))$.

Let 
$\mathbf{P}$ be the projection from $\Lambda^\bullet(T^*X\oplus E^*)\otimes F$ on $\Lambda^0(T^*X\oplus E^*)\otimes F$. For $t>0$ and $(x,Y), (x',Y')\in \widehat{\mathcal{X}}$, put
\begin{equation}
	q^X_{0,t}\big((x,Y),(x',Y')\big)=\mathbf{P}p^X_t(x,x')\pi^{-(m+n)/2}\exp\Big(-\frac{1}{2}\big(|Y|^2+|Y'|^2\big)\Big)\mathbf{P}.
	\label{eq::3.7.5ugcd}
\end{equation}

We recall a result established in \cite[Theorem 4.5.2 and Chapter 
14]{bismut2011hypoelliptic}.
\begin{theorem}\label{thm:twokernels}
Given $M\geq\epsilon>0$, there exist $C,C'>0$ such that for 
$0<b\leq M,\epsilon \leq t\leq M, 
(x,Y),(x',Y')\in\widehat{\mathcal{X}}$,
\begin{equation}
	\Big|q^{X}_{b,t}\big((x,Y),(x',Y')\big)\Big|\leq 
	C\exp\Big(-C'\big(d^{2}(x,x')+|Y|^{2}+|Y'|^{2}\big)\Big).
	\label{eq:bkernelinQ}
\end{equation}
As $b\rightarrow 0$, we have the uniform convergence on 
compact subsets of $\widehat{\mathcal{X}}\times\widehat{\mathcal{X}}$,
\begin{equation}\label{eq:3.61n}
q^X_{b,t}\big((x,Y),(x',Y')\big)\rightarrow 
q^X_{0,t}\big((x,Y),(x',Y')\big).
\end{equation}
\end{theorem}

\begin{example}[A simple example of the hypoelliptic Laplacian]
	A simple example of our setting is the real line $\R$ with 
	additive Lie group structure. In this case, $G=\R$, $K=0$, so that $X=\R$ 
	with the standard Euclidean metric. Let $x\in \R$ denote the global 
	coordinate of $X$, and let $y=y\frac{\partial}{\partial x}$ 
	denote the coordinate along the tangent vector space of $X$.
	Then $\widehat{\mathcal{X}}=\mathcal{X}=\R_{x}\times \R_{y}$ is 
	just the total space of $T\R$, where the subscripts $x$, $y$ 
	indicate the respective coordinates. By \eqref{hypoopX}, the 
	operator $\mathcal{L}^{\R}_{b}$ acting on 
	$C^{\infty}(\mathbb{R}_{x} \times \mathbb{R}_{y}, \Lambda^\bullet(\R^{*}_{y}))$ is 
	given by
	\begin{equation}
		\mathcal{L}^{\R}_{b}=\frac{1}{2b^{2}}(-\Delta_{y}+y^{2}-1)+\frac{N^{\Lambda^\bullet(\R^{*}_{y})}}{b^{2}}+\frac{1}{b} y\frac{\partial}{\partial x}.
	\end{equation}
	Note that $\frac{\partial}{\partial t}+\mathcal{L}^{\R}_{b}$ is 
	just the Kolmogorov operator (\cite{10.2307/1968123}, up to a 
	conjugation). The heat kernel of $\mathcal{L}^{\R}_{b}$ has an explicit 
	expression given in \cite[Subsection 
	10.5]{bismut2011hypoelliptic}, so that the convergence 
	\eqref{eq:3.61n} can be verified directly. Here, we would like to 
	give another straightforward computation to understand this 
	convergence. 
	
	The geodesic flow $\varphi_{t}$, $t\in\R$ on 
	$\mathcal{X}$ is given by $\varphi_{t}(x,y)=(x+ty,y)$. For $f\in 
	C^{\infty}(\R_{x},\R)$, we identify it with the section 
	$f(x)\frac{1}{\pi^{1/4}}\exp(-y^{2}/2)\in C^{\infty}(\mathbb{R}_{x} \times \mathbb{R}_{y}, 
	\Lambda^\bullet(\R^{*}_{y}))$. This identification preserves the 
	$L_{2}$-metrics for $L_{2}$-functions. A direct computation shows 
	that for $b>0$, 
	\begin{equation}\label{eq:4.5.11new}
		\mathcal{L}^{\R}_{b}\varphi^{*}_{-b}\big(f(x)\frac{1}{\pi^{1/4}}e^{-y^{2}/2}\big)=\varphi^{*}_{-b}\big(-\frac{1}{2}\Delta_{x}(f)\frac{1}{\pi^{1/4}}e^{-y^{2}/2}\big).
	\end{equation}
	This gives an explicit relation, conjugation by the geodesic 
	flow, between the hypoelliptic 
	Laplacian $\mathcal{L}^{\R}_{b}$ and the elliptic Laplacian 
	$-\frac{1}{2}\Delta_{x}$ on $X=\R$. If we take $b\rightarrow 0$ 
	in \eqref{eq:4.5.11new}, it explains well the 
	convergence in \eqref{eq:3.61n}.
\end{example}

\section{A proof of Theorem 
\ref{thm_orbitalintegral}}\label{section:proof}
The purpose of this section is to establish Theorem 
\ref{thm_orbitalintegral}. The geometric constructions in Sections \ref{s1} and \ref{section3} play 
important roles in the proof. In particular, due to the geometric formulations of 
the twisted orbital supertrace $\mathrm{Tr_s}^{[\gamma\sigma]}[\exp(-t\mathcal{L}^X_{A,b})]$, the local 
index techniques used in \cite[Chapter 9]{bismut2011hypoelliptic} are 
still applicable to compute explicitly its limit as $b\rightarrow 
\infty$. Therefore, our proof to Theorem \ref{thm_orbitalintegral} is partly derived from \cite[Chapter 9]{bismut2011hypoelliptic}.

\subsection{A fundamental identity for twisted orbital 
supertraces}\label{section4.4}
Recall that $\mathcal{L}^X_A$, $\mathcal{L}^X_{A,b}$ 
were defined in Subsections \ref{s4.2} and \ref{s3.7}, and that $p^X_t, 
q^X_{b,t}$ are the associated elliptic and hypoelliptic heat kernels. Using \eqref{eq:bkernelinQ} and the fact that $\mathcal{L}_{A,b}^X$ 
commutes with $\sigma$, for $b>0,\;t>0$, 
$q^X_{b,t}\in\mathfrak{Q}^{\sigma,\infty}$ (cf. Definition \ref{def:3.3.1bis}). By the results of Subsection 
\ref{section-infinite}, 
$\mathrm{Tr_s}^{[\gamma\sigma]}[\exp(-t\mathcal{L}^X_{A,b})]$ is well-defined. The following theorem extends \cite[Theorem 
4.6.1]{bismut2011hypoelliptic}.

\begin{theorem}\label{thm_trequaltrs}
For any $b>0,t>0$, the following identity holds,
\begin{equation}\label{eq:3.62n}
\mathrm{Tr_s}^{[\gamma\sigma]}\left[\exp(-t\mathcal{L}^X_{A,b})\right]=\mathrm{Tr}^{[\gamma\sigma]}\left[\exp(-t \mathcal{L}^X_A)\right].
\end{equation}
\end{theorem}

\begin{proof}
By \eqref{orbitaldef2} and using Proposition \ref{prop_3.4.3}, we get
\begin{equation}\label{eq:3.64n}
\frac{\partial}{\partial b} 
\mathrm{Tr_s}^{[\gamma\sigma]}[\exp(-t\mathcal{L}^X_{A,b})]= -t 
\mathrm{Tr_s}^{[\gamma\sigma]} \Big[\big(\frac{\partial}{\partial 
b}\mathcal{L}^X_{A,b}\big) \exp(-t \mathcal{L}^X_{A,b})\Big]. 
\end{equation}
By \eqref{eq:4.4.12hh} and \eqref{eq:3.6.5ugc}, we have
\begin{equation}\label{eq:3.65n}
\frac{\partial }{\partial b}\mathcal{L}^X_{A,b}= \frac{1}{2}[\mathfrak{D}^X_b, \frac{\partial}{\partial b} \mathfrak{D}^X_b],\;\; [\mathfrak{D}^X_b, \mathcal{L}^X_{A,b}]=0.
\end{equation}
Then we get
\begin{equation}
	\begin{split}
		\frac{\partial}{\partial b} 
		\mathrm{Tr_s}^{[\gamma\sigma]}[\exp(-t\mathcal{L}^X_{A,b})]&=-\frac{t}{2}\mathrm{Tr_s}^{[\gamma\sigma]} \Big[\big[\mathfrak{D}^X_b, \frac{\partial}{\partial b} \mathfrak{D}^X_b\big] \exp(-t \mathcal{L}^X_{A,b})\Big] \\
		&=-\frac{t}{2}\mathrm{Tr_s}^{[\gamma\sigma]} 
		\Big[\big[\mathfrak{D}^X_b, \big(\frac{\partial}{\partial b} 
		\mathfrak{D}^X_b\big)\exp(-t \mathcal{L}^X_{A,b})\big]\Big].
	\end{split}
	\label{eq:3.66n}
\end{equation}

Applying again Proposition \ref{prop_3.4.3}, we get
\begin{equation}
	\frac{\partial}{\partial b} 
	\mathrm{Tr_s}^{[\gamma\sigma]}[\exp(-t\mathcal{L}^X_{A,b})]=0.
	\label{eq:3.67n}
\end{equation}

Now we only need to prove that
\begin{equation}
	\lim_{b\rightarrow 0} 
	\mathrm{Tr_s}^{[\gamma\sigma]}\left[\exp(-t\mathcal{L}^X_{A,b})\right]= 
	\mathrm{Tr}^{[\gamma\sigma]}\left[\exp(-t \mathcal{L}^X_A)\right].
	\label{eq:limit}
\end{equation}
By  \eqref{lineardist} and Theorem \ref{thm:twokernels}, given $t>0$, 
there exist $C,C'>0$ such that for $0<b\leq 1, 
f\in\pp^{\perp}_{\sigma}(\gamma), Y\in (TX\oplus N)_{e^{f}p1}$,
\begin{equation}
		\Big|q^{X}_{b,t}\big((e^{f}p1,Y),\gamma\sigma(e^{f}p1,Y)\big)\Big|\leq 
	C\exp\big(-C'(|f|^{2}+|Y|^{2})\big)
	\label{eq:3.69n}
\end{equation}

Using \eqref{orbitaldef1}, \eqref{orbitaldef2}, \eqref{eq:3.61n} and dominated convergence, we get \eqref{eq:limit}. The proof of our 
theorem is completed.
\end{proof}

\subsection{An identity for 
$J_{\gamma\sigma}(Y^{\kk}_{0})$}\label{proof5.2}

Recall that $p=\dim \pp_{\sigma}(\gamma)$, $q=\dim \kk_{\sigma}(\gamma)$, $r=\dim \z_{\sigma}(\gamma)=p+q$. Let $e_1,\cdots, e_p$ be an orthonormal 
basis of $\pp_{\sigma}(\gamma)$, and let $e_{p+1},\cdots, e_r$ be an 
orthonormal basis of $\kk_{\sigma}(\gamma)$. Let $e^1$, $\cdots$, 
$e^r$ be the corresponding dual basis of $\z_{\sigma}(\gamma)^*$. Let 
$\uzs(\gamma)$, $\uzs(\gamma)^*$ be another copies of 
$\z_{\sigma}(\gamma)$, $\z_{\sigma}(\gamma)^*$. We underline the 
obvious objects associated with $\uzs(\gamma)$, $\uzs(\gamma)^*$. Let 
$c(\z_\sigma(\gamma))$ denote the Clifford algebra associated with 
$(\z_\sigma(\gamma),B|_{\z_\sigma(\gamma)})$, we also identify the 
elements in $c(\z_\sigma(\gamma))$ with their actions on 
$\Lambda^\bullet(\g^*)$ given in \eqref{eq:actionscc}.

By \eqref{cartandecom1}, we get
\begin{equation}\label{eq_6.2}
\pp\times\g=\pp\times (\pp\oplus\kk).
\end{equation}
We denote by $y$ the tautological section of the first copy of $\pp$ 
in the right-hand side of \eqref{eq_6.2}, and by $Y^\g=Y^\pp+Y^\kk$ 
the tautological section of $\g=\pp\oplus\kk$. We also denote by
$dy,\, dY^\g=dY^\pp dY^\kk$ the volume forms on $\pp$, $\g$
respectively. Recall that
$\Delta^{\pp\oplus\kk}$ is the standard
Laplacian on $\g=\pp\oplus\kk$, i.e., the second factor in the
right-hand side of \eqref{eq_6.2}. Let $\nabla^H$ denote
differentiation in the variable $y\in\pp$, and let $\nabla^V$ denote
the differentiation in the variable $Y^\g\in\g$.

Put 
\begin{equation}
\alpha=\sum_{i=1}^{r} c(e_i)\underline{e}^i\in
c(\z_{\sigma}(\gamma))\widehat{\otimes}
\Lambda^\bullet(\uzs(\gamma)^*).
\label{eq:5.1.15didot}
\end{equation}
As an analogue in \cite[Section 
5.1]{bismut2011hypoelliptic}, if $Y^\kk_0\in\kk_{\sigma}(\gamma)$, let 
$\mathcal{P}_{a,Y_0^\kk}$ be the differential operator acting on $C^\infty(\pp\times
\g,\Lambda^\bullet(\g^*)\widehat{\otimes}
\Lambda^\bullet(\uzs(\gamma)^*))$ defined as follows,
	\begin{equation}\label{eq_operatorP}
	\begin{split}
	\mathcal{P}_{a,
Y^\kk_0}=&\frac{1}{2}\big|[Y^\kk,a]+[Y^\kk_0,Y^\pp]\big|^2-\frac{1}{2}\Delta^{\pp\oplus\kk}+\alpha-\nabla^H_{Y^\pp}\\
	&-\nabla^V_{[a+Y^\kk_0,[a,y]]}
-\widehat{c}(\mathrm{ad}(a))+c\left(\mathrm{ad}(a)+i\theta\mathrm{ad}(Y^\kk_0)\right).
	\end{split}
	\end{equation}
By H\"{o}rmander's theorem \cite{Hormander1967}, both $\mathcal{P}_{a, Y^\kk_0}$, $\frac{\partial}{\partial t}+ \mathcal{P}_{a, Y^\kk_0} $ are
hypoelliptic.

Let $R_{Y^\kk_0}$ be the smooth kernel of 
$\exp(-\mathcal{P}_{a, Y^\kk_0})$ with respect to the volume $dy 
dY^\g$ on $\pp\times\g$. Then by \cite[(5.1.10)]{bismut2011hypoelliptic}, for $(y,Y^\g),(y',Y^{\g\prime})\in 
\pp\times\g$,
\begin{equation}
R_{Y^\kk_0}\big((y,Y^\g),(y',Y^{\g\prime})\big)\in\mathrm{End}(\Lambda^\bullet(\z^\perp_\sigma(\gamma)^*))\widehat{\otimes}
c(\z_{\sigma}(\gamma))\widehat{\otimes}
\Lambda^\bullet(\uzs(\gamma)^*).
\end{equation}

\begin{definition}\label{def:hatstr}
Let $\widehat{\mathrm{Tr_s}}$ denote the supertrace functional on 
$c(\z_{\sigma}(\gamma))\widehat{\otimes} 
\Lambda^\bullet(\uzs(\gamma)^*)$  such that it 
vanishes on monomials of nonmaximal length, and gives the value
$(-1)^r$ to the monomial $c(e_1)\underline{e}^1 \cdots 
c(e_r)\underline{e}^r$. It also extends to a supertrace functional on $\mathrm{End}(\Lambda^\bullet(\z^\perp_\sigma(\gamma)^*))\widehat{\otimes}
c(\z_{\sigma}(\gamma))\widehat{\otimes}
\Lambda^\bullet(\underline{\z_{\sigma}}(\gamma)^*)$ by tensoring with the supertrace on 
$\mathrm{End}(\Lambda^\bullet(\z_{\sigma}^\perp(\gamma)^*))$. We still 
denote it by $\widehat{\mathrm{Tr_s}}$.
\end{definition}

Now we give the important result established in \cite[Theorem 5.5.1]{bismut2011hypoelliptic}.
\begin{proposition}
	For $Y^\kk_0\in\kk_{\sigma}(\gamma)$, we have
\begin{equation}\label{eq_traceJ}
	\begin{split}
		J_{\gamma\sigma}(Y^\kk_0)=(2\pi)^{r/2}\int_{y\in\pp^\perp_{\sigma}(\gamma), Y^\g\in \pp\oplus \kk^\perp_\sigma(\gamma)}
\widehat{\mathrm{Tr_s}}\bigg[\mathrm{Ad}(k^{-1}\sigma) &\\
R_{Y^\kk_0}\Big((y,Y^\g),\mathrm{Ad}(k^{-1}\sigma)(y,Y^{\g})\Big)
\bigg]dydY^\g.&
	\end{split}
\end{equation}
\end{proposition}

\begin{proof}
In the proof of \cite[Theorem 5.5.1]{bismut2011hypoelliptic},
the computations of the supertrace functional in the right-hand side 
of \eqref{eq_traceJ} only depend on the adjoint actions of $\gamma$, 
$k^{-1}$ and $a$ and the fact that they commute with each other. 
Therefore, when replacing $\gamma$, $k^{-1}$ by $\gamma\sigma$, 
$k^{-1}\sigma$, the computations in \cite[Chapter 
5]{bismut2011hypoelliptic} still hold, so that \eqref{eq_traceJ} holds. 
\end{proof}

\subsection{A proof of Theorem 
\ref{thm_orbitalintegral}}\label{proof5.3}

Recall that $\widehat{\mathcal{X}}$ is the total space of the vector 
bundle $TX\oplus N\rightarrow X$, which can be canonically identified 
with $X\times \g$ as in \eqref{eq:1.1.13sud}.
For $b\neq 0$, $s(x,Y)\in 
C^{\infty}\big(\widehat{\mathcal{X}},\widehat{\pi}^{*}(\Lambda^\bullet(T^{*}X\oplus
N^{*})\otimes F)\big)$, set
\begin{equation}
	F_{b}s(x,Y)=s(x,bY).
	\label{eq:Frescaling}
\end{equation}
For $t>0$, we denote with an extra subscript $t$ the hypoelliptic 
Laplacian defined in Subsection \ref{s3.7} associated with the 
bilinear form $B/t$. By \eqref{hypoopX}, we have
\begin{equation}
	F_{\sqrt{t}}t^{N^{\Lambda^\bullet(T^*X\oplus N^*)/2}}\mathcal{L}^X_{b,t}t^{-N^{\Lambda^\bullet(T^*X\oplus N^*)/2}}F_{\sqrt{t}}^{-1}=t\mathcal{L}^X_{\sqrt{t}b}.
	\label{eq:timerescaling}
\end{equation}
By Remark \eqref{rem:BtJ} and \eqref{eq:timerescaling}, it is 
enough to prove \eqref{eq:4.2.1} with $t=1$. When $t=1$, we will drop the subscript $t$ in the corresponding notation of heat kernels, such as $q^X_{b}=q^X_{b,1}$.

By 
\eqref{eq:3.62n}, we will make $b\rightarrow +\infty$ in 
$\mathrm{Tr_s}^{[\gamma\sigma]}[\exp(-\mathcal{L}^X_{A,b})]$. 
Generally speaking, all the analytic and geometric constructions of \cite{bismut2011hypoelliptic} only depend on the fact that $G$ acts on $X$ as a group of isometries, replacing $G$ by $G^\sigma$ does not change anything from that point of view. This is why we will freely use the arguments in \cite[Chapter 9]{bismut2011hypoelliptic}. We sketch the main steps of the proof as follows.

At first, we introduce the $\gamma\sigma$-periodic points of the 
geodesic flow on $\widehat{\mathcal{X}}$ . Let $\{\varphi_t\}_{t\in 
\mathbb{R}}$ denote the geodesic flow on $\mathcal{X}$ associated with $g^{TX}$. The flow $\{\varphi_{t}\}_{t\in\R}$ lifts to a flow of diffeomorphisms of $\widehat{\mathcal{X}}$. If $(x,Y^{TX},Y^{N})\in \widehat{\mathcal{X}}$, set
\begin{equation}
 (x_{t},Y^{TX}_{t},Y^{N}_{t})=\varphi_{t}(x,Y^{TX}, Y^{N}).
 \label{eq:1.6.19ugc}
 \end{equation}
Then $x_{t}$ is just the geodesic starting at $x$ with speed 
$Y_{t}^{TX}$, and $Y^{N}_{t}$ is the parallel transport of $Y^{N}$ along $x_{t}$.
Set
\begin{equation}
	\widehat{\mathcal{F}}_{\gamma\sigma}=\{z\in\widehat{\mathcal{X}}\;:\;
	(\gamma\sigma)^{-1} \varphi_{1} z=z\}.
	\label{eq:fixset}
\end{equation}

The vector 
$a\in \pp$ defines a constant section of $X\times \g$. By \eqref{eq:1.1.13sud}, we can view $a$ as a smooth section of $TX\oplus N$. Let $a^{TX}$, $a^{N}$ the 
corresponding parts of this section in $TX$, 
$N$ respectively. In the global coordinate system 
$(\pp,\exp_{x_{0}})$ of $X$ defined in Subsection \ref{section1-1}, 
for $Y^{\pp}\in \pp$, by \cite[Proposition 
3.2.4]{bismut2011hypoelliptic}, we have
\begin{equation}
	a^{TX}(Y^{\pp})=\cosh(\mathrm{ad}(Y^{\pp}))a.
\end{equation}
Set
\begin{equation}
		\begin{split}
			N_{\sigma}(k^{-1})&=Z^{0}_{\sigma}(\gamma)\times_{K^{0}_{\sigma}(\gamma)}\kk_{\sigma}(k^{-1})\\
			&=\{Y^{N}\in N|_{X(\gamma\sigma)}\;:\; 
			\mathrm{Ad}(k^{-1}\sigma)Y^{N}=Y^{N}\}
		\end{split}
			\label{eq:idenset}
			\end{equation}
Then we have
\begin{equation}
	\widehat{\mathcal{F}}_{\gamma\sigma}=\big\{\big(x, a^{TX}(x), 
	Y^{N}\big)\in\widehat{\mathcal{X}}\;:\; (x, Y^{N})\in 
	N_{\sigma}(k^{-1}), x\in X(\gamma\sigma)\big\}.
\end{equation}

Put
	\begin{equation}\label{eq:bonnsunday}
	\underline{\mathcal{L}}^X_{A,b}=F_{-b}\mathcal{L}^X_{A,b}F_{-b}^{-1}.
	\end{equation}
Let $\underline{q}^X_{b}$ be the kernel associated with 
$\exp(-\underline{\mathcal{L}}^X_{A,b})$ with respect to $dxdY$. By 
\eqref{orbitaldef2}, we can use $\underline{q}^X_b$ instead of $q^X_b$ to define $\mathrm{Tr_s}^{[\gamma\sigma]}[\exp(-\mathcal{L}^X_{A,b})]$.

An important property of the hypoelliptic heat kernel proved in 
\cite[Theorem 9.1.1]{bismut2011hypoelliptic}, after adapting to our 
twisted case, is that for the points $(x,Y)\in \widehat{\mathcal{X}}$ away from 
$\widehat{\mathcal{F}}_{\gamma\sigma}$, the norm of 
$\underline{q}^X_b((x,Y),\gamma\sigma(x,Y))$ decays to $0$ 
exponentially with respect to $d_{\gamma\sigma}(x)$ and $|Y|$ (cf. 
\eqref{eq:bkernelinQ}).
As a consequence, by the arguments 
	in \cite[Section 9.2]{bismut2011hypoelliptic}, if 
	$\beta\in\;]0,1]$ is fixed, $\mathrm{Tr_s}^{[\gamma\sigma]}[\exp(-\mathcal{L}^X_{A,b})]$ is given by the limit as $b\rightarrow +\infty$ of the following integral,
	\begin{equation}\label{eq:4.2.4bonn}
	\int_{\small\substack{f\in\pp_\sigma^\perp(\gamma),\; |f|\leq 
	\beta; \\ Y\in\;\g,\\ |Y^{TX} - a^{TX}|\leq \beta. }} 
	\mathrm{Tr_s}^{\Lambda^\bullet(T^*X\oplus N^*)\otimes 
	F}\Big[\gamma\sigma 
	\underline{q}^X_b\big((pe^f,Y),\gamma\sigma(pe^f, Y)\big)\Big]r(f)dfdY.
	\end{equation} 

	Recall that $dY=dY^{TX}dY^N$. Put 
	$N_{\sigma}(\gamma)=Z^0_\sigma(\gamma)\times_{K^0_\sigma(\gamma)} 
	\kk_\sigma(\gamma)$ the vector bundle on $X(\gamma\sigma)$. Let 
	$N^{\perp}_{\sigma}(\gamma)$ be the orthogonal bundle of 
	$N_{\sigma}(\gamma)$ in $N|_{X(\gamma\sigma)}$, then 
	$N^{\perp}_{\sigma}(\gamma)=Z^0_\sigma(\gamma)\times_{K^0_\sigma(\gamma)}\kk^\perp_\sigma(\gamma)$. Recall that the projection $P_{\gamma\sigma}: X\rightarrow X(\gamma\sigma)$ is described in Theorem \ref{thm_normalcoord}.
	
We trivialize the vector bundles $TX$, $N$ by parallel transport 
along the geodesics orthogonal to $X(\gamma\sigma)$ with respect to $\nabla^{TX}, \nabla^{N}$, then the vector bundles $TX$, $N$ 
on $X$ can be identified 
with $P^*_{\gamma\sigma}(TX|_{X(\gamma\sigma)})$,  $P^*_{\gamma\sigma}(N|_{X(\gamma\sigma)})$. 
If $f\in \pp^\perp_\sigma(\gamma)$, at $\rho_{\gamma\sigma}(1,f)$, we may write $Y^N\in N$ in the form
	\begin{equation}\label{eq:bonnmonday1}
	Y^N=Y^\kk_0+ Y^{N,\perp}, \; Y^\kk_0\in \kk_\sigma(\gamma),\; Y^{N,\perp}\in \kk^\perp_\sigma(\gamma).
	\end{equation}
	Let $dY^\kk_0$, $dY^{N,\perp}$ be the volume elements on $\kk_\sigma(\gamma)$, $\kk^\perp_\sigma(\gamma)$, so that $dY^N=dY^\kk_0 dY^{N,\perp}$. We rewrite the integral in \eqref{eq:4.2.4bonn} as follows,
	\begin{equation}\label{eq:4.2.6bonn}
	\begin{split}
	b^{-4m-2n+2r}&\int_{\begin{subarray}\; f\in\pp_\sigma^\perp(\gamma),\; |f|\leq b^2\beta \\ Y\in\;\g,\; |Y^{TX}|\leq b^2\beta \end{subarray}} \mathrm{Tr_s}^{\Lambda^\bullet(T^*X\oplus N^*)\otimes F}\\
	&\bigg[\gamma\sigma \underline{q}^X_b\Big(\big(pe^{f/b^2},a^{TX}+\frac{Y^{TX}}{b^2}, 
	Y^\kk_0+\frac{Y^{N,\perp}}{b^2}\big),\\
	&\gamma\sigma\big(pe^{f/b^2},a^{TX}+\frac{Y^{TX}}{b^2}, 
	Y^\kk_0+\frac{Y^{N,\perp}}{b^2}\big)\Big)\bigg]r(\frac{f}{b^{2}})dfdY^{TX}dY^\kk_0 dY^{N,\perp}.
	\end{split}
	\end{equation} 

Now we deal with the factor $b^{2r}$ in 
\eqref{eq:4.2.6bonn} as in \cite[Sections 9.3 - 
9.5]{bismut2011hypoelliptic}. Recall that $\alpha$ is defined in 
\eqref{eq:5.1.15didot}. By \eqref{eq:1.1.13sud},  $TX\oplus N$ is identified with the trivial vector bundle 
$\g$ on $X$. Let $(TX\oplus N)_{\sigma}(\gamma)$ be the subbundle of $TX\oplus 
N$ corresponding $\z_\sigma(\gamma)\subset \g$, and let $(\underline{TX\oplus N})_{\sigma}(\gamma)^*$ be one copy of the dual bundle of $(TX\oplus N)_{\sigma}(\gamma)$. We regard $\alpha$ as a constant section of the trivial bundle $c(\g)\otimes \underline{\z_\sigma}(\gamma)^*$, hence a section of $c(TX\oplus N)\otimes (\underline{TX\oplus N})_{\sigma}(\gamma)^*$.

	Set
	\begin{equation}
	\label{eq:4.2.7bonn}
	\mathfrak{L}^X_{A,b}=\underline{\mathcal{L}}^X_{A,b}+\alpha.
	\end{equation}
	It acts on $C^\infty\big(\widehat{\mathcal{X}}, 
	\widehat{\pi}^*\big(\Lambda^\bullet(T^*X\oplus N^*)\otimes 
	F\widehat{\otimes}\Lambda^\bullet((\underline{TX\oplus 
	N})_{\sigma}(\gamma)^*)\big)\big)$. Note that $\mathfrak{L}^X_{A,b}$ commutes with the action of $\gamma\sigma$, $e^a$ and $k^{-1}\sigma$. Let $\mathfrak{q}^X_b$ be the smooth kernel of $\exp(-\mathfrak{L}^X_{A,b})$ with respect to the volume $dxdY$. 

We extend the basis $\{e_i\}_{i=1}^r$ of 
	$\z_\sigma(\gamma)$ to an orthonormal basis $\{e_i\}_{i=1}^{m+n}$ 
	of $(\g, \langle\cdot,\cdot\rangle)$. Since 
	$\mathrm{End}(\Lambda^\bullet(\g^*))=c(\g)\widehat{\otimes}\widehat{c}(\g)$, 
	we can extend $\widehat{\mathrm{Tr_s}}$ in Definition 
	\ref{def:hatstr} to a linear functional on 
	$\mathrm{End}(\Lambda^\bullet(\g^*))\widehat{\otimes}\Lambda^\bullet(\underline{\z_{\sigma}}(\gamma)^{*})$ by making it vanish on all the monomials in the $c(e_i)$, $\widehat{c}(e_i)$, $\underline{e}^k$, $1\leq i\leq m+n$, $1\leq k \leq r$ except on
	\begin{equation}\label{eq:termsbonn}
	c(e_1)\underline{e}^1\cdots c(e_r)\underline{e}^r c(e_{r+1})\widehat{c}(e_{r+1})\cdots c(e_{m+n})\widehat{c}(e_{m+n}).
	\end{equation}
Moreover,
\begin{equation}
	\begin{split}
			\widehat{\mathrm{Tr_{s}}}\big[c(e_1)\underline{e}^1\cdots 
			c(e_r)\underline{e}^r 
			c(e_{r+1})\widehat{c}(e_{r+1})\cdots 
			c(e_{m+n})\widehat{c}(e_{m+n})\big]&\\
			=(-1)^{r+n-q}(-2)^{m+n-r}.&
	\end{split}
	\label{eq:5.3.888}
\end{equation}
The map $\widehat{\mathrm{Tr_s}}$ also extends to a linear functional 
on 
$\mathrm{End}(\Lambda^\bullet(\g^*))\widehat{\otimes}\Lambda^\bullet(\underline{\z_{\sigma}}(\gamma)^*)\otimes\mathrm{End}(E)$ by tensoring with $\mathrm{Tr}^{E}[\cdot]$.

By \cite[Theorem 9.5.2 and Proposition 
9.5.4]{bismut2011hypoelliptic}, the operator 
$\mathfrak{L}^X_{A,b}$ is conjugate to 
$\underline{\mathcal{L}}^X_{A,b}$, and if $(x,Y)\in \widehat{\mathcal{X}}$, then
	\begin{equation}
	\label{eq:4.2.8bonn}
	\mathrm{Tr_s}\big[\gamma\sigma\underline{q}^X_b\big((x,Y),\gamma\sigma(x,Y)\big)\big]=b^{-2r}\widehat{\mathrm{Tr_s}}\big[\gamma\sigma \mathfrak{q}^X_b\big((x,Y),\gamma\sigma(x,Y)\big)\big].
	\end{equation}
Now we proceed as in \cite[Sections 9.8 - 
	9.11]{bismut2011hypoelliptic}, we can establish an analog of 
	\cite[Theorem 9.6.1]{bismut2011hypoelliptic}, which says that as 
	$b\rightarrow +\infty$, 
		\begin{equation}
	\label{eq:4.2.15bonn}
	\begin{split}
		b^{-4m-2n}&\gamma\sigma  
		\mathfrak{q}^X_b\Big(\big(pe^{f/b^2}, a^{TX}+Y^{TX}/b^2, 
		Y^\kk_0+Y^{N,\perp}/b^2\big), \\
		&\gamma\sigma\big(pe^{f/b^2}, a^{TX}+Y^{TX}/b^2, 
		Y^\kk_0+Y^{N,\perp}/b^2\big)\Big)\\
		&\rightarrow 
		\exp\big(-|a|^{2}/2-|Y^{\kk}_{0}|^{2}/2\big)\mathrm{Ad}(k^{-1}\sigma) 
		R_{Y^{0}_{\kk}}\big((f,Y),\mathrm{Ad}(k^{-1}\sigma)(f,Y)\big)\\
		&\;\;\;\;\;\;\;\;\;\;\; 
		\rho^{E}(k^{-1}\sigma)\exp\big(-i\rho^{E}(Y^{\kk}_{0})-A\big).
	\end{split}
	\end{equation}
	
	As in \cite[Theorem 9.5.6]{bismut2011hypoelliptic}, there exist 
	$C_{\beta}>0$, $C_{\gamma\sigma,\beta}>0$ such that for $b\geq 
	1$, $f\in \pp^\perp_\sigma(\gamma)$, $|f|\leq b^2\beta$, 
	$|Y^{TX}|\leq b^2\beta$, then left-hand side of \eqref{eq:4.2.15bonn} is bounded by 
	\begin{equation}
	\label{eq:4.2.19bonn}
	C_{\beta} 
	\exp\Big(-C_{\gamma\sigma,\beta}\big(|f|^2+|Y^{TX}|^2+|Y^\kk_0|^2 
	+ |Y^{N,\perp}|\big) \Big).
	\end{equation}

	Combining \eqref{eq:4.2.6bonn} and \eqref{eq:4.2.8bonn} - \eqref{eq:4.2.19bonn}, we get 
	\begin{equation}\label{eq:5.6n}
	\begin{split}
	&\lim_{b\rightarrow 
	+\infty}\mathrm{Tr_s}^{[\gamma\sigma]}[\exp(-\mathcal{L}^X_{A,b})]=\exp(-|a|^2/2)\int_{\begin{subarray}\;\;(y,Y^\g, Y^\kk_0) \\ \in \pp^\perp_{\sigma}(\gamma)\times \big(\pp\oplus \kk^\perp_{\sigma}(\gamma)\big)\times \kk_{\sigma}(\gamma)\end{subarray}}\\
	&\qquad\qquad\widehat{\mathrm{Tr_{s}}}\Big[\mathrm{Ad}(k^{-1}\sigma)
	R_{Y^{\kk}_{0}}\big((f,Y^\g),\mathrm{Ad}(k^{-1}\sigma)(f,Y^\g)\big)\Big]\\
	&\qquad\qquad\mathrm{Tr}^E\Big[\rho^{E}(k^{-1}\sigma)\exp\big(-i\rho^{E}(Y^{\kk}_{0})-A\big)\Big]\exp(-|Y^{\kk}_{0}|^{2}/2)dydY^\g dY^\kk_0.
	\end{split}
	\end{equation}
By \eqref{eq_traceJ}, \eqref{eq:5.6n}, we get \eqref{eq:4.2.1}. 
	This completes the proof of Theorem \ref{thm_orbitalintegral}.

\section{Connections with the local equivariant index 
theory}\label{section5bonn}
In this section, we show that our formula in \eqref{eq:4.2.1} is 
compatible with the local equivariant index theorems for compact 
locally symmetric spaces. We also apply our formula to study the 
twisted $L_{2}$-torsions introduced in 
\cite{BeLip2017} for compact locally symmetric spaces.


\subsection{The classical Dirac operator on $X$}\label{s4.5}
In this subsection, we will assume $\pp$ to be even dimensional and oriented, 
and $K$ to be semisimple and simply connected. Recall that $m=\dim\pp$.

Let $\mathrm{Spin}(\pp)$ be the Spin group of Euclidean space $\pp$, 
then the adjoint 
representation $K\rightarrow \mathrm{SO}(\pp)$ 
lifts to a homomorphism $K\rightarrow
\mathrm{Spin}(\pp)$. Let $S^\pp=S_{+}^\pp\oplus S_{-}^\pp$ be the 
$\mathds{Z}_2$-graded Hermitian vector space of 
$\pp$-spinors. To avoid confusion with the notation in Subsection \ref{section3.1}, let 
$\bar{c}(\pp)$ denote the Clifford 
algebra of $(\pp, B|_{\pp})$ acting on $S^\pp$.
Therefore, $K$ acts on $S^\pp_{\pm}$ via the spin 
representation $\rho^{S^\pp_{\pm}}$.

Set
\begin{equation}
	P_{\mathrm{SO}}(X)=G\times_{K} \mathrm{SO}(\pp),\; 
	P_{\mathrm{Spin}}(X)=G\times_{K} \mathrm{Spin}(\pp),
	\label{eq:7.1.4didot}
\end{equation}
where $K$ acts on $\mathrm{SO}(\pp)$, $\mathrm{Spin}(\pp)$ by 
conjugation. Then we get a double covering  $P_{\mathrm{Spin}}(X)\rightarrow 
P_{\mathrm{SO}}(X)$, which defines a spin structure on $X$. 
Moreover, $S^\pp$ descends to the Hermitian vector bundle $S^{TX}=S_+^{TX}\oplus 
S^{TX}_-$ of $(TX,g^{TX})$-spinors. Let $\nabla^{S^{TX}}$ denote the 
induced Clifford connection on $S^{TX}$ by $\omega^{\kk}$.

We fix $\sigma\in\Sigma$, and we assume that 
its action on $\pp$ preserves the orientation. Then $K^\sigma$ acts naturally on 
$P_{\mathrm{SO}}(X)$.
We also assume that the homomorphism $K\rightarrow
\mathrm{Spin}(\pp)$ can be extended to a homomorphism 
$K^\sigma\rightarrow
\mathrm{Spin}(\pp)$, so that the action of $K^\sigma$ on 
$P_{\mathrm{SO}}(X)$ lifts to an action on 
$P_{\mathrm{Spin}}(X)$. By \cite[Definition 14.10 in Chapter 
3]{Lawson1989spin}, this is equivalent to say that $K^\sigma$ preserves the 
spin structure.

Take $(E,\rho^E)$ a unitary representation of $K^\sigma$. Now 
$G^{\sigma}$ acts on $S^{TX}\otimes F$ over $X$ preserving $\nabla^{S^{TX}\otimes F}$.  Let $D^X$ be the classical Dirac operator acting on $C^\infty(X, S^{TX}\otimes F)$. 
If $e_1,\cdots, e_m$ is an orthogonal basis of $TX$, then
\begin{equation}\label{eq_Diracoperator}
	D^X=\sum^m_{i=1} \bar{c}(e_i) \nabla^{S^{TX}\otimes F}_{e_i}.
\end{equation}

Let $\mathcal{L}^X$ be the operator defined in 
\eqref{ellipticoperator}, with $E$ replaced by $S^\pp\otimes E$. Put
\begin{equation}
	\mathcal{A}=-\frac{1}{48} 
	\mathrm{Tr}^{\kk}[C^{\kk,\kk}]-\frac{1}{2} C^{\kk,E}\in\mathrm{End}(E).
	\label{eq:calAdirac}
\end{equation}
It is clear that $\mathcal{A}$ commutes with $K^\sigma$. Then 
by \cite[Theorem 7.2.1]{bismut2011hypoelliptic}, 
\begin{equation}
	\label{eq:5.1.13bonn}
	\frac{1}{2}D^{X,2}=\mathcal{L}^X_{\mathcal{A}}.
\end{equation}

The integral kernel of $\exp(-tD^{X,2}/2)$, $t>0$, lies in 
$\mathcal{Q}^{\sigma}$, so that, by taking the supertrace with respect to the $\Z_{2}$-grading of 
$S^{TX}$, the twisted orbital integral
$\mathrm{Tr_s}^{[\gamma\sigma]}[\exp(-t D^{X,2}/2)]$ is well-defined.

Let $\gamma\in G$ be such that $\gamma\sigma$ is semisimple. We still 
assume that 
\begin{equation}
	\label{eq:7.1.8n}
	\gamma=e^a k^{-1}, a\in\pp, k\in K, \mathrm{Ad}(k)a=\sigma a.
\end{equation}

\begin{theorem}\label{thm_nonelliptic}
	If $\gamma\sigma$ is non-elliptic, 
	i.e., if $a\neq 0$, for $Y_0^\kk\in \kk_{\sigma}(\gamma)$,
	\begin{equation}\label{eq:6.2.9n}
		\mathrm{Tr_s}^{S^\pp}\Big[\rho^{S^\pp}(k^{-1}\sigma)\exp(-i 
		\rho^{S^\pp}(Y_0^\kk))\Big]=0.
	\end{equation}
	For any $t>0$, we have
	\begin{equation}
		\mathrm{Tr_s}^{[\gamma\sigma]}[\exp(-t D^{X,2}/2)]=0.
		\label{eq:6.2.12n}
	\end{equation}
\end{theorem}

\begin{proof}
	By \cite[Proposition 3.23]{berline2003heat}, we have
	\begin{equation}
		\begin{split}
			&(-1)^{m/2}\Big(\mathrm{Tr_s}^{S^\pp}\big[\rho^{S^\pp}(k^{-1}\sigma)\exp(-i 
			\rho^{S^\pp}(Y_0^\kk))\big]\Big)^{2}=\\
			&\qquad\qquad\qquad\det 
			\left(1-\mathrm{Ad}(k^{-1}\sigma)\exp(-i\mathrm{ad}(Y^{\kk}_{0}))\right)|_{\pp}.
		\end{split}
		\label{eq:squaredet}
	\end{equation}
	
	If $a\neq 0$, then $a$ is an eigenvector in $\pp$ of 
	$\mathrm{Ad}(k^{-1}\sigma)\exp(-i\mathrm{ad}(Y^{\kk}_{0}))$ 
	associated with the eigenvalue $1$, so that \eqref{eq:6.2.9n} 
	holds.

	Note that
	\begin{equation}
		\begin{split}
			&\mathrm{Tr_s}^{S^{\pp}\otimes E}\left[\rho^{S^{\pp}\otimes 
			E}(k^{-1}\sigma)\exp\left(-i\rho^{S^{\pp}\otimes E}(Y^\kk_0)-t 
			\mathcal{A}\right)\right]\\
			&=\mathrm{Tr_s}^{S^{\pp}}\left[\rho^{S^{\pp}}(k^{-1}\sigma)\exp\left(-i\rho^{S^{\pp}}(Y^\kk_0)\right)\right]\\
			&\qquad\cdot\mathrm{Tr}^{E}\left[\rho^{
			E}(k^{-1}\sigma)\exp\left(-i\rho^{E}(Y^\kk_0)-t 
			\mathcal{A}\right)\right].
		\end{split}
		\label{eq:6.2.10nnnnn}
	\end{equation}
	Using Theorem \ref{thm_orbitalintegral}, and combining \eqref{eq:6.2.9n} and \eqref{eq:6.2.10nnnnn}, we get \eqref{eq:6.2.12n}. 
\end{proof}

\subsection{The case of elliptic $\gamma\sigma$}\label{ss:5.2elliptic}
We use the same assumptions as in Subsection 
\ref{s4.5}. Now we fix an elliptic element $\gamma\sigma$, i.e. $\gamma=k^{-1}\in K$. Recall that $p=\dim \pp_\sigma(\gamma)$.

Recall that $N_{X(\gamma\sigma)/X}$ is the normal vector bundle
of $X(\gamma\sigma)$ in $X$. Then
\begin{equation}
	\label{eq:7.2.1nn}
	TX|_{X(\gamma\sigma)}=TX(\gamma\sigma)\oplus N_{X(\gamma\sigma)/X}.
\end{equation}
Note that
\begin{equation}
	\mathrm{rank}\; TX(\gamma\sigma)=p,\; \mathrm{rank}\; N_{X(\gamma\sigma)/X}=m-p.
	\label{eq:dims}
\end{equation}
Since we assume that $\sigma$ preserves the orientation of $\pp$, 
both $p$ and $m-p$ are even. Also the action of 
$\gamma\sigma$ on 
$TX|_{X(\gamma\sigma)}$ preserves the splitting in \eqref{eq:7.2.1nn}.

Let $\widehat{A}^{\gamma\sigma}(TX|_{X(\gamma\sigma)}, 
\nabla^{TX|_{X(\gamma\sigma)}})$, 
$\widehat{A}^{\gamma\sigma}(N|_{X(\gamma\sigma)}, 
\nabla^{N|_{X(\gamma\sigma)}})$ be the equivariant 
$\widehat{A}$-genus given in \cite[Subsection 
7.7]{bismut2011hypoelliptic} for vector bundles 
$TX|_{X(\gamma\sigma)}$, $N|_{X(\gamma\sigma)}$. Note that there are 
questions of signs to be taken care of in these forms, we refer to 
\cite{AtiyahBott67,AtiyahBott68} and also \cite[Theorem 14.11 in Chapter 3]{Lawson1989spin}, \cite[Chapter 6]{berline2003heat}  for more detail.
Let $o(TX(\gamma\sigma))$, 
$o(N_{X(\gamma\sigma)/X})$ be the orientation lines of 
$TX(\gamma\sigma)$, $N_{X(\gamma\sigma)/X}$ respectively.
Because of the aforementioned sign ambiguity, $\widehat{A}^{\gamma\sigma}(TX|_{X(\gamma\sigma)}, \nabla^{TX|_{X(\gamma\sigma)}})$
should be regarded as a section of $\Lambda^\bullet 
(T^{*}X(\gamma\sigma))\otimes o(N_{X(\gamma\sigma)/X})$. Since $\pp$ 
is oriented by assumption, $\widehat{A}^{\gamma\sigma}(TX|_{X(\gamma\sigma)}, 
\nabla^{TX|_{X(\gamma\sigma)}})$ can be 
viewed as a section of $\Lambda^\bullet 
(T^{*}X(\gamma\sigma))\otimes o(TX(\gamma\sigma))$. Similarly, one can 
define $\widehat{A}^{\gamma\sigma}(N|_{X(\gamma\sigma)}, 
\nabla^{N|_{X(\gamma\sigma)}})$.

Let $\widehat{A}^{\gamma\sigma|_\pp}(0)$ be the degree 
$0$ component of $\widehat{A}^{\gamma\sigma}(TX|_{X(\gamma\sigma)}, 
\nabla^{TX|_{X(\gamma\sigma)}})$, and let
$\widehat{A}^{\gamma\sigma|_\kk}(0)$ be the degree $0$ component of 
$\widehat{A}^{\gamma\sigma}(N|_{X(\gamma\sigma)}, \nabla^{N|_{X(\gamma\sigma)}})$. These are constants on $X(\gamma\sigma)$. Put
\begin{equation}\label{eq5.65}
	\widehat{A}^{\gamma\sigma}(0)=\widehat{A}^{\gamma\sigma|_\pp}(0)\cdot\widehat{A}^{\gamma\sigma|_\kk}(0).
\end{equation}

The equivariant Chern character form of the bundle $(F,\nabla^{F})$ 
is given by 
\begin{equation}\label{eq5.15}
	\mathrm{ch}^{\gamma\sigma}\big(F|_{X(\gamma\sigma)},\nabla^{F|_{X(\gamma\sigma)}}\big)=\mathrm{Tr}\Big[\rho^E(k^{-1}\sigma) \exp\Big(-\frac{R^F|_{X(\gamma\sigma)}}{2\pi i}\Big)\Big].
\end{equation}
The above closed forms on 
$X(\gamma\sigma)$ are exactly the 
ones that appear in the Lefschetz fixed point formula of Atiyah-Bott 
\cite{AtiyahBott67,AtiyahBott68}.

Let 
$\omega^{\z_{\sigma}(\gamma)}=\omega^{\kk_{\sigma}(\gamma)}+\omega^{\pp_{\sigma}(\gamma)}$ be the 
left-invariant $1$-form  on 
$Z^0_{\sigma}(\gamma)$ valued in $\z_{\sigma}(\gamma)$.  
Let $\Omega^{\z_{\sigma}(\gamma)}$ 
be the curvature form of the connection form 
$\omega^{\kk_{\sigma}(\gamma)}$ on the principal bundle
$Z^0_{\sigma}(\gamma)\rightarrow X(\gamma\sigma)$. As in 
\eqref{eq:1.1.5n}, we have
\begin{equation}
	\Omega^{\z_{\sigma}(\gamma)}=-\frac{1}{2}[\omega^{\pp_{\sigma}(\gamma)},\omega^{\pp_{\sigma}(\gamma)}]\in \Lambda^2(\pp_{\sigma}(\gamma)^{*})\otimes 
	\kk_{\sigma}(\gamma).
	\label{eq:7.5.14bis}
\end{equation}
Then the curvatures $R^{F}$, $R^{TX}$ restricting to $X(\gamma\sigma)$ are 
represented by the equivariant actions of $\Omega^{\z_{\sigma}(\gamma)}$.

Using the same arguments as in the proof of 
\cite[Proposition 7.1.1]{bismut2011hypoelliptic} and 
\eqref{eq:7.5.14bis}, we get the 
following identities of differential forms on $X(\gamma\sigma)$,
\begin{equation}\label{eq5.69}
	\begin{split}
		&\widehat{A}^{\gamma\sigma}\big(TX|_{X(\gamma\sigma)}, 
		\nabla^{TX|_{X(\gamma\sigma)}}\big)\widehat{A}^{\gamma\sigma}\big(N|_{X(\gamma\sigma)}, \nabla^{N|_{X(\gamma\sigma)}}\big)=\widehat{A}^{\gamma\sigma}(0).\\		
		&\mathrm{ch}^{\gamma\sigma}\big(TX|_{X(\gamma\sigma)},\nabla^{TX|_{X(\gamma\sigma)}}\big)+ \mathrm{ch}^{\gamma\sigma}\big(N|_{X(\gamma\sigma)},\nabla^{N|_{X(\gamma\sigma)}}\big) =\mathrm{Tr}^\g\big[\mathrm{Ad}(k^{-1}\sigma)\big].
	\end{split}
\end{equation}

Let $\Psi$ be the canonical element of norm $1$ in 
$\Lambda^p(\pp_{\sigma}(\gamma)^*)\otimes o(\pp_{\sigma}(\gamma))$ 
(respectively, a section of norm $1$ of
$\Lambda^p(T^{*}X(\gamma\sigma))\otimes o(TX(\gamma\sigma))$). For $\alpha\in \Lambda^\bullet 
(\pp_{\sigma}(\gamma)^*)\otimes o(\pp_{\sigma}(\gamma))$ (respectively $ \Lambda^\bullet 
(T^*X(\gamma\sigma))\otimes o(TX(\gamma\sigma))$), for $0\leq l \leq p$, 
let $\alpha^{(l)}$ be the component of $\alpha$ of degree $l$. We 
define $\alpha^{\mathrm{max}}\in \R$ by 
\begin{equation}
	\alpha^{(p)}=\alpha^{\mathrm{max}} \Psi.
	\label{eq:7.2.9ssss}
\end{equation}
If $A\in\mathrm{End}(\pp_{\sigma}(\gamma))$ is antisymmetric, let 
$\mathrm{Pf}[A]$ be the Pfaffian of $A$. It is a polynomial function 
of $A$ (with values twisted by $o(\pp_{\sigma}(\gamma))$), which is a square root of $A$. The form $\omega_{A}\in 
\Lambda^{2}(\pp^{*}_{\sigma}(\gamma))$ associated with $A$ is given 
by $U,V\in \pp_{\sigma}(\gamma)\mapsto \langle U, AV\rangle$. Then 
\begin{equation}
	\mathrm{Pf}[A]=\left[\exp(\omega_{A})\right]^{\mathrm{max}}.
\end{equation}

\begin{theorem}\label{thm_localindex}
	If $\gamma=k^{-1}\in K$, for any $t>0$,
	\begin{equation}\label{localindex}
		\begin{split}
			&\mathrm{Tr_s}^{[\gamma\sigma]}[\exp(-t D^{X,2}/2)]\\
			&=\frac{1}{ (2\pi t)^{p/2}} \int_{\kk_{\sigma}(\gamma)}  
			J_{\gamma\sigma}(Y_0^\kk)\mathrm{Tr_s}^{S^\pp\otimes 
			E}\Big[\rho^{S^\pp\otimes 
			E}(k^{-1}\sigma)\exp\left(-i\rho^{S^\pp\otimes E}(Y_0^\kk)- t 
			\mathcal{A}\right)\Big]\\
			&\hspace{73mm}\cdot\exp(-|Y^\kk_0|^2/2t) \frac{dY^\kk_0}{(2\pi t)^{q/2}}\\
			&= \Big[\widehat{A}^{\gamma\sigma}\big(TX|_{X(\gamma\sigma)}, 
			\nabla^{TX|_{X(\gamma\sigma)}}\big)\mathrm{ch}^{\gamma\sigma}\big(F|_{X(\gamma\sigma)},\nabla^{F|_{X(\gamma\sigma)}}\big)\Big]^{\mathrm{max}}.
		\end{split}
	\end{equation}
\end{theorem}

\begin{proof}
	The first identity in \eqref{localindex} follows from 
	Theorem \ref{thm_orbitalintegral} and \eqref{eq:5.1.13bonn}. 
	If $(E,\rho^E)$ is an irreducible unitary representation of $K^\sigma$ which is not irreducible when 
	restricting to 
	$K$, then by \eqref{eq:1.1.5n}, \eqref{eq:Krepbonn} and \eqref{eq5.15}, we get
	\begin{equation}
		\mathrm{ch}^{\gamma\sigma}(F|_{X(\gamma\sigma)},\nabla^{F|_{X(\gamma\sigma)}})=0.
		\label{eq:5.1.43bonn}
	\end{equation}
	Then the second identity of \eqref{localindex} follows from 
	\eqref{eq:3.4.16evian}.
	
	We only need to prove the second identity in \eqref{localindex} for 
	the case where $(E,\rho^{E})$ is irreducible for both groups 
	$K^{\sigma}$ and $K$. Recall that $K_{\sigma}(1)\subset K$ is just the fixed 
	point set of $\sigma$ action on $K$. 
	
	Since $K$ is semisimple and simply connected, by \cite[Lemma 
	(3.15.4), Corollary (3.15.5)]{Duistermaat_2000}, 
	$K_{\sigma}(1)$ is a connected subgroup of $K$, and there exists $v\in 
	\kk_{\sigma}(1)$ such that $v$ is regular in $\kk$. Then 
	$\kt=\kk(v)$ is a Cartan subalgebra of $\kk$. Let $T$ be the maximal torus of $K$ 
	corresponding to $\kt$, and let $R^{+}$ 
	be the positive root system of $(K,T)$ corresponding to the Weyl 
	chamber of $v$. Then we get a decomposition of 
	$\mathrm{Aut}(K)$ as in \eqref{eq:5.1.30bonn} with respect to $(T, 
	R^{+})$. There exists $k_{0}\in T$, $\tau\in \mathrm{Out}(K)$ such that the action of 
	$\sigma$ on $K$ is given by $C(k_{0})\circ \tau$. Moreover, 
	$\ks=\kt\cap \kk_{\sigma}(1)$ is a Cartan subalgebra of $\kk_{\tau}(1)$, 
	which is just the fixed point set of $\tau$ in $\kt$. Let 
	$S$ be the corresponded maximal torus of $K_{\tau}(1)$.

	We extend 
	$\tau\in\mathrm{Out}(K)$ to a $\widehat{\tau}\in\Sigma$, so 
	that if $g\in G$, then
	\begin{equation}
		\widehat{\tau}(g)=k_{0}^{-1}\sigma(g)k_{0}\in G.
		\label{eq:tauactionG}
	\end{equation}
	Note that $\widehat{\tau}$ may not be 
	of finite order. When acting on $\g, \pp, \kk$, the adjoint action of 
	$k^{-1}\sigma$ is the same as the adjoint action of 
	$k^{-1}k_{0}\widehat{\tau}$. Following the same constructions as in the proof 
	of Proposition \ref{lm:Krepbonn}, we may assume that $E$ is an 
	irreducible representation for both $K^{\widehat{\tau}}$ and $K$. The 
	group $K^{\widehat{\tau}}$ also acts on $S^{\pp}$, so that the 
	analogue of \eqref{eq:tauvalue} holds. Then we will prove the second 
	identity in \eqref{localindex} for $k^{-1}\widehat{\tau}$ instead of 
	$k^{-1}\sigma$ with $k\in K$. By \cite[Proposition I.4]{Segal1968}, and using the fact that the both sides of the second 
	identity in \eqref{localindex} are invariant by conjugations of $K$, 
	we can continue to assume that $k\in S$. Thus $S$ is also a maximal 
	torus of $K_{\widehat{\tau}}(k^{-1})$. 
	
	Since $(E,\rho^{E})$ is an irreducible representation of $K$, then 
	its highest weight $\lambda\in P_{++}$ is fixed by $\tau$. 
	Set
	\begin{equation}
		\rho_{\kk}=\frac{1}{2}\sum_{\alpha\in R^{+}}\alpha.
		\label{eq:5.2.14cologne}
	\end{equation}
	It is also fixed by $\tau$. Then
	\begin{equation}
		\lambda,\; \rho_{\kk},\; \lambda+\rho_{\kk}\in \ks^{*}.
		\label{eq:5.2.13bs}
	\end{equation}
	By \cite[Proposition 7.5.2]{bismut2011hypoelliptic}, we have
	\begin{equation}
		\mathcal{A}=2\pi^{2} |\rho_{\kk}+\lambda|^{2}.
		\label{eq:5.2.14bs}
	\end{equation}

	As in \cite[(7.7.7)]{bismut2011hypoelliptic}, if 
	$y\in\kk_{\widehat{\tau}}(k^{-1})$,
	\begin{equation}
		\label{eq:7.3.51kk}
		\begin{split}
			&\mathrm{Tr_s}^{S^\pp}\Big[\rho^{S^\pp}(k^{-1}\widehat{\tau})\exp(-i 
			\bar{c}(\mathrm{ad}(y)))\Big]\\
			&=\mathrm{Pf}\big[\mathrm{ad}(y)|_{\pp_{\widehat{\tau}}(k^{-1})}\big]\widehat{A}^{-1}\big(i 
			\mathrm{ad}(y)|_{\pp_{\widehat{\tau}}(k^{-1})}\big)\Big(\widehat{A}^{\widehat{\tau}^{-1}k 
			e^{iy}|_{\pp^\perp_{\widehat{\tau}}(k^{-1})}}(0)\Big)^{-1}.
		\end{split}
	\end{equation}
	Then by \eqref{Jfunction}, we have
	\begin{equation}\label{eq:7.3.27n}
		\begin{split}
			&J_{k^{-1}\widehat{\tau}}(y)\mathrm{Tr_s}^{S^\pp}\Big[\rho^{S^\pp}(k^{-1}\widehat{\tau})\exp\big(-i\bar{c}(\mathrm{ad}(y))\big)\Big]\\
			&=(-1)^{\dim 
			\pp_{\widehat{\tau}}^{\perp}(k^{-1})/2}\mathrm{Pf}\big[\mathrm{ad}(y)|_{\pp_{\widehat{\tau}}(k^{-1})}\big]
			\widehat{A}^{-1}\big(i 
			\mathrm{ad}(y)|_{\kk_{{\tau}}(k^{-1})}\big)\\
			&\qquad\widehat{A}^{\widehat{\tau}^{-1}k|_{\pp_{\widehat{\tau}}^{\perp}(k^{-1})}}(0)
			\left\{ \frac{\det(1-\exp(-i\mathrm{ad}(y)) 
			\mathrm{Ad}(k^{-1}{\tau}))|_{\kk_{{\tau}}^{\perp}(k^{-1})}}
			{\det(1- 
			\mathrm{Ad}(k^{-1}{\tau}))|_{\kk_{{\tau}}^{\perp}(k^{-1})}}\right\}^{1/2}.
		\end{split}
	\end{equation}

	Combining \eqref{eq:4.2.1} with 
	\eqref{eq:7.3.27n}, we get
	\begin{equation}
		\begin{split}
			&\mathrm{Tr_s}^{[k^{-1}\widehat{\tau}]}\big[\exp(-t 
			D^{X,2}/2)\big]\\
			&=\frac{(-1)^{\dim \pp^{\perp}_{\widehat{\tau}}(k^{-1})/2}}{ (2\pi t)^{p/2}} 
			e^{-2\pi^2 t|\lambda+\rho_\kk|^2}\\
			&\int_{\kk_{{\tau}}(k^{-1})} 
			\mathrm{Pf}\big[\mathrm{ad}(y)|_{\pp_{\widehat{\tau}}(k^{-1})}\big]
			\widehat{A}^{-1}\big(i 
			\mathrm{ad}(y)|_{\kk_{{\tau}}(k^{-1})}\big)\\
			&\qquad\widehat{A}^{\widehat{\tau}^{-1}k|_{\pp^{\perp}_{\widehat{\tau}}(k^{-1})}}(0)
			\left\{ \frac{\det(1-\exp(-i\mathrm{ad}(y)) 
			\mathrm{Ad}(k^{-1}{\tau}))|_{\kk^{\perp}_{{\tau}}(k^{-1})}}
			{\det(1- 
			\mathrm{Ad}(k^{-1}{\tau}))|_{\kk^{\perp}_{\tau}(k^{-1})}}\right\}^{1/2}\\
			&\qquad\qquad\qquad\cdot\mathrm{Tr}^{E}\Big[\rho^{E}(k^{-1}{\tau})
			\exp(-i\rho^{E}(y))\Big]\exp(-|y|^2/2t) \frac{dy}{(2\pi t)^{q/2}}.
		\end{split}
		\label{eq:7.3.28n}
	\end{equation}
	
	Let $\Omega^{\z_{\widehat{\tau}}(k^{-1})}$ be the 
	curvature form associated with 
	$Z^0_{\widehat{\tau}}(k^{-1})\rightarrow X(k^{-1}\widehat{\tau})$ as an analogue of 
	$\Omega$ in \eqref{eq:1.1.5n}, when replacing $\g$ by 
	$\z_{\widehat{\tau}}(k^{-1})$. In 
	particular,
	\begin{equation}
		\Omega^{\z_{\widehat{\tau}}(k^{-1})}\in 
		\Lambda^2\big(\pp_{\widehat{\tau}}(k^{-1})^{*}\big)\otimes 
		\kk_{\tau}(k^{-1}).
		\label{eq:7.3.30bis}
	\end{equation}
	If $\alpha, \beta\in  \Lambda^{\bullet}(\pp_{\widehat{\tau}}(k^{-1})^{*})$, $a, 
	b\in\kk_{\tau}(k^{-1})$, we define
	\begin{equation}
		\langle \alpha\otimes a,\beta\otimes b\rangle'=\alpha\wedge 
		\beta \langle a,b\rangle\in 
		\Lambda^{\bullet}\big(\pp_{\widehat{\tau}}(k^{-1})^{*}\big).
		\label{eq:norm}
	\end{equation}
	A direct calculation shows that
	\begin{equation}
		\big|\Omega^{\z_{\widehat{\tau}}(k^{-1})}\big|^{\prime,2}=0.
		\label{eq:6.2.23DD}
	\end{equation}

	By \cite[(7.5.17)]{bismut2011hypoelliptic}, we have
	\begin{equation}
		\mathrm{Pf}\big[\mathrm{ad}(y)|_{\pp_{\widehat{\tau}}(k^{-1})}\big]=\Big[\exp\big(-\langle 
		y, \Omega^{\z_{\widehat{\tau}}(k^{-1})}\rangle^{\prime}\big)\Big]^{\mathrm{max}}.
		\label{eq:6.2.24DD}
	\end{equation}
	
	Let $\Delta^{\kk_{\tau}(k^{-1})}$ and $\Delta^\ks$ be the 
	standard (negative) Laplacian in $\kk_{\tau}(k^{-1})$ and $\ks$ 
	respectively. Using the integral kernel of 
	$\exp(t\Delta^{\kk_{\tau}(k^{-1})}/2)$ and 
	\eqref{eq:6.2.23DD}, \eqref{eq:6.2.24DD}, we can rewrite \eqref{eq:7.3.28n} as follows,
	\begin{equation}\label{eq5.104}
		\begin{split}
			&\mathrm{Tr_s}^{[k^{-1}\widehat{\tau}]}\big[\exp(-t 
			D^{X,2}/2)\big]\\
			&=\frac{(-1)^{\dim \pp^{\perp}_{\widehat{\tau}}(k^{-1})/2}}{ (2\pi t)^{p/2}} 
			e^{-2\pi^2 t|\lambda+\rho_\kk|^2}\\
			&\bigg[\exp\big(t\Delta^{\kk_{\tau}(k^{-1})}/2\big)\bigg\{
			\widehat{A}^{-1}\big(i 
			\mathrm{ad}(y)|_{\kk_{{\tau}}(k^{-1})}\big)\\
			&\qquad\widehat{A}^{\widehat{\tau}^{-1}k|_{\pp^{\perp}_{\widehat{\tau}}(k^{-1})}}(0)
			\Big\{ \frac{\det(1-\exp(-i\mathrm{ad}(y)) 
			\mathrm{Ad}(k^{-1}{\tau}))|_{\kk^{\perp}_{{\tau}}(k^{-1})}}
			{\det(1- 
			\mathrm{Ad}(k^{-1}{\tau}))|_{\kk^{\perp}_{\tau}(k^{-1})}}\Big\}^{1/2}\\
			&\qquad\qquad\qquad\cdot\mathrm{Tr}^{E}\Big[\rho^{E}(k^{-1}{\tau})
			\exp(-i\rho^{E}(y))\Big]\bigg\}\big(-t\Omega^{\z_{\widehat{\tau}}(k^{-1})}\big)\bigg]^{\mathrm{max}}.
		\end{split}
	\end{equation}

		Let $R'$ be the root system of $(\kk_{\tau}(k^{-1}),\ks)$ and let $R'_+$ be a positive root system in $R'$.
		Let $\pi_{\kk_{\tau}(k^{-1})}(y)$, $y\in\ks$ be the functions 
		defined as $\Pi_{\beta\in R'_+}\langle 2\pi i 
		\beta,y\rangle$. Note that the function contained in 
		$\{\cdots\}$ in \eqref{eq5.104} is invariant by adjoint 
		action of $K_{\tau}(k^{-1})$. Then, we have
	\begin{equation}\label{eq5.107}
		\begin{split}
			&\exp(t\Delta^{\kk_{\tau}(k^{-1})}/2)\big\{\cdots\big\}(y)\\
			&=\frac{1}{\pi_{\kk_{\tau}(k^{-1})}(y)}\exp(t\Delta^{\ks}/2)\Big[\pi_{\kk_{\tau}(k^{-1})}(y)\big\{\cdots\big\}\Big](y).
		\end{split}
	\end{equation}
	The function in the right-hand side of \eqref{eq5.107} is
	viewed as a function in $y\in\ks$, which is 
	invariant by the Weyl group $W(K^0_{\tau}(k^{-1}), S)$, 
	and lifts to a central function on $\kk_{\tau}(k^{-1})$. 
	Then it can be evaluated 
	at $-t\Omega^{\z_{\widehat{\tau}}(k^{-1})}$.
	
	If $\alpha\in R^{+}$, let $\kk_{\alpha}\subset \kk_{\bbC}$ be the 
	associated root space. Put
	\begin{equation}
		\kn=\sum_{\alpha\in R^{+}} \kk_{\alpha}.
		\label{eq:5.2.15bs}
	\end{equation}
	Then $\tau$ preserves $\kn$. For $t\in T$, set 
	$\delta(t\tau)=\det(1-\mathrm{Ad}(\tau^{-1}t^{-1}))|_{\kn}$.
	Note that up to multiplication by some constant, for $y\in\ks$, the 
	analytic function $e^{2\pi\langle\rho_{\kk},y\rangle}\delta(e^{-iy} k^{-1}\tau 
	)$ coincides with 
	\begin{equation}\label{eq:7.3.39n}
		\begin{split}
			\pi_{\kk_{\tau}(k^{-1})}(y) \widehat{A}^{-1}\big(i 
			\mathrm{ad}(y)|_{\kk_{{\tau}}(k^{-1})}\big)
			\Big[\det(1-\exp(-i\mathrm{ad}(y)) 
			\mathrm{Ad}(k^{-1}{\tau}))|_{\kk^{\perp}_{{\tau}}(k^{-1})}\Big]^{1/2}.
		\end{split}
	\end{equation}
	
	A Weyl character formula for the non-connected compact Lie group 
	$K^{\tau}$ was given in \cite[Section 4.13, Proposition 
	4.13.1]{Duistermaat_2000}, we apply it to the $K^{\tau}$-representation 
	$(E,\rho^E)$, we get that for $y\in \mathfrak{s}$,
	\begin{equation}\label{eq5.113}
		\begin{split}
			&e^{2\pi \langle \rho_{\kk},y\rangle}\delta(k^{-1}{\tau} 
			e^{-iy})\mathrm{Tr}^E\big[k^{-1}\tau \exp(-iy)\big]\\
			&=\sum_{\omega\in W(\tau)} \det(\omega) 
			\det\big(\mathrm{Ad}(\tau^{-1}k)\big)|_{\kk_{R_+\backslash \omega\cdot R_+}}
			\mathrm{Tr}\big[\rho^E(k^{-1}\tau)|_{E_{\omega\cdot 
			\lambda}}\big] 
			e^{2\pi \langle \omega\cdot(\rho_{\kk}+\lambda), y\rangle},
		\end{split}
	\end{equation}
	where $W(\tau)$ is a subgroup of Weyl group $W(K,T)$, and we 
	refer to \cite[Section 4.13]{Duistermaat_2000} for the precise 
	meaning of other notation. We only use this formula to make an 
	observation 
	that, by \eqref{eq:5.2.13bs}, \eqref{eq:7.3.39n} and 
	\eqref{eq5.113}, the function 
	$\pi_{\kk_{\tau}(k^{-1})}(y)\big\{\cdots\big\}$, $y\in\ks$, in 
	the right-hand side of \eqref{eq5.107},
	is an eigenfunction of $\Delta^{\ks}$ associated with the eigenvalue 
	$2\mathcal{A}=4\pi^{2}|\rho_{\kk}+\lambda|^{2}$.

	Then by \eqref{eq5.104}, we get
	\begin{equation}\label{eq:7.3.70kk}
		\begin{split}
			&\mathrm{Tr_s}^{[k^{-1}\widehat{\tau}]}[\exp(-t D^{X,2}/2)]\\
			&=\frac{1}{ (2\pi 
			t)^{p/2}}\left[\widehat{A}^{k^{-1}\widehat{\tau}}(0) 
			\Big(\widehat{A}^{k^{-1}\widehat{\tau}}\Big)^{-1}\big(i\mathrm{ad}(t\Omega^{\z_{\widehat{\tau}}(k^{-1})})|_{\kk}\big)\mathrm{Tr}^E\big[k^{-1}\tau e^{it\Omega^{\z_{\widehat{\tau}}(k^{-1})}}\big]
			\right]^{\mathrm{max}}.
		\end{split}
	\end{equation}
	Also the parameter $t$ is killed automatically in the right-hand side of \eqref{eq:7.3.70kk}.

	Note that the curvatures $R^{N|_{X(k^{-1}\widehat{\tau})}}$, $R^{TX|_{X(k^{-1}\widehat{\tau})}}$, 
	$R^F|_{X(k^{-1}\widehat{\tau})}$ are 
	given by the actions of the curvature form 
	$\Omega^{\z_{\widehat{\tau}}(k^{-1})}$ associated with 
	$Z^0_{\widehat{\tau}}(k^{-1})\rightarrow 
	X(k^{-1}\widehat{\tau})$. Then the right-hand side in \eqref{eq:7.3.70kk} is just
	\begin{equation}
		\label{eq:7.3.71kk}
		\begin{split}
			&\bigg[\widehat{A}^{k^{-1}\widehat{\tau}}(0) 
		\Big(\widehat{A}^{k^{-1}\widehat{\tau}}\Big)^{-1}\big(N|_{X(k^{-1}\widehat{\tau})}, 
		\nabla^{N|_{X(k^{-1}\widehat{\tau})}}\big)\\
		&\qquad\qquad \cdot\mathrm{Tr}^E\Big[\rho^E(k^{-1}\tau)
		\exp\Big(-\frac{R^F|_{X(k^{-1}\widehat{\tau})}}{2\pi i}\Big)\Big]
		\bigg]^{\mathrm{max}}
		\end{split}
	\end{equation}
	Then by \eqref{eq5.15}, \eqref{eq5.69}, \eqref{eq:7.3.71kk}, we 
	get the second identity in \eqref{localindex} for the semisimple 
	element $k^{-1}\widehat{\tau}$. This completes the proof of our theorem.
\end{proof}

\subsection{The local equivariant index theorem on $Z$}
\label{s4.8}
In this subsection, we make the same assumptions as in Subsections \ref{s1.9}, 
\ref{section4.9} and \ref{s4.5}. In particular, $\Gamma$ is a cocompact 
torsion-free discrete subgroup of $G$ such that 
$\sigma(\Gamma)=\Gamma$. Then $Z=\Gamma\backslash X$ is a compact 
manifold on which $\Sigma^{\sigma}$ acts isometrically. The bundle $S^{TX}$ 
descends to the bundle of $TZ$-spinors $S^{TZ}$. The assumptions 
in Subsection \ref{s4.5} make $S^{TZ}\otimes F$ an 
equivariant Clifford module over $Z$ equipped with the equivariant action of 
$\Sigma^{\sigma}$.

The operator $D^X$ descends to the Dirac operator 
$D^Z$, which acts on $C^\infty(Z, S^{TZ}\otimes F)$ and commutes with 
$\Sigma^\sigma$, and the operator $\mathcal{L}^X_{\mathcal{A}}$ descends to $\mathcal{L}^Z_{\mathcal{A}}$. By \eqref{eq:5.1.13bonn},
\begin{equation}
	\label{eq:Zd}
	\frac{1}{2}D^{Z,2}=\mathcal{L}^Z_{\mathcal{A}}.
\end{equation}

Let $\ker D^Z \subset C^\infty(Z, 
S^{TZ}\otimes F)$ be the kernel of $D^Z$, which is a finite-dimensional representation of 
$\Sigma^\sigma$. 
The equivariant index of $D^Z$ (or the Lefschetz number) associated with 
$\sigma$ is defined as
\begin{equation}
	\mathrm{Ind}_{\Sigma^\sigma}\big(\sigma, D^Z\big)=\mathrm{Tr_{s}}^{\ker 
	D^Z}[\sigma].
	\label{eq:7.4.2waw}
\end{equation}
By McKean-Singer formula 
(\cite{McKS67}, \cite[Proposition 6.3]{berline2003heat}), for $t>0$,
\begin{equation}
	\mathrm{Ind}_{\Sigma^\sigma}\big(\sigma, 
	D^Z\big)=\mathrm{Tr_s}\Big[\sigma^Z\exp\big(-tD^{Z,2}/2\big)\Big].
	\label{eq:7.4.2wawbonn}
\end{equation}

Recall that ${}^\sigma Z\subset Z$ is the fixed point set of 
$\sigma$. By \eqref{eq:1.9.21mm}, it is a finite disjoint union of 
$[X(\gamma\sigma)]_{Z}$, $\underline{[\gamma]}_{\sigma}\in \underline{E}_{\sigma}$.
Let $\widehat{A}^{\sigma}(TZ|_{^\sigma Z}, \nabla^{TZ|_{^\sigma 
Z}})$, $\mathrm{ch}^{\sigma}\big(F|_{^\sigma Z},\nabla^{F|_{^\sigma Z}}\big)$ be the closed differential 
forms on ${}^\sigma Z$ defined as in Subsection \ref{ss:5.2elliptic}.

By \cite{AtiyahBott67,AtiyahBott68} and \cite[Theorem 14.11 in Chapter 3]{Lawson1989spin}, $\mathrm{Ind}_{\Sigma^\sigma}(\sigma, D^Z)$ 
can be computed by the Lefschetz fixed point formula of Atiyah-Bott, so that
\begin{equation}\label{eq_ASSindex}
	\mathrm{Ind}_{\Sigma^\sigma}\big(\sigma,D^Z\big)=\int_{{}^\sigma 
	Z} \widehat{A}^{\sigma}\big(TZ|_{^\sigma Z}, 
	\nabla^{TZ|_{{}^\sigma 
	Z}}\big)\mathrm{ch}^{\sigma}\big(F|_{^\sigma Z},\nabla^{F|_{^\sigma Z}}\big).
\end{equation}

By Proposition \ref{prop:1.9.6}, if $\underline{[\gamma]}_{\sigma}\in 
\underline{E}_{\sigma}$, the action of $\sigma$ on 
$S^{TZ}\otimes F|_{[X(\gamma\sigma)]_{Z}}$ is equivalent to the action of 
$k^{-1}\sigma$ on the corresponding vector bundle $S^{TX}\otimes F$ 
over $\Gamma\cap Z(k^{-1}\sigma)\backslash X(k^{-1}\sigma)$.  
Then on each component $[X(\gamma\sigma)]_{Z}$ of ${}^\sigma Z$, the 
following function is constant,
\begin{equation}
	\Big[\widehat{A}^{\sigma}\big(TZ|_{^\sigma Z}, 
	\nabla^{TZ|_{{}^\sigma 
	Z}}\big)\mathrm{ch}^{\sigma}\big(F|_{^\sigma Z},\nabla^{F|_{^\sigma Z}}\big)\Big]^{\mathrm{max}}
	\label{eq:7.4.4waw}
\end{equation}
and it is equal to
\begin{equation}
	\Big[\widehat{A}^{k^{-1}\sigma}\big(TX|_{X(k^{-1}\sigma)}, 
	\nabla^{TX|_{X(k^{-1}\sigma)}}\big)\mathrm{ch}^{k^{-1}\sigma}\big(F|_{X(k^{-1}\sigma)},\nabla^{F|_{X(k^{-1}\sigma)}}\big)\Big]^{\mathrm{max}}.
	\label{eq:7.4.5waw}
\end{equation}

Then by \eqref{eq:tracebonn} and using Theorems 
\ref{thm_nonelliptic}, \ref{thm_localindex}, 
we get
\begin{equation}\label{eq:7.4.6mm}
	\mathrm{Tr_s}\big[\sigma^Z e^{-tD^{Z,2}/2}\big]
	=\sum_{\underline{[\gamma]}_{\sigma}\in 
	\underline{E}_{\sigma}}\int_{[X(\gamma\sigma)]_{Z}}\widehat{A}^{\sigma}\big(TZ|_{^\sigma Z}, \nabla^{TZ|_{{}^\sigma Z}}\big)\mathrm{ch}^{\sigma}\big(F|_{^\sigma Z},\nabla^{F|_{^\sigma Z}}\big).
\end{equation}
By \eqref{eq:7.4.10bb} and \eqref{eq:7.4.2wawbonn}, we see that \eqref{eq:7.4.6mm} is equivalent to \eqref{eq_ASSindex}.

\subsection{The de Rham operator associated with a flat bundle}\label{s7.7}
From now on, we assume that $G$ is a connected linear reductive Lie group 
with compact center. Then the center 
$\z_{\g}$ of $\g$ is included in $\kk$. We do not assume anymore that $K$ is semisimple or 
simply connected. We do not assume that $\sigma$ preserves the 
orientation of $\pp$ either.

Put
\begin{equation}
	\g_{\bbC}=\g\otimes_{\R}\bbC,\; \ku=\sqrt{-1}\pp\oplus \kk.
	\label{eq:compactformb}
\end{equation}
Then $\ku$ is a real Lie algebra, which is called the compact form of 
$\g$. It is clear that
\begin{equation}
	\ku_{\bbC}=\g_{\bbC}.
	\label{eq:5.5.2bs}
\end{equation}
The form $B$ extends to an invariant negative definite bilinear form on 
$\ku$ and to an invariant $\bbC$-bilinear form on $\g_{\bbC}$. Let 
$G_{\bbC}$ be the connected group of complex matrices associated with 
$\g_{\bbC}$, and let $U$ be the analytic subgroup of $G_{\bbC}$ 
associated with $\ku$. Since $G$ has compact center, by 
\cite[Proposition 5.3]{KnappRep1986}, $U$ is a compact Lie group. By \cite[Proposition 5.6]{KnappRep1986}, 
$G_\bbC$ is still reductive, 
and $G$, $U$ are closed subgroups of $G_\bbC$. 
In particular, $U$ is a maximal compact subgroup of $G_\bbC$.

Let $U\ku$, $U\g_\bbC$ be the enveloping algebras of $\ku$, $\g_\bbC$ respectively.
Then $U\g_{\bbC}$ can 
be identified with the left-invariant holomorphic differential 
operators on $G_{\bbC}$. 
Let $C^\ku$ be the Casimir operator of $U$ associated with $B$, by \eqref{eq:3.3.3gg}, we have
\begin{equation}
	C^\ku=C^\g\in U\g\cap U\ku.
	\label{eq:7.7.1hh}
\end{equation}

We extend the action $\sigma$ to $\g_{\bbC}$ as a complex linear 
isomorphism of $\g_{\bbC}$. We assume that $\sigma$ extends to an automorphism of $U$, then it 
also acts on $G_{\bbC}$ holomorphically. Set
\begin{equation}
	U^{\sigma}= U\rtimes \Sigma^{\sigma}.
	\label{eq:5.5.3bs}
\end{equation}

In the sequel, we fix a $(E,\rho^{E})\in\mathrm{Irr}(U^{\sigma})$ 
with an invariant Hermitian metric $h^{E}$. By Weyl's unitary trick 
\cite[Proposition 5.7]{KnappRep1986}, it extends uniquely 
to an irreducible representation of $G^{\sigma}$. We use the same notation $\rho^E$ for the restrictions of 
this representation to $G$, to $K$ and to $K^\sigma$. 
By \eqref{eq:7.7.1hh}, we have
\begin{equation}
	C^{\ku,E}=C^{\g,E}\in \mathrm{End}(E).
	\label{eq:7.7.2hh}
\end{equation}

Put $F=G\times_K E$. Let $\nabla^F$ be the Hermitian connection 
induced by the connection form $\omega^\kk$. Then the map $(g,v)\in 
G\times_K E\rightarrow \rho^E(g)v\in E$ gives a canonical 
identification of vector bundles on $X$,
\begin{equation}
	G\times_K E=X\times E.
	\label{eq:5.5.8bs}
\end{equation}
Then $F$ is equipped with a canonical flat connection $\nabla^{F,f}$ 
so that
\begin{equation}
	\nabla^{F,f}=\nabla^{F}+ \rho^E(\omega^\pp).
	\label{eq:5.5.9bs}
\end{equation}

Let $(\Omega_c^\bullet(X,F), d^{X,F})$ be the (compactly supported) de Rham 
complex associated with $(F, \nabla^{F,f})$. 
Let $d^{X,F,*}$ be the adjoint operator of $d^{X,F}$ 
with respect to the $L_2$ metric on $\Omega_c^\bullet(X,F)$. 
The Dirac operator $\mathbf{D}^{X,F}$ of this de Rham complex is given by
\begin{equation}\label{eq5.90}
	\mathbf{D}^{X,F}=d^{X,F}+d^{X,F,*}.
\end{equation}

As in \eqref{eq:actionscc}, $c(TX)$, 
$\widehat{c}(TX)$ act on $\Lambda^\bullet(T^*X)$. 
We still use $e_1$, $\cdots$, $e_m$ to denote an orthonormal basis of 
$\pp$ or $TX$, and let $e^1$, $\cdots$, $e^m$ be the corresponding 
dual basis of $\pp^{*}$ or $T^*X$. Let $\nabla^{\Lambda^\bullet(T^*X)\otimes F,u}$ be the 
connection on $\Lambda^\bullet(T^*X)\otimes F$ 
induced by $\nabla^{TX}$ and $\nabla^F$. 
Then the standard Dirac operator is given by
\begin{equation}
	D^{X,F}=\sum^m_{j=1} c(e_j)\nabla^{\Lambda^\bullet(T^*X)\otimes F,u}_{e_j}.
	\label{eq:5.5.14bs}
\end{equation}
By \cite[(8.42)]{BMZ2015toeplitz}, we have
\begin{equation}
	\mathbf{D}^{X,F}= D^{X,F}+ \sum^m_{j=1} \widehat{c}(e_j)\rho^E(e_j).
	\label{eq:5.5.15bs}
\end{equation}

Note that $C^{\g,E}$ defines an invariant parallel section of endomorphism of $F$. 
Recall that the operator $\mathcal{L}^{X}$ acting on 
$\Omega^{\bullet}(X,F)$ is defined as in \eqref{ellipticoperator}.
By \cite[Proposition 8.4]{BMZ2015toeplitz} and \eqref{eq:4.4.5nm}, 
\eqref{ellipticoperator}, we have
\begin{equation}\label{eq:7.7.7}
	\begin{split}
	\frac{\mathbf{D}^{X,F, 2}}{2}&= \mathcal{L}^{X}-
	\frac{1}{2}C^{\g,E}-\frac{1}{8} B^*(\kappa^\g,\kappa^\g)\\
	&=\frac{1}{2}C^{\g,X}-
	\frac{1}{2}C^{\g,E}.
	\end{split}
\end{equation}
Moreover, $\mathbf{D}^{X,F,2}$ commutes with the action of 
$G^\sigma$.

The real rank (resp. complex rank) $\mathrm{rk}_{\R}G$ (resp. 
$\mathrm{rk}_{\bbC}G$ ) of $G$ is defined as the real dimension of 
the maximal abelian subspace of $\pp$ (resp. the Cartan subalgebra of 
$\g$). The fundamental rank of $G$ is defined as
\begin{equation}
	\delta(G)=\mathrm{rk}_{\bbC} G -\mathrm{rk}_{\bbC} K\in \mathbb{N}.
	\label{eq:5.5.16bs}
\end{equation}

We still assume that $\gamma\sigma$ 
is a semisimple element given by \eqref{eq:7.1.8n}. As explain in 
Remark \ref{rk:new2022}, $Z^{0}_{\sigma}(\gamma)$ is real reductive 
equipped with a Cartan involution $\theta|_{Z^{0}_{\sigma}(\gamma)}$.
Let $S$ be a 
maximal torus of $K^{0}_{\sigma}(\gamma)$ with Lie algebra 
$\ks\subset\kk_{\sigma}(\gamma)$. Set
\begin{equation}
	\kb_{\sigma}(\gamma)=\{f\in\pp_{\sigma}(k^{-1})\;:\; [f,\ks]=0\}.
	\label{eq:5.5.17bonn}
\end{equation}
Then
\begin{equation}
	a\in \kb_{\sigma}(\gamma),\;\dim_{\R}\kb_{\sigma}(\gamma)\geq 
	\delta(Z^{0}_{\sigma}(\gamma)).
	\label{eq:5.5.18bonn}
\end{equation}
The quantity $\dim_{\R}\kb_{\sigma}(\gamma)$ only depends on the 
$\sigma$-conjugacy class of $\gamma$ in $G$. If $\gamma\sigma$ is 
elliptic, then $\dim_{\R}\kb_{\sigma}(\gamma)=
	\delta(Z^{0}_{\sigma}(\gamma))$.

Let $e\big(TX(\gamma\sigma),\nabla^{TX(\gamma\sigma)}\big)$ be the Euler form 
of $TX(\gamma\sigma)$ associated with the Levi-Civita 
connection $\nabla^{TX(\gamma\sigma)}$. If $\dim\pp_{\sigma}(\gamma)$ 
is even, then
\begin{equation}
	e\big(TX(\gamma\sigma),\nabla^{TX(\gamma\sigma)}\big)=\mathrm{Pf}\bigg[\frac{R^{TX(\gamma\sigma)}}{2\pi}\bigg].
	\label{eq:5.5.19bonn}
\end{equation}
If $\dim \pp_{\sigma}(\gamma)$ is odd, then 
$e\big(TX(\gamma\sigma),\nabla^{TX(\gamma\sigma)}\big)$ vanishes 
identically.

Recall that the notation $[\cdot]^{\mathrm{max}}$ refers 
to the forms on $X(\gamma\sigma)$. The following theorem extends 
\cite[Theorem 7.8.2]{bismut2011hypoelliptic}.

\begin{theorem}\label{thm:7.7.1hh}
	For $t>0$, the following identity holds:
	\begin{equation}\label{eq:7.7.8}
		\begin{split}
			&\mathrm{Tr_s}^{[\gamma\sigma]}\big[\exp(-t\mathbf{D}^{X,F,2}/2)\big]\\
			&=\frac{\exp(-|a|^2/2t)}{(2\pi t)^{p/2}} 
			\exp\Big(\frac{t}{8} B^*(\kappa^\g,\kappa^\g)\Big)\int_{\kk_{\sigma}(\gamma)} 
			J_{\gamma\sigma}(Y^\kk_0)\\
			&\qquad\mathrm{Tr_s}^{\Lambda^\bullet(\pp^*)\otimes E}
			\Big[\rho^{\Lambda^\bullet(\pp^*)\otimes E}(k^{-1}\sigma)
			\exp\big(-i\rho^{\Lambda^\bullet(\pp^*)\otimes 
			E}(Y^\kk_0)+\frac{t}{2}C^{\g,E}\big)\Big]\\
			&\hspace{40mm}\exp\big(-|Y^\kk_0|^2/2t\big) 
			\frac{dY^\kk_0}{(2\pi t)^{q/2}}.
		\end{split}
	\end{equation}
	
	If $\dim \kb_{\sigma}(\gamma)\geq 1$, then 
	\begin{equation}
		\mathrm{Tr_s}^{[\gamma\sigma]}\big[\exp(-t\mathbf{D}^{X,F,2}/2)\big]=0.
		\label{eq:7.7.9bis}
	\end{equation}
	If $\gamma\sigma$ is elliptic, then
	\begin{equation}
		\mathrm{Tr_s}^{[\gamma\sigma]}\big[\exp(-t\mathbf{D}^{X,F,2}/2)\big]=\left[e\big(TX(\gamma\sigma),\nabla^{TX(\gamma\sigma)}\big)\right]^{\mathrm{max}} 
		\mathrm{Tr}^E\big[\rho^E(\gamma\sigma)\big].
		\label{eq:7.7.10bis}
	\end{equation}
\end{theorem}

\begin{proof}
	The identity in \eqref{eq:7.7.8} follows from \eqref{eq:4.4.5nm}, 
	\eqref{eq:4.2.1}, \eqref{eq:7.7.7}. As in \eqref{eq:6.2.10nnnnn}, the integrand in \eqref{eq:7.7.8} 
	contains the following factor
	\begin{equation}\label{eq:7.5.8mm}
		\mathrm{Tr_{s}}^{\Lambda^\bullet(\pp^{*})}\left[\rho^{\Lambda^\bullet(\pp^{*})}(k^{-1}\sigma)e^{-i\rho^{\Lambda^\bullet(\pp^{*})}(Y^\kk_{0})}\right]=\det 
		\big(1-\exp(i\mathrm{ad}(Y^\kk_{0}))\mathrm{Ad}(\sigma^{-1} 
		k)\big)|_{\pp}.
	\end{equation}
	If $\dim \kb_{\sigma}(\gamma)\geq 1$, then the right-hand side in 
	\eqref{eq:7.5.8mm} vanishes identically for $Y^\kk_{0}\in 
	\kk_{\sigma}(\gamma)$. Then \eqref{eq:7.7.9bis} follows.

	Now take $\gamma=k^{-1}\in K$. Then 
	\begin{equation}
		\mathfrak{b}_{\sigma}(\gamma)\subset \pp_{\sigma}(\gamma).
		\label{eq:7.5.13mm}
	\end{equation}
	Moreover, by \cite[pp.129]{KnappRep1986}, 
	$\mathfrak{b}_{\sigma}(\gamma)\oplus\ks$ 
	is a Cartan subalgebra of $\z_{\sigma}(\gamma)$. In this case, 
	$\dim \pp_{\sigma}(\gamma)-\dim \kb_{\sigma}(\gamma)$ is even. Note 
	that $\Omega^{\z_{\sigma}(\gamma)}$ is the curvature form given in 
	\eqref{eq:7.5.14bis}.

	By \eqref{eq:7.5.14bis}, as an analogue of \eqref{eq5.69}, we have the following identities
	\begin{equation}
		\begin{split}
			&\widehat{A}^{-1}\big(i\mathrm{ad}(-t\Omega^{\z_{\sigma}(\gamma)})|_{\z_{\sigma}(\gamma)}\big)
			\left[\frac{
			\det(1-e^{-i\mathrm{ad}(-t\Omega^{\z_{\sigma}(\gamma)})}\mathrm{Ad}(k^{-1}\sigma))_{\z^\perp_{\sigma}(\gamma)}}
			{\det(1-\mathrm{Ad}(k^{-1}\sigma))_{\z^\perp_{\sigma}(\gamma)}}\right]^{1/2}=1,\\
			&	
			\mathrm{Tr}^E\left[\rho^{E}(k^{-1}\sigma)\exp(-i\rho^{E}(-t\Omega^{\z_{\sigma}(\gamma)}))\right]=\mathrm{Tr}^{E}\big[\rho^{E}(k^{-1}\sigma)\big].
		\end{split}
		\label{eq:7.5.27bis}
	\end{equation}
	
	Note 
	that if 
	$\dim \kb_{\sigma}(\gamma)\geq 1$, if 
	$Y^{\kk}_{0}\in\kk_{\sigma}(\gamma)$, then
	\begin{equation}
		\mathrm{Pf}\big[\mathrm{ad}(Y^{\kk}_{0})\big]=0.
		\label{eq:7.5.14nn}
	\end{equation}
	By \eqref{eq:7.5.14bis}, \eqref{eq:5.5.19bonn}, \eqref{eq:7.5.14nn}, 
	we get that $e\big(TX(\gamma\sigma),\nabla^{TX(\gamma\sigma)}\big)=0$. 
	Then \eqref{eq:7.7.9bis} is compatible with \eqref{eq:7.7.10bis}. 
	We only 
	need to consider the case where $\dim \kb_{\sigma}(\gamma)=0$, so 
	that $\ks$ is also a Cartan subalgebra of $\z_{\sigma}(\gamma)$.
	
	If we make the same assumptions on $K$, $\pp$ and $\sigma$ as in Subsection 
	\ref{s4.5}, then \eqref{eq:7.7.10bis} is a special 
	case of Theorem \ref{thm_localindex}. In general, we can proceed 
	as in the proof of Theorem \ref{thm_localindex} with the group 
	$U$ instead of $K$. Note that the Lie algebra of $\mathrm{Aut}(U)$ is isomorphic to 
	$[\ku,\ku]$. By \cite[Lemma (3.15.4)]{Duistermaat_2000}, if 
	$\sigma\in\mathrm{Aut}(U)$, then $[\ku,\ku](\sigma)$ contains regular 
	elements in $[\ku,\ku]$, so that there always exists $v\in 
	\ku(\sigma)\cap \ku^{\mathrm{reg}}$. Then we fix the corresponding 
	maximal torus and a positive root system $R^{+}$ for $U$ as in the proof 
	of Theorem \ref{thm_localindex}. Let 
	$\rho_{\ku}$ denote the element defined as in 
	\eqref{eq:5.2.14cologne}.
	
	We may suppose that $(E,\rho^{E})$ is irreducible for both $U$ 
	and $U^{\sigma}$, so that $C^{\g,E}$ is scalar. Let 
	$\lambda$ be the highest weight for this $U$-representation. By \cite[Proposition 
	7.5.2]{bismut2011hypoelliptic}, we have
	\begin{equation}
		-C^{\g,E}-\frac{1}{4}B^{*}(\kappa^\mathfrak{g},\kappa^\mathfrak{g})= 4\pi^{2}|\rho_{\ku}+\lambda|^{2}.
		\label{eq:6.4.B}
	\end{equation}

	Based on the above constructions, the arguments 
	in the proof of Theorem \ref{thm_localindex} still work without 
	assuming $U$ to be semisimple or simply connected. Using instead 
	\eqref{eq:7.5.27bis} and \eqref{eq:6.4.B}, we can prove 
	\eqref{eq:7.7.10bis} in full generality. This completes the proof of 
	our theorem.
\end{proof}

\subsection{Twisted $L_2$-torsion}
Following the idea in last subsection, our formula for twisted 
orbital integrals is quite promising in studying the equivariant real
analytic torsions for compact locally symmetric spaces. We refer to 
another publication of the author \cite{LIU2021109117} for a detailed 
investigation on 
this topic. Here, 
we give a brief discussion on the twisted 
$L_{2}$-torsion introduced by Bergeron and Lipnowski \cite{BeLip2017}.

We make the same assumptions as in Subsections \ref{s1.9}, 
\ref{section4.9} and \ref{s7.7}. In particular, $G$ is linear 
reductive and with compact center, $\Gamma$ is a cocompact 
torsion-free discrete subgroup of $G$ such that 
$\sigma(\Gamma)=\Gamma$. Let $(E,\rho^{E})$ be an 
irreducible unitary representation for both $U$ and $U^{\sigma}$. 
Furthermore, we make an assumption on the representation 
$(E,\rho^{E})$: as $G$-representations, 
$(E,\rho^{E})\ncong (E,\rho^{E}\circ\theta)$. By \cite[\S VI, Theorem 
5.3]{MR1721403} and \cite[Lemma 4.1]{BV2013torsion}, the flat vector 
bundle $F\rightarrow Z=\Gamma\backslash X$ is (strongly) acyclic.

Let 
$\bar{\sigma}\in\mathrm{Aut}(\Gamma)$ be the induced isomorphism of 
$\sigma$. Then $\bar{\sigma}$ is of finite order $N_{0}\in \bN^{*}$ 
(since $\Gamma$ is always finitely generated). 
\begin{lemma}
	The action of $\sigma^{N_{0}}$ on $X$ is the identity map. Then 
	for $\gamma\in\Gamma$, $\gamma\sigma$ is elliptic if and only if 
	$(\gamma\bar{\sigma})^{N_{0}}=1$.
\end{lemma}
\begin{proof}
	The first statement is equivalent to that $\sigma^{N_{0}}$ acts on $\pp$ 
	as identity. In fact, if a nonzero $a\in \pp$ is such that 
	$a'=\sigma^{N_{0}}(a)$, then the function $t\in\R\mapsto 
	d(pe^{ta},pe^{ta'})$ is either constant $0$ or tending to 
	infinity as $t\rightarrow +\infty$ (cf. \cite[Proposition 1.4.1, 
	pp.19]{eberlein1996geometry}). For any $\gamma\in\Gamma$, we 
	have $d(pe^{ta},\gamma)=d(pe^{ta'},\gamma)$ since 
	$\sigma^{N_{0}}$ fixes $\gamma$. Then $d(pe^{ta},pe^{ta'})\leq 
	2d(pe^{ta},\gamma)$, the cocompactness of $\Gamma$ 
	infers that $d(pe^{ta},pe^{ta'})$ is bounded hence must be constant $0$, 
	we conclude that $a=a'=\sigma^{N_{0}}(a)$.
	
	For $\gamma\in\Gamma$, $\gamma\sigma$ is semisimple, we can take 
	$g\in G$ such that $\sigma^{N_{0}}(g)=g$, and
	$\gamma=ge^{a}k^{-1}\sigma(g^{-1})$ where $a\in\pp$, $k\in K$ and 
	$\mathrm{Ad}(k)a=\sigma(a)$. Then the second part follows 
	directly from 
	$(\gamma\sigma)^{N_{0}}=ge^{N_{0}a}k^{-1}\sigma(k^{-1})\cdots\sigma^{N_{0}-1}(k^{-1})g^{-1}\sigma^{N_{0}}$.	
\end{proof}

Recall that $\underline{E}_{\sigma}$ denotes the finite set of elliptic 
classes in $[\Gamma]_{\sigma}=[\Gamma]_{\bar{\sigma}}$. Set 
$Z^{1}(\bar{\sigma},\Gamma)=\{\gamma \in \Gamma\;:\; 
(\gamma\bar{\sigma})^{N_{0}}=1\in\Gamma\}$, and let 
$H^{1}(\bar{\sigma},\Gamma)$ denote the quotient of 
$Z^{1}(\bar{\sigma},\Gamma)$ by the equivalent relation defined by 
the $\sigma$-conjugation by elements in $\Gamma$. The above lemma 
implies the identification
\begin{equation}
	H^{1}(\bar{\sigma},\Gamma)=\underline{E}_{\sigma}.
\end{equation}

Let $N^{\Lambda^\bullet (\pp^*)}$, $N^{\Lambda^\bullet (T^*X)}$ be the 
number operators of 
$\Lambda^\bullet (\pp^*)$, $\Lambda^\bullet (T^*X)$. For 
$\underline{[\gamma]}_{\sigma}\in \underline{E}_{\sigma}$, $t>0$, set
\begin{equation}
	\cE_{X,\gamma\sigma}(F,t)=\mathrm{Tr_s}^{[\gamma\sigma]}\Big[\big(N^{\Lambda^\bullet(T^*X)}-\frac{m}{2}\big)\exp\big(-t\mathbf{D}^{X,F,2}/2\big)\Big].
	\label{eq:6.3.1pl}
\end{equation}

Since $\gamma\sigma$ is elliptic, there exist 
$g\in G$ such that $k=g\gamma\sigma(g^{-1})\in K$.
Let $\lambda$ still denote the highest weight for the 
$U$-representation $\rho^{E}$ as in the proof of Theorem 
\ref{thm:7.7.1hh}. 
By \eqref{eq:4.2.1}, \eqref{eq:6.4.B}, we have
 	\begin{equation}
 		\begin{split}
 			\cE_{X,\gamma\sigma}(F,t)&=\frac{1}{(2\pi t)^{p/2}} \exp\big(-2\pi^{2} 
 			t|\lambda+\rho_{\ku}|^{2}\big)\\
 			&\cdot\int_{\kk_{\sigma}(k)} 
 			J_{k\sigma}(Y^\kk_0)\mathrm{Tr_s}^{\Lambda^\bullet(\pp^*)}\Big[\big(N^{\Lambda^\bullet(\pp^*)}-\frac{m}{2}\big)\rho^{\Lambda^\bullet(\pp^{*})}(k\sigma)e^{-i\rho^{\Lambda^\bullet(\pp^{*})}(Y^\kk_0)}\Big]\\
 			& 
 			\qquad\quad\mathrm{Tr}^{E}\Big[\rho^{E}(k\sigma)\exp(-i\rho^{E}(Y^\kk_0))\Big] e^{-|Y^\kk_0|^2/2t}\frac{dY^\kk_0}{(2\pi t)^{q/2}}.
 		\end{split}
 		\label{eq:7.2.5ss20jan}
 	\end{equation}
Set
\begin{equation}
	\begin{split}
		\mathrm{Tr_{s}}^{\Gamma,\prime}\big[\sigma^{Z}e^{-t\mathbf{D}^{Z,F,2}/2}\big]=\sum_{\underline{[\gamma]}_{\sigma}\in 
		\underline{E}_{\sigma}}\mathrm{Vol}(\Gamma\cap 
		Z_{\sigma}(\gamma)\backslash 
		X(\gamma\sigma))\cE_{X,\gamma\sigma}(F,t),
	\end{split}
	\label{eq:6.5.4disc}
\end{equation}
where the prime $(')$ refers to the number operator 
$N^{\Lambda^\bullet(\pp^*)}$ involved.

\begin{proposition}\label{prop:6.3.1est}
	If $\delta(Z^{0}_{\sigma}(\gamma))\neq 1$, then for $t>0$,
	\begin{equation}
		\cE_{X,\gamma\sigma}(F,t)=0.
		\label{eq:6.5.4last}
	\end{equation}
	For non-vanishing cases, there exists a 
	constant $C>0$ 
	such that for $t\in\; ]0,1]$
	\begin{equation}
		\begin{split}
			&\big|\sqrt{t}\cE_{X,\gamma\sigma}(F,t)\big|\leq C,\\
			&\big|(1+2t\frac{\partial}{\partial 
			t})\cE_{X,\gamma\sigma}(F,t)\big|\leq 
			C\sqrt{t}.
		\end{split}
		\label{eq:6.3.3pl}
	\end{equation}
	Then, as $t\rightarrow 0$, $\cE_{X,\gamma\sigma}(F,t)$ has the asymptotic expansion in the form of 
	\begin{equation}
		\frac{1}{\sqrt{t}}\sum_{j=0}^{+\infty} a^{\gamma\sigma}_j 
		t^j,\;\text{with }a^{\gamma\sigma}_j\in \bbC \text{ for } j\in\bN.
		\label{eq:6.3.4pl}
	\end{equation}

	There exist constants $C'>0$, $c'>0$ such that for $t\gg 0$, we have
	\begin{equation}
		\big|\cE_{X,\gamma\sigma}(F,t)\big|\leq C'e^{-c' t}.
		\label{eq:6.5.6last}
	\end{equation}
\end{proposition}
\begin{proof}
	Note that \eqref{eq:6.5.4last} was proved in \cite[Proposition 
	3.3.3 and Corollary 3.3.4]{LIU2021109117}, it follows from the 
	identities as in \eqref{eq:3.4.9paris}. The estimates \eqref{eq:6.3.3pl} were proved in \cite[(4.4.5) 
	in Theorem 4.4.1]{LIU2021109117}, and as a consequence, we get 
	the asymptotic expansion \eqref{eq:6.5.6last}.

	In the context of cyclic base change with $\gamma=1$, the 
	estimate \eqref{eq:6.5.6last} was proved in \cite[Lemma 
	4.10]{BeLip2017}. For general setting as here, it was proved in 
	\cite[(4.4.6) in Theorem 4.4.1]{LIU2021109117}. Note that for 
	this conclusion, the assumption on $\rho^{E}\circ\theta$ is 
	crucial (called a nondegeneracy condition for $\rho^{E}$).
\end{proof}
\begin{definition} We define the $\sigma$-twisted $L_{2}$-torsion for 
	$Z=\Gamma\backslash X$ associated with the flat vector bundle $F$ as follows,
	\begin{equation}
			\mathcal{T}_{\sigma,L_{2}}(Z,F)=-\frac{1}{2}\int_{0}^{+\infty}\big(1+2t\frac{\partial}{\partial t}\big)\mathrm{Tr_{s}}^{\Gamma,\prime}\Big[\sigma^{Z}e^{-t\mathbf{D}^{Z,F,2}/2}\Big]\frac{dt}{t}.
			\label{eq:6.5.9lastk}
	\end{equation}
\end{definition}
By Proposition \ref{prop:6.3.1est} and \eqref{eq:6.5.4disc}, 
$\mathcal{T}_{\sigma,L_{2}}(Z,F)$ is well-defined as a number. In 
particular, only the elliptic class $\underline{[\gamma]}_{\sigma}$ 
such that $\delta(Z^{0}_{\sigma}(\gamma))=1$ contributes to 
$\mathcal{T}_{\sigma,L_{2}}(Z,F)$. If there is no such elliptic class, we get 
$\mathcal{T}_{\sigma,L_{2}}(Z,F)=0$.
\begin{example}
	As in \cite{BeLip2017}, assume that $\sigma$ has finite 
	order, and that $H^{1}(\sigma,G)=1$ (which implies that for $\gamma\in 
	H^{1}(\sigma,\Gamma)$, it is $\sigma$-conjugate to $1$ by 
	elements in $G$). Recall that ${}^{\sigma}Z\subset Z$ is the 
	fixed point set of $\sigma$ in $Z$, then
	\begin{equation}
		\mathrm{Tr_{s}}^{\Gamma,\prime}\big[\sigma^{Z}e^{-t\mathbf{D}^{Z,F,2}/2}\big]=\mathrm{Vol}({}^{\sigma}Z)\mathrm{Tr_{s}}^{[\sigma]}\Big[\big(N^{\Lambda^\bullet(T^*X)}-\frac{m}{2}\big)\exp(-t\mathbf{D}^{X,F,2}/2)\Big].
	\end{equation}
	By \cite[Theorem 4.11]{BeLip2017}, 
	$\mathcal{T}_{\sigma,L_{2}}(Z,F)$ appeared as the limit of the 
	$\sigma$-equivariant analytic torsions under a tower of finite 
	coverings of $Z=\Gamma\backslash X$.
	 
	As explained in Subsection \ref{section:basechange}, if we write further the twisted orbital integral 
	$\mathrm{Tr_{s}}^{[\sigma]}[\cdots]$ in terms of the ordinary 
	identity orbital integral associated with the subgroup 
	$Z_{\sigma}(1)$, fixed point set of $\sigma$ in $G$. Then 
	$\mathcal{T}_{\sigma,L_{2}}(Z,F)$ is actually a linear 
	combination of the ordinary $L_{2}$-torsions 
	(\cite{MR1158345}, \cite{MATHAI1992369}) of ${}^{\sigma}Z$.
\end{example}

\begin{example}
	In \cite{MR2838248, BMZ2015toeplitz}, Bismut, Ma and Zhang showed 
	that for a universally constructed sequence of flat vector bundles $F_{d}$, $d\in 
	\bN$ over a closed manifold $Z$, under the nondegeneracy 
	condition, as $d\rightarrow +\infty$,
	\begin{equation}
		\mathcal{T}(Z,F_{d})=\mathcal{T}_{L_{2}}(Z,F_{d})+\mathcal{O}(e^{-cd}),
		\label{eq:6.5.11}
	\end{equation}
	where $\mathcal{T}(Z,F_{d})$, $\mathcal{T}_{L_{2}}(Z,F_{d})$ 
	denote the real analytic torsions, $L_{2}$-torsions respectively. 
	In the context of a locally symmetric space, a new proof of 
	\eqref{eq:6.5.11} using Selberg trace formula was given in 
	\cite{MR3128980}. 
	
	In \cite[Section 4]{LIU2021109117}, the author considered the 
	asymptotic expansion of $\sigma$-equivariant analytic torsions 
	$\mathcal{T}_{\sigma}(Z,F_{d})$ as $d\rightarrow +\infty$ for the compact locally symmetric 
	space $Z=\Gamma\backslash X$. We fix a nondegenerate unitary representation $(E,\rho^{E})\in \mathrm{Irr}(U^{\sigma})\cap 
	\mathrm{Irr}(U)$ which has the highest weight 
	$\lambda$. Associated with it, in \cite[Subsection 
	4.2]{LIU2021109117}, the author constructed a canonical sequence 
	$(E_{d},\rho^{E_{d}})\in \mathrm{Irr}(U^{\sigma})\cap 
	\mathrm{Irr}(U), d\in\bN^{*}$, such that $\rho^{E_{d}}$ has 
	highest weight $d\lambda$. By \cite[Proposition 
	4.6.1]{LIU2021109117}, as $d\rightarrow +\infty$,
	\begin{equation}
		\mathcal{T}_{\sigma}(Z,F_{d})=\mathcal{T}_{\sigma,L_{2}}(Z,F_{d})+\mathcal{O}(e^{-cd}),
		\label{eq:6.5.100}
	\end{equation}
	The main result of \cite[Section 
	4]{LIU2021109117} showed that the leading term (in $d$) of 
	$\mathcal{T}_{\sigma,L_{2}}(Z,F_{d})$ is given in terms of 
	$W$-invariants, for the fixed point set ${}^{\sigma}Z$, 
	introduced in \cite{MR2838248, BMZ2015toeplitz}.
\end{example}

\bibliographystyle{abbrv}

\bibliography{References}

\end{document}